\documentclass[11pt]{article}

\usepackage[margin=1in]{geometry}
\usepackage{amsmath,amssymb,amsthm,mathtools}
\usepackage{enumitem}
\usepackage{booktabs}
\usepackage{array}
\usepackage{tabularx}
\usepackage{xcolor}
\usepackage{etoolbox}
\usepackage{authblk}

\usepackage[colorlinks=true,linkcolor=blue!60!black,citecolor=blue!60!black,urlcolor=blue!60!black]{hyperref}

\newtheorem{theorem}{Theorem}[section]
\newtheorem{lemma}[theorem]{Lemma}
\newtheorem{proposition}[theorem]{Proposition}
\newtheorem{corollary}[theorem]{Corollary}

\theoremstyle{definition}
\newtheorem{definition}[theorem]{Definition}
\newtheorem{example}[theorem]{Example}
\newtheorem{remark}[theorem]{Remark}

\newcommand{\exremarkendmark}{%
  \renewcommand{\qedsymbol}{\ensuremath{\lozenge}}%
  \qed}
\AtEndEnvironment{example}{\exremarkendmark}
\AtEndEnvironment{remark}{\exremarkendmark}

\newcommand{\D}{\mathcal D}
\newcommand{\N}{\mathbb N}
\newcommand{\Z}{\mathbb Z}
\newcommand{\WRP}{\mathsf{WRP}}
\newcommand{\RR}{\mathsf{RR}}
\newcommand{\SWR}{\mathsf{sRR}_1}
\newcommand{\TDFT}{\mathsf{2DFT}}
\newcommand{\MSO}{\mathsf{MSO}}
\newcommand{\area}{\operatorname{area}}
\newcommand{\dinv}{\operatorname{dinv}}
\newcommand{\bounce}{\operatorname{bounce}}
\newcommand{\coarea}{\operatorname{coarea}}
\newcommand{\defc}{\operatorname{defc}}
\newcommand{\fas}{\operatorname{fas}}
\newcommand{\tailU}{\operatorname{tailU}}
\newcommand{\val}{\operatorname{val}}
\newcommand{\dr}{\operatorname{dr}}
\newcommand{\pk}{\operatorname{pk}}
\newcommand{\Nar}{\operatorname{Nar}}
\newcommand{\polyreg}{\mathsf{PolyReg}}
\newcommand{\id}{\operatorname{id}}
\newcommand{\rc}{\operatorname{rc}}
\newcommand{\rev}{\operatorname{rev}}
\newcommand{\comp}{\operatorname{comp}}

\title{A Computational Obstruction to Swapping Area and Dinv:\\
       An Automata-Theoretic View of the $q,t$-Catalan Symmetry}

\author[1]{Jineon Baek}
\author[2]{Byung-Hak Hwang}
\author[3]{Joonhyun La}
\author[4]{Hongseok Yang}
\affil[1]{June E Huh Center for Mathematical Challenges, Korea Institute for Advanced Study (KIAS), Korea}
\affil[2]{Center for AI and Natural Sciences, Korea Institute for Advanced Study (KIAS), Korea}
\affil[3]{School of Mathematics, Korea Institute for Advanced Study (KIAS), Korea}
\affil[4]{School of Computational Sciences, Korea Institute for Advanced Study (KIAS), Korea}
\affil[ ]{\newline\texttt{{jineon.kias}@gmail.com}, \texttt{bhwang@kias.re.kr}, \texttt{joonhyun@kias.re.kr}, \texttt{hongseokyang@kias.re.kr}}

\date{August 2026}

\begin{document}
\maketitle

\begin{abstract}
Algebraic combinatorics often seeks bijections that explain identities between
distributions object by object.  Once combinatorial objects are encoded as
words, automata theory lets us study such a bijection as a word-to-word
computation and measure its memory, access to input positions, and control of
output order.  This viewpoint refines the question of existence by asking which
computational mechanisms a bijection requires.  We develop it here for Dyck
paths.

Our motivating example is the $q,t$-Catalan polynomial.  Let $\D_n$ be the set
of Dyck paths of semilength $n$, let $\D=\bigcup_{n\ge 0}\D_n$, and let
$\area,\dinv,\bounce\colon\D\to\N$ be the standard Dyck-path statistics.  Then
\[
   C_n(q,t)=\sum_{P\in\D_n} q^{\area(P)}t^{\bounce(P)}
   =
   \sum_{P\in\D_n} q^{\dinv(P)}t^{\area(P)}.
\]
Haglund's zeta map $\zeta\colon\D\to\D$ gives a bijective proof of this
identity: it preserves semilength and satisfies
$\area(\zeta(P))=\dinv(P)$ and $\bounce(\zeta(P))=\area(P)$ for every
$P\in\D$.  By contrast, to our knowledge, the full symmetry
$C_n(q,t)=C_n(t,q)$ still lacks such a direct explanation: no explicit,
uniform, semilength-preserving bijection is known that swaps $\area$ and
$\dinv$ on every Dyck path.

Polyregular maps from automata theory provide a natural computational starting
point, but we prove that neither $\zeta$ nor the classical height-sweep
bijection witnessing the Narayana symmetry is polyregular.  The missing
mechanism is a global ordering by numerical levels whose range grows with the
input.  We call its general form a \emph{rank sort} and introduce
\emph{weighted-rank polyregular maps} ($\WRP$), which extend polyregular maps
by one such sort and contain both bijections.  The class is nevertheless a
proper subclass of deterministic logspace.  We prove that $\zeta^{-1}$ lies
outside $\WRP$ and that no $\WRP$ map can realise a semilength-preserving
area--dinv swap.  Thus, the rank-sorting strategy behind $\zeta$ cannot be
extended within $\WRP$ to a full exchange of the two statistics.
\end{abstract}

\section{Introduction}

\paragraph{Bijective proofs as computations.}

Many identities in algebraic combinatorics say that two numerical
measurements, called statistics, take each value equally often on a family of
objects.  A bijective proof explains such an identity one object at a time: it
constructs a bijection $F$ such that, for every object $P$, the value of one
statistic on $P$ equals the value of the other on $F(P)$.  We study the
computational content of such explanations.  Viewed as a computation, a
bijection $F$ takes an object $P$ as input and produces $F(P)$ as output.  What
resources does this computation require?  We use the answer to distinguish
bijective proofs by the computational mechanisms they employ.

Catalan combinatorics provides a natural setting in which the combinatorial and
computational viewpoints meet.  We focus on Dyck paths, classical Catalan
objects with a natural encoding as words.  In the step-word encoding, a Dyck
path of semilength $n$ is a word in $\{U,D\}^{2n}$ with $n$ occurrences of each
letter and no prefix containing more $D$'s than $U$'s.  Write $\D_n$
for the set of such paths and let $\D=\bigcup_{n\ge 0}\D_n$.  A Dyck-path
statistic is a map $X\colon\D\to\N$.  Examples include area, dinv, and bounce,
together with the numbers of peaks, valleys, and returns to height zero.  Two
statistics $X$ and $Y$ are equidistributed on every $\D_n$ if, for every $n$
and $k$, there are equally many paths in $\D_n$ with $X$-value $k$ as with
$Y$-value $k$.  The corresponding bijective problem is to find a single
semilength-preserving bijection $F\colon\D\to\D$ satisfying
$X(F(P))=Y(P)$ for every $P\in\D$.  This condition turns the equality of counts
into an object-by-object correspondence: for every $n$ and $k$, the map $F$
sends the paths in $\D_n$ of $Y$-value $k$ bijectively onto those of $X$-value
$k$.  Under the step-word encoding, $F$ is a word-to-word transformation on the
Dyck language.  We study such transformations through a hierarchy of
automata-theoretic models that differ in their memory, access to input
positions, and control of output order.

Our central test case is the $q,t$-Catalan polynomial $C_n(q,t)$.  In terms of
the standard Dyck-path statistics
$\area,\dinv,\bounce\colon\D\to\N$, it has two expressions:
\[
   C_n(q,t)=\sum_{P\in\D_n} q^{\area(P)}t^{\bounce(P)}
   =
   \sum_{P\in\D_n} q^{\dinv(P)}t^{\area(P)}.
\]
Haglund's zeta map $\zeta\colon\D\to\D$ gives an object-by-object explanation
of the equality between these two sums: it preserves semilength and satisfies
$\area(\zeta(P))=\dinv(P)$ and $\bounce(\zeta(P))=\area(P)$ for every
$P\in\D$.  By contrast, to our knowledge, the full symmetry
$C_n(q,t)=C_n(t,q)$ still has no direct bijective explanation.  Such an
explanation would be a single, explicit, semilength-preserving bijection
$F\colon\D\to\D$, uniform across all semilengths, satisfying
$\area(F(P))=\dinv(P)$ and $\dinv(F(P))=\area(P)$ for every $P\in\D$.
Finding such a bijection is a long-standing open problem~\cite{Pons2022}.

Since $\zeta$ sends $\dinv$ to $\area$, it already realises one half of the
desired exchange.  This raises a natural question: can the computational
mechanism behind $\zeta$ be extended to a bijection that swaps $\area$ and
$\dinv$?  The operation driving $\zeta$ is a variant of the sweep-map
construction.  A sweep map assigns an integer \emph{level} to each step and
outputs the steps in level order~\cite{ArmstrongLoehrWarrington2015,
ArmstrongLoehrWarrington2016}.  In the rational Dyck-path literature, the
corresponding numerical label of a step's starting point is also called its
\emph{rank}~\cite{Xin2015}.  The operation needed here is more general: its
prospective output items need not be individual steps, and their sorting keys
may be vectors of integers.  We call this operation a \emph{rank sort}.
Classical sweep maps are the scalar, one-item-per-step case.  We treat rank
sorting as a computational resource within standard models of word-to-word
computation.

\paragraph{Computational models and main results.}

Automata theory provides a standard hierarchy of increasingly expressive
models for transforming words.  A one-way finite-state transducer processes
the input from left to right.  A two-way transducer may revisit input
positions, while a polyregular map may systematically combine a fixed number
of positions at a time, much like a program with a fixed number of nested
\texttt{for}-loops.  Moving up the hierarchy allows a map to revisit and
combine input positions more freely, produce multiple output pieces, and exert
greater control over their order.  In every model, however, the transformation
is specified by a fixed
finite-state or logical rule that is independent of the input length.  We
prove that $\zeta$ is not polyregular and therefore cannot be realised by any
model in this standard hierarchy.

To extend this hierarchy with rank sorting in a controlled way, we introduce
\emph{weighted-rank polyregular maps} ($\WRP$).  The $\WRP$ model retains
the polyregular machinery for selecting atoms, assigning a letter to each
selected atom, and specifying their order.  Each selected atom contributes one
letter to the final output word.  The model then adds one global rank
sort: each atom receives a fixed-dimensional integer vector, called its
\emph{atom rank}, computed from
finite-state scans that accumulate local weights.  These atom ranks may grow
without bound as the input length grows.  The output word is formed by listing
the atoms in lexicographic atom-rank order, with the underlying $\MSO$-defined
order breaking ties, and reading their letters.  Assigning every atom the same
atom rank recovers the underlying
polyregular map, so every polyregular map is a $\WRP$ map.  The theorem below
places $\zeta$ in $\WRP$; combined with the non-polyregularity stated above,
this shows that the inclusion is strict.  In combinatorial terms, the model
captures a general \emph{decorate, rank, and sort} pattern: for $\zeta$, selected
atoms derived from the input path receive height-based atom ranks, while other
running step weights give additive level sorts.  Thus, $\WRP$ is both a
principled extension of an established computational model and a way to
isolate a recurring source of global reordering in Catalan and lattice-path
combinatorics.  The results below delineate what the model can express and
where its limitations begin.

On the computational side, we prove the strict inclusions
\[
   \polyreg\ \subsetneq\ \WRP
   \ \subsetneq\ \text{deterministic logspace}.
\]
The upper bound has a concrete memory interpretation.  On an input of length
$n$, naming an input position, or storing a counter of polynomial magnitude,
costs $\Theta(\log n)$ bits.  A logarithmic-space evaluator can therefore
retain only a fixed number of such indices and counters at once.  A $\WRP$ map
can nevertheless carry out its global rank sort within this constraint by
recomputing atoms from the input instead of storing the polynomial-size
collection of potential atoms.  In this sense, $\WRP$ adds a genuinely global
ordering mechanism while remaining close to the finite-memory viewpoint of
automata.

Our first result establishes both that $\WRP$ captures its motivating example,
the zeta map, and that the rank-sort operation genuinely increases expressive
power.
\begin{theorem}[Zeta in $\WRP$, not $\polyreg$]
\label{thm:intro-zeta-classification}
There exists a $\WRP$ map $Z$ on words over $\{U,D\}$ such that
$Z(P)=\zeta(P)$ for every $P\in\D$.
No polyregular map agrees with $\zeta$ on every Dyck path.
\end{theorem}

A second classification result concerns a classical height-sweep bijection.
This bijection is known to exchange two natural local statistics on Dyck
paths, the numbers of valleys and double rises.  We show both that the map is
realised by the particularly simple fragment $\SWR$ of $\WRP$ and that no
polyregular map realises it on all Dyck paths.
\begin{theorem}[A classical Narayana bijection in $\WRP$]
\label{thm:intro-narayana-sweep}
Let $\val(P)$ and $\dr(P)$ denote, respectively, the number of valleys
($DU$ factors) and double rises ($UU$ factors) in a Dyck path $P$.
There exists a map $H\in\SWR\subseteq\WRP$, called the \emph{height-sweep
map}, such that, for every $n$, the restriction of $H$ to $\D_n$ is a
bijection from $\D_n$ onto $\D_n$ satisfying
$\dr(H(P))=\val(P)$ and $\val(H(P))=\dr(P)$ for every
$P\in\D_n$.
Consequently, $H$ realises the bivariate Narayana
symmetry
\[
   \sum_{P\in\D_n} q^{\val(P)}t^{\dr(P)}
   =
   \sum_{P\in\D_n} q^{\dr(P)}t^{\val(P)}.
\]
Furthermore, no polyregular map agrees with $H$ on every Dyck path.
\end{theorem}

Thus, both $\zeta$ and $H$ require the rank-sort mechanism beyond
polyregularity.  These two independent examples show that $\WRP$ captures a
recurring Catalan mechanism rather than a single isolated construction.  We
next ask whether its rank-sorting mechanism is powerful enough to produce a
uniform bijection explaining the full $q,t$-Catalan symmetry.  Our central
theorem gives a negative answer.
\begin{theorem}[No $\WRP$ area--dinv swap]
\label{thm:intro-no-swap}
There is no $\WRP$ map $T$ on words over $\{U,D\}$ that sends every Dyck path
$P$ to a Dyck path of the same semilength and satisfies
$\area(T(P))=\dinv(P)$ and $\dinv(T(P))=\area(P)$ for every $P\in\D$.
\end{theorem}

The theorem is stronger than the nonexistence of a $\WRP$ bijection: it rules
out every semilength-preserving $\WRP$ map satisfying both identities, whether
or not the map is bijective.  It does not, however, assert that no
area--dinv swapping bijection exists.  The zeta map realises one of the two
required identities, since $\area(\zeta(P))=\dinv(P)$, but no $\WRP$ map can
realise both simultaneously.  Thus, any uniform bijection giving a direct
bijective proof of $C_n(q,t)=C_n(t,q)$ must use computational resources beyond
the single rank-sort layer available in $\WRP$.

A second limitation concerns inversion.  Although $\zeta$ belongs to $\WRP$,
its inverse cannot be realised within the same model:
\begin{theorem}[Inverse zeta outside $\WRP$]
\label{thm:intro-inverse-zeta}
The inverse zeta bijection $\zeta^{-1}\colon\D\to\D$ cannot be realised by any
$\WRP$ map.
\end{theorem}

The no-swap and inverse-zeta theorems are the paper's two principal lower
bounds on what $\WRP$ itself can realise.  The first rules out every $\WRP$
realisation of an area--dinv swap.  The second shows that, under the
realisation convention, $\WRP$-realisability is not preserved by inversion:
$\zeta$ belongs to $\WRP$, whereas $\zeta^{-1}$, for which Thomas and
Williams~\cite{ThomasWilliams2018} gave an explicit construction, does not.

Both lower bounds exploit the same restriction on $\WRP$ maps.  We evaluate a
hypothetical map on specially chosen Dyck paths and associate with each output
a tuple of simple integer-valued features, such as the numbers of
$U$-steps before and after its first $D$.  For a $\WRP$ map realising either
target transformation, the set of such feature tuples must be a finite union
of patterns obtained from fixed base tuples by repeatedly adding fixed integer
steps; sets of this kind are called \emph{semilinear}.  An area--dinv swap
would instead force its feature tuples to fill a region with a quadratic
boundary, while $\zeta^{-1}$ would force its feature tuples to follow a
pattern governed by integer division.  Neither set is semilinear.  Together
with the $\WRP$ realisations of zeta and the classical Narayana bijection,
these lower bounds delineate the scope of the model: one rank-sort layer
suffices to realise several natural Catalan constructions, but it cannot
realise a full area--dinv swap or compute $\zeta^{-1}$.

\paragraph{Proof idea for the no-swap theorem.}

The proof compares two descriptions of the same set.  Suppose that a $\WRP$
map $T$ swaps area and dinv while preserving semilength, and apply $T$ to
\[
   W_n=U(UD)^nD,\qquad n\ge 1,
\]
and write $Q_n=T(W_n)$.  From each output we extract its \emph{first-ascent
pair} $(\fas(Q_n),\tailU(Q_n))$, where the two coordinates count the $U$-steps
before and after the first $D$, respectively.  Let
$S_T=\{(\fas(Q_n),\tailU(Q_n)):n\ge 1\}$.
A general property of $\WRP$ maps implies that $S_T$ must be semilinear.
We define this term below; at this stage, only the logical contrast matters.
The two statistic identities required of $T$ force the same set $S_T$ not to
be semilinear.  Thus, a hypothetical map $T$ would make $S_T$ both semilinear
and nonsemilinear, which is impossible.

We now explain the two sides.  For each $n\ge1$, the path $W_n$ has semilength
$n+1$, area $n$, and dinv $\binom n2$.
Because $T$ preserves semilength and swaps the two statistics, its output
$Q_n$ must satisfy
\[
   Q_n\in\D_{n+1},\qquad
   \area(Q_n)=\binom n2,\qquad
   \dinv(Q_n)=n.
\]
The maximum possible area in $\D_{n+1}$ is $\binom{n+1}{2}$, so
$\binom{n+1}{2}-\area(Q_n)=n=\dinv(Q_n)$.
Thus, equality holds in the general inequality
\[
   \dinv(Q)\le \binom N2-\area(Q),
   \qquad Q\in\D_N.
\]
The equality cases are rigid: for each $n$, the required statistics determine
$Q_n$ uniquely as a near-staircase with a short two-level tail.  As $n$
varies, the first-ascent pairs of these forced outputs range over exactly
\[
   S_{\mathrm{tri}}
   =
   \left\{(a,b)\in\N^2
      \;:\;
      b\ge 1,\quad
      \binom b2+1\le a\le\binom{b+1}{2}+1
   \right\}.
\]
Consequently, the swap identities force $S_T=S_{\mathrm{tri}}$.

A set of integer tuples is \emph{semilinear} if it is a finite union of
sets obtained from fixed base tuples by adding arbitrary nonnegative integer
combinations of finitely many fixed step vectors.  The semilinearity theorem proved in
Section~\ref{sec:slice-semilinearity} shows that the set $S_T$ associated with
a $\WRP$ map must be semilinear.
However, $S_{\mathrm{tri}}$ is not semilinear: for each fixed $b$, its smallest
first coordinate is $\binom b2+1$, which grows quadratically with $b$.  By
contrast, the lower boundary of a
semilinear set with finite sections must eventually be affine on each residue
class.  Therefore $S_T$ cannot equal $S_{\mathrm{tri}}$.

Section~\ref{sec:slice-semilinearity} proves the semilinearity conclusion for
$\WRP$ maps, while Section~\ref{sec:wrapped-flat} derives the forced set
$S_{\mathrm{tri}}$ and proves that it is not semilinear.  Combining the two
conclusions proves the no-swap theorem.

\paragraph{Relation to previous work.}

Our approach brings together the theory of Catalan sweep maps and
automata-theoretic models of word transformations.  The rank-sorting
constructions motivating $\WRP$ come from the sweep-map
literature~\cite{ArmstrongLoehrWarrington2015,
ArmstrongLoehrWarrington2016,ThomasWilliams2018}, while its logical component
extends the standard hierarchy of regular and polyregular
transductions~\cite{EngelfrietHoogeboom2001,Bojanczyk2018}.  Related
complexity-theoretic programmes ask when counting functions have combinatorial
interpretations~\cite{Pak2024WhatIs,IkenmeyerPak2022}; here we ask instead for
the resources needed to compute an explicit bijection between already known
objects.  Section~\ref{sec:related-work} returns to these comparisons and
formulates the resulting open problems after the models and results needed to
make the relationships precise have been developed.

\paragraph{Outline.}

Section~\ref{sec:dyck} reviews Dyck paths, their statistics, and the named
bijections considered in the paper.  Section~\ref{sec:models} reviews the
established automata-theoretic models of word-to-word computation used
throughout the paper.  Section~\ref{sec:wrp} introduces $\WRP$, develops its
structural and computational properties, and places the principal Catalan
maps in the resulting hierarchy.  The longer structural proofs are deferred
to Appendix~\ref{app:wrp-structural-proofs}.  Section~\ref{sec:zeta} then
determines the position of $\zeta$ in the hierarchy.
Section~\ref{sec:narayana-sweep} places the classical height-sweep bijection
$H$, which realises the bivariate Narayana symmetry, in
$\SWR\setminus\polyreg$ and hence inside $\WRP$.  A self-contained proof of
its known combinatorial properties
is given in Appendix~\ref{app:narayana-sweep-proof}.
Sections~\ref{sec:slice-semilinearity} and~\ref{sec:wrapped-flat} develop the
semilinearity method and use it to prove the no-swap theorem.
Section~\ref{sec:inverse-zeta} studies $\zeta^{-1}$ using a related argument.
Section~\ref{sec:lean} describes the scope and status of the accompanying
machine-checked Lean formalisation.  Finally,
Section~\ref{sec:discussion} places the results in the related
algebraic-combinatorial and automata-theoretic literatures and presents the
open problems that they leave.

\section{Dyck paths, statistics, and named bijections}\label{sec:dyck}

We now fix the Dyck-path conventions and the step-word encoding used
throughout the paper.  We define area sequences and the statistics
$\area$, $\dinv$, and $\coarea$; describe Haglund's zeta map in a form
that scans the numerical levels of its area sequence; and list the named
bijections considered later.  The underlying Catalan notions are standard,
but we recall them to make the paper self-contained and to fix notation.

Throughout the paper, $|w|$ denotes the length of a word $w$, while $|S|$
denotes the cardinality of a finite set $S$.

\begin{definition}[Dyck paths and area sequences]
A \emph{word} over an alphabet is a finite sequence of symbols from that
alphabet.  A \emph{Dyck path} of semilength $n$ is a word
$P\in\{U,D\}^{2n}$ such that every prefix has at least as many $U$'s as $D$'s
and the whole word has exactly $n$ $U$'s and $n$ $D$'s.  Write $\D_n$ for the
set of such paths and $\D=\bigcup_{n\ge 0}\D_n$.  For a word
$w\in\{U,D\}^*$ and $0\le k\le |w|$, define the height after its first $k$
letters by
\[
   h_w(k)=|\{1\le i\le k:w_i=U\}|-|\{1\le i\le k:w_i=D\}|.
\]
If $P\in\D_n$ and $p_1<\cdots<p_n$ are the positions of its $U$-steps, the
\emph{area sequence} of $P$ is
\[
   a(P)=(a_1,\ldots,a_n),\qquad a_i=h_P(p_i-1).
\]
\end{definition}

Thus, $a_i$ is the height immediately before the $i$th up-step.  Under the
usual identification of $U$ with a north step and $D$ with an east step,
this agrees with the standard area sequence of a Dyck path.

\begin{remark}[The step-word encoding]\label{rmk:step-word}
The preceding definition represents a Dyck path of semilength $n$ by its
length-$2n$ word over $\{U,D\}$, with one letter per step, read from left to
right.  We call this representation the \emph{step-word encoding}.  Every
classification of a Dyck-path map in this paper is relative to this encoding,
for both its input and its output.

This qualification matters because the transformation representing a fixed
abstract bijection may belong to different complexity classes under different
encodings.  For example, an area-sequence encoding makes the height before
each up-step explicit, whereas in the step-word encoding this information must
be recovered from prefix sums.
\end{remark}

\begin{definition}[Area, dinv, coarea]
For $P\in\D_n$ with area sequence $a=(a_1,\ldots,a_n)$, define
\[
   \area(P)=\sum_{i=1}^n a_i,\qquad
   \dinv(P)=|\{(i,j):1\le i<j\le n,\ a_i-a_j\in\{0,1\}\}|.
\]
The maximum possible area in $\D_n$ is $\binom n2$; write
$\coarea(P)=\binom n2-\area(P)$.  Following Lee, Li, and
Loehr~\cite[Definition~1.4]{LeeLiLoehr2018}, the \emph{deficit} of $P$ is
$\defc(P)=\coarea(P)-\dinv(P)$.
A path with $\defc(P)=0$ is called \emph{deficit-zero}.
\end{definition}

\begin{example}
The $U$-steps of $P=UUDDUD\in\D_3$ occur at positions $1,2,5$ and begin at
heights $0,1,0$.  Thus, $a(P)=(0,1,0)$, $\area(P)=1$, and
$\coarea(P)=\binom32-1=2$.
Of the three pairs $(i,j)$ with $1\le i<j\le3$, exactly $(1,3)$ and
$(2,3)$ contribute to $\dinv(P)$: they give
$a_1-a_3=0$ and $a_2-a_3=1$, respectively.  Hence $\dinv(P)=2$.
\end{example}

\begin{definition}[The Haglund zeta map]\label{def:zeta}
Set $\zeta(\varepsilon)=\varepsilon$ for the empty path.  For
$P\in\D_n$ with $n\ge 1$ and area sequence $a=(a_1,\ldots,a_n)$, the
\emph{zeta map} $\zeta(P)$ is defined by the following area-sequence scan: for
each level $r=0,1,\ldots,\max_i a_i+1$, scan $a_1,\ldots,a_n$ from left to
right and, at each position $i$, append
\[
   U\text{ if }a_i=r,\qquad D\text{ if }a_i=r-1,\qquad
   \text{nothing otherwise.}
\]
\end{definition}

The algebraic $q,t$-Catalan sequence was introduced by Garsia and
Haiman~\cite{GarsiaHaiman1996}.  A precursor of the inverse zeta map appeared
in work of Andrews, Krattenthaler, Orsina, and Papi~\cite{AndrewsEtAl2002};
its relation to the zeta and sweep maps is explained by Armstrong, Loehr, and
Warrington~\cite{ArmstrongLoehrWarrington2015}.  The Dyck-path terminology and
formulation used here follow Haglund~\cite{Haglund2008}.  The classical zeta
theorem states that
$\zeta(P)$ is a Dyck path of semilength $n$, that
$\zeta$ restricts to a bijection $\D_n\to\D_n$, and that
\[
   \area(\zeta(P))=\dinv(P),
   \qquad
   \bounce(\zeta(P))=\area(P).
\]
Here $\bounce$ denotes the standard bounce statistic from $q,t$-Catalan
theory.  Its definition plays no role in our arguments; we use only the
identity $\bounce(\zeta(P))=\area(P)$.  Readers unfamiliar with bounce may
therefore regard it simply as the output statistic appearing in the classical
zeta theorem.

\begin{example}
For $P=UUDDUD$, the area sequence is $(0,1,0)$.  At level $0$, entries $1$
and $3$ of the area sequence are equal to $0$, so they produce $UU$.  At
level $1$, the three entries contribute $D,U,D$ from left to right, producing
$DUD$.  At level $2$, only entry $2$ contributes, producing $D$.  Concatenating
the three blocks gives
$\zeta(UUDDUD)=UUDUDD$.  Indeed,
$\area(UUDUDD)=0+1+1=2=\dinv(UUDDUD)$.
\end{example}

The principal Dyck-path bijections studied in this paper are summarised in
Table~\ref{tab:named-bijections}.  Some are defined later, when they first
enter the arguments.  Section~\ref{sec:landscape} places all of them in the
model hierarchy for word-to-word computation.

\begin{table}[t]
\centering
\small
\renewcommand{\arraystretch}{1.25}
\begin{tabularx}{\textwidth}{@{}>{\raggedright\arraybackslash}p{0.17\textwidth}
                                   >{\raggedright\arraybackslash}X
                                   >{\raggedright\arraybackslash}X@{}}
\toprule
Name & Definition & Notes \\
\midrule
Identity $\id$ & $P\mapsto P$ & Preserves every statistic. \\
Reverse-complement $\rc$ &
  $w\mapsto \comp(\rev(w))$, where $\rev$ reverses the word and
  $\comp$ swaps $U\leftrightarrow D$. &
  An involution preserving each $\D_n$. \\
Zeta $\zeta$ & Definition~\ref{def:zeta}. &
  Sends $\dinv$ to $\area$. \\
Height sweep $H$ &
  Orders steps by increasing starting height, with right-to-left ties
  (Section~\ref{sec:narayana-sweep}). &
  A classical bijection that swaps valleys ($DU$ factors) and double rises
  ($UU$ factors). \\
Zeta inverse $\zeta^{-1}$ &
  The inverse of Definition~\ref{def:zeta}
  (Section~\ref{sec:inverse-zeta}). &
  Satisfies $\dinv(\zeta^{-1}(P))=\area(P)$. \\
\bottomrule
\end{tabularx}
\caption{Principal Dyck-path bijections studied in this paper.}
\label{tab:named-bijections}
\end{table}

\section{Established models of word-to-word computation}\label{sec:models}

We now review the established models for word-to-word computation used in the
rest of the paper.  We first explain what it means for a machine to realise a
map specified only on Dyck paths, then introduce $\MSO$ on words and its
connection with regular languages.  From there we move through deterministic
two-way finite-state transducers to polyregular maps.  In automata theory, a
machine that reads an input word and produces an output word is called a
\emph{transducer}.  Section~\ref{sec:wrp} will extend these standard models by
the numerical ordering operation needed for zeta-like maps.

\subsection{Realisation}

Our Catalan maps are specified only on Dyck paths, whereas the machine models
introduced below have arbitrary words as their ambient inputs.  \emph{Realisation}
is the convention that reconciles these two domains.  To state it, recall that
an \emph{alphabet} is a finite set of symbols and that $\Sigma^*$ denotes the set
of all finite words over $\Sigma$.  Thus, $\{U,D\}^*$ contains every finite step
word, with the Dyck paths forming a proper subset.

\begin{definition}[Realisation]\label{def:relative}
Fix finite alphabets $\Sigma$ (inputs) and $\Gamma$ (outputs), and let
$X\subseteq\Sigma^*$ and $f\colon X\to\Gamma^*$.  A partial map
$T\colon\Sigma^*\rightharpoonup\Gamma^*$, such as one computed by the machine
models defined below, \emph{realises $f$} if
$X\subseteq\mathrm{dom}(T)$ and $T(x)=f(x)$ for every $x\in X$.  Its behaviour
on inputs outside $X$, including whether it is defined there at all, is left
unconstrained.
\end{definition}

The point of this convention is to separate two different tasks.  Given an
arbitrary word in $\{U,D\}^*$, its height after a prefix is the number of $U$'s
in that prefix minus the number of $D$'s.  Deciding whether the word is a Dyck
path requires checking that every prefix has nonnegative height and that the
final height is zero.  Our aim, however, is to measure the resources
needed to compute a Catalan map when the input is a Dyck path, not the resources
needed to decide whether an arbitrary input is a Dyck path.  Realisation
therefore requires agreement with the intended map on every Dyck input but
imposes no condition on inputs that are not Dyck paths.

\subsection{Monadic second-order logic on words}\label{sec:mso}

Monadic second-order logic ($\MSO$) provides the common logical language for
the models of word transformations introduced below.  On an input word
$w=a_1a_2\cdots a_n$, it can state properties of the numbered positions
$1,\ldots,n$: for example, a formula can select positions carrying $U$, assign
output letters to selected positions, or specify their order in the output.
Crucially, the same finite collection of formulas works uniformly for words of
every length.  We first describe the structure on which these formulas are
interpreted, and then recall the connection between $\MSO$ sentences and
finite automata.

\paragraph{Words as labelled positions.}
Fix a finite alphabet $\Sigma$ (for us, $\Sigma=\{U,D\}$).  We represent a word
$w=a_1a_2\cdots a_n$ by its set of \emph{positions}
$\mathrm{Pos}(w)=\{1,\ldots,n\}$.  This set carries the left-to-right order
$<$ and, for each letter $a\in\Sigma$, a predicate $P_a$ marking the
positions labelled by $a$.  Thus, $P_a(i)$ holds exactly when $a_i=a$; on a
Dyck word, for example, $P_U(i)$ says that position $i$ is an up-step.  This
ordered, letter-labelled set is the word structure seen by $\MSO$ formulas.

\paragraph{Formulas.}
An $\MSO$ formula is interpreted in the structure associated with a fixed
input word $w$.  Position variables $x,y,z,\ldots$ range over
$\mathrm{Pos}(w)$, while set variables $X,Y,Z,\ldots$ range over its subsets.
Thus, position variables do not range over all integers, nor do set variables
range over arbitrary sets of integers: the input word supplies the entire
universe of the formula.  The adjective ``monadic'' refers to this
quantification over sets of individual positions.  The atomic formulas are
$x<y$ (``$x$ lies left of $y$''), $x=y$, $P_a(x)$ (``position $x$ carries
letter $a$''), and
$x\in X$ (``position $x$ belongs to the set $X$'').  Larger formulas are built
from these with the Boolean connectives $\wedge,\vee,\neg,\rightarrow$ and the
quantifiers $\exists x,\forall x$ over positions and $\exists X,\forall X$ over
sets of positions.

\paragraph{Satisfaction.}
A formula with no free variables is a \emph{sentence}.  We write
$w\models\varphi$ when the word structure associated with $w$ satisfies
$\varphi$.  For example, $w\models\exists x\,P_U(x)$ means that $w$ contains at
least one $U$.  If
$\varphi(x_1,\ldots,x_k)$ has free position variables and
$i_1,\ldots,i_k\in\mathrm{Pos}(w)$, then
$w\models\varphi(i_1,\ldots,i_k)$ means that $\varphi$ is true when $x_j$ is
assigned the concrete position $i_j$ for every $j$.  Thus, one formula
uniformly selects the
set of position tuples
\[
   \{(i_1,\ldots,i_k)\in\mathrm{Pos}(w)^k : w\models\varphi(i_1,\ldots,i_k)\}
\]
on each input word $w$.  The transduction definitions below use such selected
tuples to construct output words.

\begin{example}[Peaks and parity in $\MSO$]
The formula $\mathrm{succ}(x,y):=x<y\wedge\neg\exists z\,(x<z\wedge z<y)$ says
``$y$ is the position immediately after $x$''.  Using it,
\[
   \varphi(x)\ :=\ P_U(x)\ \wedge\ \exists y\,\bigl(\mathrm{succ}(x,y)\wedge P_D(y)\bigr)
\]
defines on each word the set of positions $x$ that start a $UD$ factor (a
\emph{peak}), and the sentence $\exists x\,\varphi(x)$ says that the word has a
peak.  Quantification over sets can also impose a global marking across the
word.  For instance, $\MSO$ can say that the number of $U$-steps is even.  Let
\[
\begin{aligned}
   \mathrm{first}_U(x)&:=P_U(x)\wedge\neg\exists z\,(z<x\wedge P_U(z)),\\
   \mathrm{last}_U(x)&:=P_U(x)\wedge\neg\exists z\,(x<z\wedge P_U(z)),\\
   \mathrm{next}_U(x,y)&:=P_U(x)\wedge P_U(y)\wedge x<y
      \wedge\neg\exists z\,(x<z\wedge z<y\wedge P_U(z)),\\
   \mathrm{opp}_X(x,y)&:=
   \bigl(x\in X\wedge y\notin X\bigr)
   \vee
   \bigl(x\notin X\wedge y\in X\bigr).
\end{aligned}
\]
Then, the sentence
\[
\exists X\,\Bigl[
   \forall x\,(\mathrm{first}_U(x)\rightarrow x\in X)
   \wedge
   \forall x\,(\mathrm{last}_U(x)\rightarrow x\notin X)
   \wedge
   \forall x\,\forall y\,(\mathrm{next}_U(x,y)\rightarrow
      \mathrm{opp}_X(x,y))
\Bigr]
\]
says that the $U$-positions can be marked alternately by $X$, beginning with a
marked position and ending with an unmarked one.  The sentence therefore holds
exactly when the number of $U$-steps is even; when there are no $U$-steps, all
three conditions are vacuous, so the sentence holds as required.
\end{example}

\paragraph{Regular languages.}
Having seen how formulas with free variables select positions, we now turn to
sentences, which will be used to restrict the domains of transductions.  A
\emph{language} is a set of words $L\subseteq\Sigma^*$.  Every $\MSO$ sentence
$\varphi$ defines the language $\{w:w\models\varphi\}$.  Its machine-theoretic
counterpart is the deterministic finite automaton, the standard model for
recognising languages with a fixed amount of memory.

\begin{definition}[Deterministic finite automaton]\label{def:dfa}
A \emph{deterministic finite automaton} (DFA) over a finite alphabet $\Sigma$ is
a tuple $M=(Q,q_0,\delta,F)$ consisting of a finite set $Q$ of \emph{states}, an
\emph{initial state} $q_0\in Q$, a \emph{transition function}
$\delta\colon Q\times\Sigma\to Q$, and a set $F\subseteq Q$ of \emph{accepting
states}.
\end{definition}

To describe how a DFA processes a whole word, write $\varepsilon$ for the
empty word and $wa$ for the result of appending $a$ to $w$.  The transition
function extends uniquely to $\delta^*\colon Q\times\Sigma^*\to Q$ satisfying
\[
   \delta^*(q,\varepsilon)=q,
   \qquad
   \delta^*(q,wa)=\delta(\delta^*(q,w),a).
\]
For $q\in Q$ and $w\in\Sigma^*$, the value $\delta^*(q,w)$ is the state reached
after reading $w$ from state $q$.

Equivalently, the \emph{run} of $M$ on $w=a_1\cdots a_n$ is the state sequence
$q_0,q_1,\ldots,q_n$ with $q_i=\delta(q_{i-1},a_i)$.  Its final state is
$\delta^*(q_0,w)$.  The automaton \emph{accepts} $w$ when this state belongs to
$F$, and it \emph{recognises} the language
$L(M)=\{w\in\Sigma^*: \delta^*(q_0,w)\in F\}$.

\begin{definition}[Regular language]\label{def:regular-language}
A language $L\subseteq\Sigma^*$ is \emph{regular} if $L=L(M)$ for some DFA
$M$ over $\Sigma$.
\end{definition}

Equivalently, a regular language can be decided by scanning the input once
from left to right with only a fixed amount of memory, and in particular
without a counter that can grow with the input.  The classical
\emph{B\"uchi--Elgot--Trakhtenbrot
theorem}~\cite{Thomas1997} connects this machine model to the logic above: a
language is regular if and only if it equals
$\{w:w\models\varphi\}$ for some $\MSO$ sentence $\varphi$.  Thus
``$\MSO$-definable'' and ``regular'' describe the same languages.
Accordingly, throughout the paper a \emph{regular condition} on words means a
condition whose satisfying words form a regular language; equivalently, it can
be expressed by an $\MSO$ sentence.

The Dyck language is a standard nonregular example: recognising it requires
tracking an unbounded prefix height.  This is why the realisation convention
of Definition~\ref{def:relative} separates recognition of Dyck inputs from
computation of a Catalan map on those inputs.  The $\MSO$--DFA equivalence will
let us move between logical domain conditions and finite-state language
arguments throughout the paper.

\subsection[Two-way finite-state transducers]{Two-way finite-state transducers and their $\MSO$ presentation}
\label{sec:2dft-mso}

Having used finite automata to recognise languages, we now turn to machines
that compute word-to-word transformations.  A deterministic two-way
finite-state transducer (2DFT) has a finite set of control states and an input
head that can move one position to the left or right.  Each transition may
emit a word, so the machine can revisit input positions and produce output in
an order different from their left-to-right order.  It has no counter or work
tape whose size grows with the input: apart from the position of its head, its
only memory is its current state.  This is the first model of word-to-word
transformation in our hierarchy.

There is an equivalent declarative description using the $\MSO$ introduced in
Section~\ref{sec:mso}.  Rather than tracing the movements of a head, an
\emph{$\MSO$ string transduction} begins with finitely many copy names, each
attached to the input positions.  Formulas select some of the resulting
potential atoms, assign a letter to each selected atom, and linearly order the
selected atoms;
reading their labels in that order produces the output word.  The theorem of
Engelfriet and Hoogeboom (Theorem~\ref{thm:eh}) says that these operational and
logical descriptions define exactly the same partial word-to-word maps.  We
retain both viewpoints: the machine description provides operational
intuition, while the logical description is extended in the definitions of
polyregular and $\WRP$ maps.

\begin{definition}[Deterministic two-way finite-state transducer]
\label{def:2dft}
A \emph{deterministic two-way finite-state transducer} (2DFT)
$T=(Q,\Sigma,\Gamma,q_0,F,\eta)$ has a finite state set $Q$, an input alphabet
$\Sigma$, an output alphabet $\Gamma$, an initial state $q_0\in Q$, accepting
states $F\subseteq Q$, and a partial transition-output function
\[
   \eta\colon Q\times(\Sigma\cup\{\vdash,\dashv\})
     \rightharpoonup Q\times\{-1,+1\}\times\Gamma^*,
\]
where ${\vdash},{\dashv}\notin\Sigma$ are the left and right end markers.  We require
the end-marker discipline that a transition from $\vdash$, if defined, moves
right, and a transition from $\dashv$, if defined, moves left.
\end{definition}

The equation $\eta(q,a)=(q',d,u)$ describes one step of the machine.  When the
current state is $q$ and the head scans $a$, the machine changes its state to
$q'$, moves one position to the left if $d=-1$ and to the right if $d=+1$, and
emits $u$.  Thus, the three components of the transition value specify the new
state, the head movement, and the emitted word, respectively.

\begin{definition}[Configuration, run, and output]
Fix a 2DFT $T=(Q,\Sigma,\Gamma,q_0,F,\eta)$.  For an input
$w=a_1\cdots a_n$, put $a_0=\vdash$ and $a_{n+1}=\dashv$.  A
\emph{configuration of $T$ on $w$} is a pair
$(q,i)\in Q\times\{0,\ldots,n+1\}$.  The \emph{run of $T$ on $w$} is the unique
maximal sequence of configurations
\[
   (q_0,i_0),(q_1,i_1),(q_2,i_2),\ldots
\]
with $i_0=0$ such that, whenever $\eta(q_t,a_{i_t})=(q_{t+1},d_t,u_t)$,
the next head position is $i_{t+1}=i_t+d_t$ and the transition emits
$u_t\in\Gamma^*$.  The run is \emph{finite} if it ends in a configuration
$(q_m,i_m)$ for which $\eta(q_m,a_{i_m})$ is undefined; otherwise it is
\emph{infinite}.  A finite run is \emph{accepting} if $q_m\in F$, and in that case,
the \emph{output of $T$ on $w$} is the concatenation of the emitted words:
\[
   T(w)=u_0u_1\cdots u_{m-1}.
\]
The empty concatenation, when $m=0$, is $\varepsilon$.
If the run is infinite or finite but nonaccepting, then $T(w)$ is undefined.
\end{definition}

Determinism and partiality play different roles here.  Because $\eta$ is a
function, every configuration has at most one successor, so the maximal run is
unique.  Because $\eta$ need not be defined everywhere, the run may halt; its
halting state need not be accepting, and the run may instead be infinite.
Thus, a 2DFT generally computes a partial word-to-word map.

We now describe the same class of maps without referring to runs.

\begin{definition}[$\MSO$ string transduction]\label{def:mso-transduction}
Fix an integer $K\ge 1$, called the \emph{copy count}, and let
$C=\{1,\ldots,K\}$ be the finite set of \emph{copy names}.  For an input word
$w$, each \emph{potential atom} is a pair $(c,i)$ consisting of a copy name
$c\in C$ and a concrete input position
$i\in\mathrm{Pos}(w)$.  Thus, $w$ has $K\,|w|$ potential atoms.  A partial function
$f\colon\Sigma^*\rightharpoonup\Gamma^*$ is a \emph{deterministic $\MSO$
string transduction} if it is specified by the following finite family of
$\MSO$ formulas:
\begin{enumerate}[label=(\roman*)]
\item a sentence $\varphi_{\mathrm{dom}}$ fixing the domain,
$\mathrm{dom}(f)=\{w:w\models\varphi_{\mathrm{dom}}\}$;
\item for each $c\in C$, a \emph{selection formula} $\varphi_c(x)$, which marks
the positions retained under that copy name.  The \emph{selected atoms} of $w$ are $\mathrm{At}(w)=
  \{(c,i):c\in C,\ i\in\mathrm{Pos}(w),\
                 w\models\varphi_c(i)\}$;
\item for each $c\in C$ and each output letter $\gamma\in\Gamma$, a
\emph{label formula} $\psi_{c,\gamma}(x)$.  For every
$w\in\mathrm{dom}(f)$ and every selected atom $(c,i)\in\mathrm{At}(w)$, there
must be exactly one $\gamma$ such that $w\models\psi_{c,\gamma}(i)$.  This
unique letter is the \emph{label} of the atom, written
$\mathrm{lab}(c,i):=\gamma$;
\item for each ordered pair $c,c'\in C$, an \emph{ordering formula}
$\chi_{c,c'}(x,y)$.  On every $w\in\mathrm{dom}(f)$, these formulas must
define a linear order $<_\chi$ on $\mathrm{At}(w)$ by
\[
   (c,i)<_\chi(c',j)
   \quad\Longleftrightarrow\quad
   w\models\chi_{c,c'}(i,j).
\]
\end{enumerate}
These formulas determine the output as follows.  For $w\in\mathrm{dom}(f)$,
list the selected atoms in increasing $\chi$-order, say
$(c_1,i_1)<_\chi\cdots<_\chi(c_N,i_N)$, and concatenate their labels in that
order:
\[
   f(w)=\mathrm{lab}(c_1,i_1)\,\mathrm{lab}(c_2,i_2)\cdots\mathrm{lab}(c_N,i_N).
\]
If there are no selected atoms, this concatenation is $\varepsilon$.
Each selected atom contributes one output letter, so
$|f(w)|=|\mathrm{At}(w)|\le K\,|w|$; in particular the output is at most linear
in the input length.
\end{definition}

\begin{theorem}[Engelfriet--Hoogeboom~\cite{EngelfrietHoogeboom2001}]
\label{thm:eh}
A partial function $f\colon\Sigma^*\rightharpoonup\Gamma^*$ is computed by a
deterministic 2DFT if and only if it is a deterministic $\MSO$ string
transduction.  The class defined by these equivalent conditions is closed
under composition.  Moreover, if $f$ belongs to this class and
$L\subseteq\Gamma^*$ is regular, then the inverse image $f^{-1}(L)$ is regular.
\end{theorem}

The later arguments use two of these facts: the equivalence of the two
descriptions, and preservation of regularity under inverse images.  Closure
under composition is recorded for completeness; where a composition is
actually formed, in the proof of
Proposition~\ref{prop:two-pyramid-criterion}, it is taken at the polyregular
level instead~\cite{Bojanczyk2018}.

\begin{example}[Reverse-complement, in both presentations]\label{ex:rc-two-ways}
The reverse-complement $\rc$ of Table~\ref{tab:named-bijections} sends
$w=a_1\cdots a_n$ to $\rc(w)=\comp(a_n)\,\comp(a_{n-1})\cdots\comp(a_1)$, where
$\comp$ swaps $U$ and $D$; it is an involution carrying $\D_n$ to $\D_n$.  Here
is the same map described in both ways.

\vspace{0.5em}
\noindent\emph{As a 2DFT.}  Make one rightward pass to the right end marker emitting
nothing, then one leftward pass emitting $\comp(a)$ as the head crosses each
input letter $a$, and halt in an accepting state at the left marker.  The
leftward pass performs the reversal directly: after reaching the end, the
machine revisits the input letters from right to left and emits each
complemented letter immediately.

\vspace{0.5em}
\noindent\emph{As an $\MSO$ string transduction.}  Take
$\varphi_{\mathrm{dom}}:=\mathrm{true}$ and $C=\{1\}$ (so $K=1$).  Keep every
position by setting $\varphi_1(x):=\mathrm{true}$, so
$\mathrm{At}(w)=\{1\}\times\mathrm{Pos}(w)$.  Relabel each position by its
complement: set $\psi_{1,U}(x):=P_D(x)$ and
$\psi_{1,D}(x):=P_U(x)$, so $\mathrm{lab}(1,i)=\comp(a_i)$.  Finally, reverse
the order by taking $\chi_{1,1}(x,y):=(y<x)$; equivalently,
$(1,i)<_\chi(1,j)$ exactly when $j<i$.  Listing the atoms in $\chi$-order
therefore visits positions $n,n-1,\ldots,1$ with labels
$\comp(a_n),\ldots,\comp(a_1)$, whose concatenation is $\rc(w)$.

The two descriptions compute the same map.  Theorem~\ref{thm:eh} says that
every 2DFT has an equivalent logical description of this kind, and conversely.
\end{example}

\subsection{Polyregular maps}

Polyregular maps form a standard extension of regular/2DFT word transformations
in automata theory~\cite{Bojanczyk2018}.  The class includes
natural transformations beyond 2DFTs, such as repeating a word once for each
input position, concatenating all of its prefixes, or listing all of its
contiguous subwords.  Its breadth makes $\polyreg$ a natural benchmark for the zeta map:
asking whether $\zeta$ is polyregular tests whether its global reordering lies
within this powerful, established class of word transformations.

The tuple-based definition below makes the extension beyond 2DFTs precise.  In
a deterministic $\MSO$ string transduction, each selected atom pairs one copy
name with one input position, whereas a polyregular map may pair the copy name
with a fixed-length tuple of positions.  Since an input of length $n$ has $O(n^k)$
$k$-tuples, this is the logical counterpart of a fixed amount of nested
iteration over input positions and permits polynomial-size output.

\begin{definition}[Polyregular presentations and maps]
\label{def:polyregular}
A \emph{polyregular presentation} is a tuple-based $\MSO$ presentation that
generalises Definition~\ref{def:mso-transduction} by assigning each copy name $c$
a fixed \emph{arity} $k_c\ge 1$.  A potential atom is now a pair
$(c,\bar i)$ with $\bar i\in\mathrm{Pos}(w)^{k_c}$.  A polyregular
presentation consists of the following finite data:
\begin{enumerate}[label=(\roman*)]
\item a nonempty finite set $C$ of \emph{copy names}, each with an \emph{arity}
$k_c\ge 1$;
\item a sentence $\varphi_{\mathrm{dom}}$ specifying the domain;
\item for each $c\in C$, a \emph{selection formula}
$\varphi_c(x_1,\ldots,x_{k_c})$, whose satisfying tuples are the \emph{selected
atoms}
\[
   \mathrm{At}(w)
   =\{(c,\bar i):c\in C,\ \bar i\in\mathrm{Pos}(w)^{k_c},
     \ w\models\varphi_c(\bar i)\};
\]
\item for each $c\in C$ and each output letter $\gamma\in\Gamma$, a
\emph{label formula} $\psi_{c,\gamma}(x_1,\ldots,x_{k_c})$, required so that
every selected atom $(c,\bar i)\in\mathrm{At}(w)$ satisfies exactly one of them;
that unique $\gamma$ is the output letter attached to the atom, or its
\emph{label} $\mathrm{lab}(c,\bar i)$;
\item an \emph{ordering formula}
$\chi(c,x_1,\ldots,x_{k_c},c',y_1,\ldots,y_{k_{c'}})$, which on each $w$
linearly orders the selected atoms $\mathrm{At}(w)$.\footnote{Strictly speaking,
this notation suppresses a finite family of $\MSO$ formulas, one for each pair
$c,c'\in C$, with the variable list determined by the arities $k_c$ and
$k_{c'}$.  Copy names are external finite tags rather than variables of the
word structure; when comparing atoms $(c,\bar i)$ and $(c',\bar j)$, we use the
corresponding formula $\chi_{c,c'}(\bar x,\bar y)$.}
\end{enumerate}
For every $w\models\varphi_{\mathrm{dom}}$, list the selected atoms in increasing
$\chi$-order and concatenate their labels.  These data determine a partial
word-to-word map $f\colon\Sigma^*\rightharpoonup\Gamma^*$ with
$\mathrm{dom}(f)=\{w:w\models\varphi_{\mathrm{dom}}\}$ and
$|f(w)|=|\mathrm{At}(w)|$.  A partial word-to-word map is \emph{polyregular}
if it is determined by some polyregular presentation.  The \emph{arity} of the
presentation is $\max_c k_c$, and the class of all polyregular maps is denoted
$\polyreg$.
\end{definition}

\begin{example}[A quadratic Dyck-path map]
For a word $w\in\{U,D\}^*$, let $u(w)$ be the number of its $U$'s and set
$\mathrm{Sq}(w)=w^{u(w)}$, where the exponent means repeated concatenation.
This map has a particularly
simple arity-$2$ polyregular presentation.  Take $C$ to contain one copy name,
select a pair $(x,y)$ exactly when position $x$ carries a $U$, label the pair by
the letter at
position $y$, and order the selected pairs lexicographically, first by $x$ and
then by $y$.  In the notation of Definition~\ref{def:polyregular}, take
\[
  \varphi_{\mathrm{dom}}:=\mathrm{true},\qquad
  \varphi(x,y):=P_U(x),\qquad
  \psi_U(x,y):=P_U(y),\qquad
  \psi_D(x,y):=P_D(y),
\]
and
\[
  \chi(x,y,x',y') :=\ (x<x')\ \lor\ (x=x'\land y<y').
\]
For each $U$-position $x$, the second coordinate $y$ runs once through the
whole input, producing one copy of $w$; the first coordinate orders these
copies from left to right.  Hence the output is $w^{u(w)}$ as claimed.

If $P\in\D_n$, then $u(P)=n$, so
$\mathrm{Sq}(P)=P^n\in\D_{n^2}$.  Thus, this construction sends a Dyck path of
semilength $n$ to one of semilength $n^2$.  Its quadratic output growth cannot be achieved by an
arity-$1$ presentation; the two tuple coordinates have the concrete roles of
choosing a repetition and choosing a position within that repetition.
\end{example}

This definition also has a concrete programming interpretation.  By the
equivalence with pebble transducers and \emph{for-programs}
from~\cite{Bojanczyk2018,BojanczykKieferLhote2019}, a polyregular map is computed
by a fixed program whose loops range over input positions, whose memory
consists of finitely many position-valued and Boolean variables, and whose
output is appended letter by letter.  There are no \texttt{while}-loops or
other forms of unbounded iteration.  With $k$ nested position loops, the
program ranges over $k$-tuples and can produce $O(n^k)$ output letters.  At
arity $1$, there is no nesting: one recovers the deterministic $\MSO$
string transductions of Definition~\ref{def:mso-transduction}, equivalently
the two-way transducers of Section~\ref{sec:2dft-mso}.

The structural fact about $\polyreg$ that we need is the
\emph{linear-growth collapse}.

\begin{theorem}[Linear-growth collapse~\cite{Bojanczyk2023Growth}]
\label{thm:polyregular-linear-collapse}
If $f\in\polyreg$ has linear output growth on its domain (that is, if
there is a constant $K$ with $|f(w)|\le K(|w|+1)$ for every
$w\in\mathrm{dom}(f)$), then $f$ is a deterministic $\MSO$ string
transduction, equivalently a map computed by a deterministic 2DFT.
\end{theorem}

The cited collapse theorem is stated for total polyregular functions.  The
domain of a polyregular presentation is $\MSO$-definable and hence regular.
The partial form above follows by extending $f$ with the empty output outside
its domain, applying the total result, and then restricting the resulting
transduction back to $\operatorname{dom}(f)$.
Thus, nested position loops add no expressive power when the resulting output
remains linear: every such polyregular map is already a deterministic $\MSO$
string transduction.
Theorem~\ref{thm:polyregular-linear-collapse} is the key reduction in our
non-polyregularity proof for $\zeta$.  Starting from a hypothetical
polyregular realisation, we precompose it with a suitable 2DFT and restrict
its domain so that the resulting polyregular map has linear output growth.
The collapse theorem then makes this restricted map a 2DFT, whose preservation
of regularity under inverse images conflicts with an explicit probe language.
The details are carried out in Theorem~\ref{thm:zeta-not-polyregular}.

The common feature of the standard models reviewed above is that the final output order
is specified by finite-state or $\MSO$ data.  They do not provide a primitive
for sorting selected atoms by an integer statistic, such as the running height
of a Dyck path, whose range grows with the input.  This is the operation added
in the next section.

\section[Weighted-rank polyregular maps]{Weighted-rank polyregular maps ($\WRP$)}
\label{sec:wrp}

This section introduces the main computational model of the paper.  We first
define the numerical atom ranks used to order selected atoms, then define $\WRP$
and its principal fragments.  We next establish the structural properties that
locate the model in the transducer hierarchy and explain its limitations.
The complete proofs of the longer structural statements are collected in
Appendix~\ref{app:wrp-structural-proofs}.

\paragraph{Levels, heights, and ranks.}

We distinguish three related numerical labels.  A \emph{level} is the scalar
running label attached to a step of a classical sweep map.  For the map $H$
considered below, the step weights are $+1$ and $-1$, so the level of a
step is its starting \emph{height}.  In the model developed here, the key may
be vector-valued and may belong to an arbitrary selected atom rather than to a
single step.  We call the key assigned to a selected atom its \emph{atom rank},
and the resulting ordering of selected atoms a \emph{rank sort}.  Thus, a
classical sweep map is the scalar, one-atom-per-step case.

We now add the operation missing from the standard hierarchy.  A polyregular
presentation selects atoms, labels them, and orders them by an $\MSO$
formula $\chi$.  A \emph{weighted-rank polyregular map} ($\WRP$) keeps the
selection and labelling but changes the final order.  Each atom receives an
atom rank, an integer vector computed by deterministic weighted scans; the atoms are
sorted by atom rank, with $\chi$ breaking ties.  This is the single global rank
sort used by zeta-like maps; for zeta, its source is the running height.

The rank sort is global in its effect, but evaluating it does not require the
complete list of atoms to be stored.  Later in this section we prove that a
fixed $\WRP$ presentation can be evaluated using only $O(\log |w|)$ bits of
working memory.  Concretely, the evaluator retains only a fixed number of
position-sized indices and counters at once.

\subsection{Additive rank sources and prefix-additive rank functions}

\begin{definition}[Deterministic additive rank source]
\label{def:rank-source}
A \emph{$d$-dimensional deterministic additive rank source} over an alphabet
$\Sigma$ is a deterministic finite-state scanner whose transitions carry
weights in $\Z^d$.  Formally, it is a tuple
\[
   A=(Q,q_0,\delta,\omega),
\]
where $Q$ is a finite state set, $q_0\in Q$ the initial state,
$\delta\colon Q\times\Sigma\to Q$ a deterministic transition function, and
$\omega\colon Q\times\Sigma\to\Z^d$ a weight function.
\end{definition}

Thus, a rank source has the same finite control as a DFA\@.  It has no accepting
set because it is used to scan the input and accumulate an integer vector,
rather than to recognise a language.

\begin{definition}[Prefix rank]\label{def:prefix-rank}
Let $A$ be a deterministic additive rank source and let
$w=a_1\cdots a_n$.  Define
$q_i=\delta(q_{i-1},a_i)$ for $i=1,\ldots,n$.  The \emph{prefix rank} of $A$
before position $i\in\{1,\ldots,n+1\}$ is the total weight accumulated before
reading position $i$:
\[
   \rho_A^w(i)=\sum_{j<i}\omega(q_{j-1},a_j)\in\Z^d.
\]
We also write $q^A_i:=q_{i-1}$ for the state of $A$ just before position $i$.
\end{definition}

Thus, $\rho_A^w(1)=0$, while $\rho_A^w(n+1)$ is the total weight accumulated
on the whole word.

\begin{example}[Height as a rank source]\label{ex:height-rank}
Let $A$ be the one-state source with
$\omega(\bullet,U)=+1$ and $\omega(\bullet,D)=-1$.  Then, $\rho_A^w(i)$ is the
height of $w$ just before position $i$.  In particular, if $P$ is a Dyck path
and $p_j$ is the position of its $j$-th up-step, then
$\rho_A^P(p_j)=a_j$, the $j$-th area-sequence entry.
\end{example}

Finite-state devices that accumulate numerical weights are standard in the
theory of weighted automata and cost-register automata; the latter make
explicit that several numerical quantities may be maintained and combined by
operations such as addition and scaling~\cite{AlurEtAl2013}.  A closer
precedent for the ordering mechanism used here comes from ranked $\MSO$
enumeration~\cite{BourhisEtAl2021}.  For a fixed input word, an $\MSO$ formula
with free variables selects its satisfying assignments, and a weighted $\MSO$
formula assigns a cost in an ordered abelian group to each assignment.  The
enumeration task is to list all satisfying assignments, without repetition,
in nondecreasing cost order.  One example assigns to a pair $(x,y)$ a vector
obtained from a prefix count at $x$ and another at $y$, and orders these
vectors lexicographically.  Our prefix-additive rank functions form a
deterministic fragment of this cost formalism: selected atoms play the role of
satisfying assignments, and their atom ranks play the role of costs.

\begin{definition}[Prefix-additive rank function]
\label{def:prefix-additive-rank}
Fix $k\ge 1$.  For each coordinate $r\in\{1,\ldots,k\}$, fix a
$d$-dimensional deterministic additive rank source
$A_r=(Q_r,q_{0,r},\delta_r,\omega_r)$ on $\Sigma$, together with a fixed
local-correction table
$\beta_r\colon Q_r\times\Sigma\to\Z^d$.  A
\emph{$d$-dimensional prefix-additive rank function} on $k$-tuples of input
positions is a function of the form
\[
   \kappa^w(x_1,\ldots,x_k)
   =c_0+\sum_{r=1}^{k}
       \Bigl(\rho_{A_r}^w(x_r)
             +\beta_r(q^{A_r}_{x_r},a_{x_r})\Bigr),
   \qquad c_0\in\Z^d.
\]
We usually suppress the superscript $w$.  A coordinate that makes no
contribution uses the one-state zero-weight source and the zero local
correction.
\end{definition}

The inclusion in the ranked $\MSO$ cost formalism is explicit.  For fixed
$r,q,a$, the condition that a position $z<x_r$ carries $a$ and that $A_r$ is
in state $q$ just before $z$ is $\MSO$-definable.  Summing the fixed vector
$\omega_r(q,a)$ over all such $z$, and then over the finitely many pairs
$(q,a)$, gives $\rho_{A_r}^w(x_r)$.  The correction at $x_r$ is a finite
$\MSO$ case distinction on its letter and the state there, and $c_0$ is a
fixed cost.  Hence every
prefix-additive rank function is a cost function defined by a weighted $\MSO$
formula in the sense of
\cite{BourhisEtAl2021}.  We use this restricted deterministic form rather
than the general weighted $\MSO$ cost formalism because its operational and
semilinear behaviour is transparent in the evaluation and lower-bound
arguments below.

The two contributions in Definition~\ref{def:prefix-additive-rank} play
different roles.  The prefix ranks carry the genuinely unbounded quantities;
for $\zeta$, this is the height at an up-step
(Example~\ref{ex:height-rank}).  Each $\beta_r$ is only a bounded correction
read off from the state and letter at the marked position.  It may, for
instance, convert the prefix rank ``just before'' a position into the
corresponding value ``just after'' it.  Such corrections make the definition
insensitive to off-by-one conventions; $\zeta$ itself uses one prefix rank and
no local correction.

One might instead allow a finite affine combination of prefix ranks, with
several sources evaluated at the same coordinate.  That apparently more
general syntax defines exactly the same functions.  Concretely, consider
\[
   \widetilde\kappa^w(\bar x)
   =c_0+\sum_{t=1}^{m}
      \Bigl(c_t\,\rho_{A_t}^w(x_{\pi(t)})
            +\beta_t(q^{A_t}_{x_{\pi(t)}},a_{x_{\pi(t)}})\Bigr),
\]
where the $A_t$ are $d$-dimensional sources, $c_t\in\Z$, $\pi(t)$ chooses a
tuple coordinate, and each $\beta_t$ is a local-correction table.  For each
coordinate $r$, group the terms with $\pi(t)=r$.  Their product automaton is a
single source $A_r$: its state records all component states, its transition
weight is $\sum_{\pi(t)=r}c_t\,\omega_t(q_t,a)$, and its local correction is
the sum of the corresponding correction tables.
Its contribution at $x_r$ is exactly the sum of the old contributions there.
Keeping $c_0$ unchanged gives $\widetilde\kappa$ in the form of
Definition~\ref{def:prefix-additive-rank}.  Conversely, that definition is
already such an affine combination with one coefficient-$1$ term per
coordinate.  Thus, integer linear combinations are absorbed into the
transition weights rather than exposed as a separate operation in the
definition.

\subsection{The WRP model and its fragments}

\begin{definition}[Weighted-rank polyregular map]\label{def:wrp}
A partial word-to-word map $T\colon\Sigma^*\rightharpoonup\Gamma^*$ is a
\emph{weighted-rank polyregular map} if it admits a \emph{polyregular
presentation}, in the sense of Definition~\ref{def:polyregular}, augmented by:
\begin{itemize}
\item a fixed \emph{rank dimension} $d\ge 0$;
\item for each $c\in C$, a $d$-dimensional prefix-additive rank function
$\kappa_c(x_1,\ldots,x_{k_c})\in\Z^d$
(Definition~\ref{def:prefix-additive-rank}).
\end{itemize}
For a selected atom $\alpha=(c,\bar i)$, the vector
$\kappa_c^w(\bar i)$ is its \emph{atom rank}.  We usually suppress $w$ when
the input is clear.  Thus, a prefix rank is the running total produced by one
rank source, whereas an atom rank is the final sorting key assembled from such
prefix ranks, local corrections, and a constant.
The output order on selected atoms is the lexicographic order $\prec$
defined by
\[
   (c,\bar i)\prec(c',\bar i')
   \quad\iff\quad
   \kappa_c(\bar i)<_{\mathrm{lex}}\kappa_{c'}(\bar i')
   \ \text{or}\
   \bigl(\kappa_c(\bar i)=\kappa_{c'}(\bar i')\ \text{and}\
          \chi\bigl((c,\bar i),(c',\bar i')\bigr)\bigr),
\]
where $\chi$ is the ordering formula of the underlying polyregular
presentation (Definition~\ref{def:polyregular}).  The atom rank is the primary
key, and $\chi$ is used only for atoms of equal atom rank; we therefore call
$\chi$ the \emph{tie-order}.  Since $\chi$ is a strict total order on selected atoms,
$\prec$ is also a strict total order.  For $w\in\mathrm{dom}(T)$, concatenate
the atom labels in increasing $\prec$-order to obtain $T(w)$.  The class of
all such maps is denoted $\WRP$.  Equivalently, first list the selected atoms
in $\chi$-order, then stably sort that list by the lexicographic key
$(c,\bar i)\mapsto\kappa_c(\bar i)$.  This final stable sorting step is the
rank sort of the presentation.

The \emph{arity} of the presentation is $\max_c k_c$.  Two fragments recur
below.  The \emph{ranked-regular} class $\RR$ is the arity-$1$ case, in which
each atom is a pair $(c,i)$ indexed by one input position.  Its
\emph{scan-order one-dimensional} restriction $\SWR$ requires a scalar atom rank
($d=1$) and a \emph{scan order} as the tie-order.  Such an order is determined
by a direction on positions (left-to-right or right-to-left) and a linear order
on the finite set $C$ of copy names: position order decides first, and the order
on $C$ breaks ties between atoms at the same position.  Thus, an $\SWR$ map lists atoms by
increasing integer atom rank and breaks atom-rank ties by scan order.  We have
\[
   \SWR\ \subseteq\ \RR\ \subseteq\ \WRP.
\]
\end{definition}

\begin{example}[A two-dimensional atom rank combining two positions]
Classical sweep maps attach one additive level to each input
position~\cite{ArmstrongLoehrWarrington2015,
ArmstrongLoehrWarrington2016}.  To show how a prefix-additive rank function can
use every coordinate of a selected tuple, we define a quadratic construction
whose key combines the prefix heights at two input positions.

\vspace{0.5em}
\noindent\emph{The transformation.}
For a step word $w$, let $\operatorname{ht}_w(i)=h_w(i-1)$ be the height
immediately before position $i$.  For every ordered pair
$(i,j)$ of distinct positions carrying $U$, define the key
\[
   \kappa(i,j)
   =\bigl(\operatorname{ht}_w(i)+\operatorname{ht}_w(j),\
           \operatorname{ht}_w(j)-\operatorname{ht}_w(i)\bigr).
\]
The first component is the total height of the pair, and the second is its
signed height difference.  Associate to the pair the Dyck block
\[
   B_{i,j}=
   \begin{cases}
      UD,   & i<j,\\
      UUDD, & j<i.
   \end{cases}
\]
Order the pairs lexicographically by $\kappa(i,j)$, breaking equal-key ties by
the lexicographic order on $(i,j)$, and concatenate their associated blocks in
that order.  Since each block is a Dyck path, so is the resulting word.  For
an input in $\D_n$, there are $\binom n2$ blocks of each type, so the output
lies in $\D_{3\binom n2}$.

The key has a direct area-sequence interpretation.  If $P\in\D_n$ has
up-step positions $p_1<\cdots<p_n$ and area sequence
$a=(a_1,\ldots,a_n)$, then $\operatorname{ht}_P(p_r)=a_r$, so
$\kappa(p_r,p_s)=(a_r+a_s,a_s-a_r)$.
Thus, the first component orders pairs by their combined area level, while the
second orders pairs on the same level by their signed separation.  The latter
is the pairwise difference underlying $\dinv$: a pair $r<s$ contributes to
$\dinv(P)$ precisely when the second component is $0$ or $-1$.

For example, take $P=UUDDUD\in\D_3$, whose area sequence is $(0,1,0)$.
The six ordered pairs of distinct up-steps have the following atom ranks and
blocks:
\[
\begin{array}{c|cccccc}
(r,s) &(1,2)&(1,3)&(2,1)&(2,3)&(3,1)&(3,2)\\ \hline
\kappa(p_r,p_s)&(1,1)&(0,0)&(1,-1)&(1,-1)&(0,0)&(1,1)\\
B_{p_r,p_s}&UD&UD&UUDD&UD&UUDD&UUDD
\end{array}
\]
Sorting by key, with the stated tie-order, gives
\[
   (UD)(UUDD)(UUDD)(UD)(UD)(UUDD)\in\D_9.
\]
Among the four pairs of total height $1$, the signed-difference coordinate
places the two pairs of difference $-1$ before the two pairs of difference
$1$.

\vspace{0.5em}
\noindent\emph{A $\WRP$ presentation.}
We now realise this block description in the formalism of
Definition~\ref{def:wrp}.  Define two one-state, two-dimensional additive rank
sources by
\[
\begin{array}{c|cc}
 & U & D\\ \hline
\omega_{A_-} & (1,-1) & (-1,1)\\
\omega_{A_+} & (1,1)  & (-1,-1).
\end{array}
\]
Use $A_-$ for coordinate $x_1$ and $A_+$ for coordinate $x_2$, with
$c_0=(0,0)$ and both local-correction tables zero.  Their prefix ranks satisfy
\[
   \rho_{A_-}^w(x_1)=(\operatorname{ht}_w(x_1),-\operatorname{ht}_w(x_1)),\qquad
   \rho_{A_+}^w(x_2)=(\operatorname{ht}_w(x_2),\operatorname{ht}_w(x_2)),
\]
and their sum is $\kappa(x_1,x_2)$.  Thus, $\kappa$ has exactly the
one-source-per-coordinate form of
Definition~\ref{def:prefix-additive-rank}; no appeal to the more general
many-source affine syntax discussed above is needed.

Take the six arity-$2$ copy names
$C=\{f_1,f_2,b_1,b_2,b_3,b_4\}$ and the domain sentence $\mathrm{true}$.
Select
their position pairs by
\[
\begin{aligned}
   \varphi_{f_\ell}(x_1,x_2)
   &:=P_U(x_1)\land P_U(x_2)\land x_1<x_2
      &&(\ell=1,2),\\
   \varphi_{b_\ell}(x_1,x_2)
   &:=P_U(x_1)\land P_U(x_2)\land x_2<x_1
      &&(\ell=1,2,3,4).
\end{aligned}
\]
Give $f_1,f_2$ the constant labels $U,D$, respectively, and give
$b_1,b_2,b_3,b_4$ the constant labels $U,U,D,D$.  Every copy name uses the
prefix-additive rank function $\kappa(x_1,x_2)$ defined above.  Let the
tie-order first compare the position pairs lexicographically and, for the same pair, use $f_1<f_2$ or
$b_1<b_2<b_3<b_4$.  The selection and labelling data and this tie-order are
all $\MSO$-definable.  For a pair $x_1<x_2$, the two selected atoms spell
$UD$; for a pair $x_2<x_1$, the four selected atoms spell $UUDD$.  Hence this
presentation produces exactly the block transformation defined above.

The two coordinate sources are evaluated at different tuple positions and
their contributions are added.  They therefore cannot be replaced by one
prefix-rank evaluation at a single position.  Thus, both tuple coordinates
contribute to the atom rank, and both components of the atom rank affect the
output order.
\end{example}

\begin{remark}[The scalar scan-order case]
The stable-sort description above becomes especially concrete for $\SWR$.
An $\SWR$ map produces only a bounded number of selected atoms at each input
position, gives each atom a single integer atom rank, and uses a scan order to
break ties.  Thus, the initial list is the selected atoms in scan order, and the
output is the stable sort of that list by the integer atom rank.  Atoms with
equal atom rank keep their scan order, while atoms of smaller atom rank move
earlier in the output.
When selection and labelling can be decided in one left-to-right pass, as for
the zeta map and the height-sweep example used later, the output can be computed in
$O(|w|^2)$ time and $O(\log|w|)$ space
(Corollary~\ref{cor:srr-quadratic}).  The zeta map is an $\SWR$ map
(Theorem~\ref{thm:zeta-wrp}); in that case, the rank source is the running
height.
\end{remark}

The model is deliberately \emph{one-layered}.  A $\WRP$ presentation first
selects and labels its atoms by the polyregular, $\MSO$-definable part, and
only then sorts those already-selected atoms by atom rank.  The selection and
label formulas cannot ask numerical questions about atom ranks, and the output of one
rank-sort step cannot feed a second rank-sort layer.  Thus, $\WRP$ captures one
zeta-style rank sort, not an arbitrary pipeline of repeated sorts.

\subsection{Position in the hierarchy}

\begin{proposition}[Conservativity]
\label{prop:conservative}
The rank-zero fragment of $\WRP$ (rank dimension $d=0$, or equivalently
every prefix-additive rank function identically zero) coincides with $\polyreg$.
Consequently, $\polyreg\subseteq\WRP$.
\end{proposition}

\begin{proof}
With rank dimension $d=0$, every prefix-additive rank function takes values in
$\Z^0=\{0\}$, the lexicographic order $<_{\mathrm{lex}}$ on $\Z^0$ is the empty
relation, and the output order is determined entirely by the $\MSO$
tie-order $\chi$.  This is exactly Definition~\ref{def:polyregular}.
\end{proof}

The three models differ in two concrete ways: how many selected atoms they can
create and how they may order those atoms.  A deterministic two-way
finite-state transduction has only linearly many atoms, each attached to one
input position.  A polyregular map may attach an atom to a fixed tuple of
positions and therefore has polynomial output growth.  A $\WRP$ map keeps
these polyregular atoms but may finally sort them by an unbounded numerical
key computed from additive finite-state scans.

Accordingly, the results of this paper place the models in the strict chain
\[
   \TDFT/\MSO\ \subsetneq\ \polyreg\ \subsetneq\ \WRP
   \ \subsetneq\ \text{deterministic logspace}.
\]
The first strict inclusion is classical~\cite{Bojanczyk2018,
BojanczykKieferLhote2019} and can already be witnessed by the quadratic map
$a^n\mapsto(a^n)^n$.  The zeta map witnesses the second:
it is semilength-preserving and belongs to $\WRP$, but is not polyregular
(Theorem~\ref{thm:intro-zeta-classification}).  Thus, the additional power is
not greater output length; it is the ability to order selected atoms by an
unbounded atom rank.  The logspace upper bound and the strictness of the final
inclusion are established below.

\subsection{Closure under basic constructions}

The class $\WRP$ is stable under a variety of standard constructions for
modifying and combining word-to-word maps.  The following theorem states the
precise closure properties that we will use.

\begin{theorem}[Basic closures]\label{thm:wrp-closures}
For every fixed arity $k\ge 1$, the class of $\WRP$ maps admitting a
presentation of arity at most $k$ is closed under the following operations.
All maps combined in one construction are assumed to have the same input
alphabet.
\begin{enumerate}[label=(\roman*)]
\item \emph{Restriction and definition by cases.}
If $T\colon\Sigma^*\rightharpoonup\Gamma^*$ is a $\WRP$ map and
$L\subseteq\Sigma^*$ is regular (Definition~\ref{def:regular-language}), then
the restriction of $T$ to $L\cap\operatorname{dom}(T)$ is in $\WRP$.
More generally, let $L_1,\ldots,L_r$ be pairwise disjoint regular languages,
and let $T_j\colon\Sigma^*\rightharpoonup\Gamma^*$ be a $\WRP$ map for each
$j$.  The partial map with domain
\[
   \bigcup_{j=1}^r\bigl(L_j\cap\operatorname{dom}(T_j)\bigr)
\]
and values
\[
   T(w)=T_j(w)\qquad
   \text{when }w\in L_j\cap\operatorname{dom}(T_j)
\]
is in $\WRP$.

\item \emph{Combining outputs with source tags.}
Let $T_j\colon\Sigma^*\rightharpoonup\Gamma_j^*$ be $\WRP$ maps, for
$1\le j\le r$.  If $v=a_1\cdots a_m\in\Gamma_j^*$, write
\[
   \operatorname{tag}_j(v)=(j,a_1)\cdots(j,a_m).
\]
Thus, the tag records which map produced each letter.  On the common domain
$\bigcap_{j=1}^r\operatorname{dom}(T_j)$, the map
\[
   w\longmapsto
   \operatorname{tag}_1(T_1(w))\,
   \operatorname{tag}_2(T_2(w))\cdots
   \operatorname{tag}_r(T_r(w)),
\]
whose output alphabet is the disjoint union
$\biguplus_{j=1}^r(\{j\}\times\Gamma_j)$, is in $\WRP$.  In particular, the
tagged outputs are placed one after another in the order $1,\ldots,r$; they
are not interleaved.

\item \emph{Concatenation with fixed separators on nonempty inputs.}
Let $T_1,\ldots,T_r\colon\Sigma^*\rightharpoonup\Gamma^*$ be $\WRP$ maps and
fix words $s_1,\ldots,s_{r-1}\in\Gamma^*$ that do not depend on the input.
On the nonempty words in the common domain of the $T_j$, the map
\[
   w\longmapsto
   T_1(w)\,s_1\,T_2(w)\,s_2\cdots s_{r-1}\,T_r(w)
\]
is in $\WRP$.  For example, taking two maps and $s_1=\#$ gives
$w\mapsto f(w)\,\#\,g(w)$.

\item \emph{Replacing or deleting output letters according to a fixed rule.}
Let $T\colon\Sigma^*\rightharpoonup\Gamma^*$ be in $\WRP$.  Assign to every
$a\in\Gamma$ either one fixed letter of an alphabet $\Delta$ or the empty word
$\varepsilon$, and apply this assignment separately to every letter of
$T(w)$.  Equivalently, for a fixed map $h\colon\Gamma\to\Delta^*$ satisfying
$|h(a)|\le 1$, set
$\widehat h(a_1\cdots a_m)=h(a_1)\cdots h(a_m)$.
Then $\widehat h\circ T$ is in $\WRP$.  The choice $h(a)=b$ replaces every
$a$ by $b$, whereas $h(a)=\varepsilon$ deletes every occurrence of $a$.

\item \emph{Output reversal.}
If $T\colon\Sigma^*\rightharpoonup\Gamma^*$ is in $\WRP$, then so is
$w\mapsto\operatorname{rev}(T(w))$ on the same domain, where
$\operatorname{rev}(a_1\cdots a_m)=a_m\cdots a_1$.
\end{enumerate}
If the maps used in one of these constructions have presentation arities
$k_1,\ldots,k_r$, the resulting map has a presentation of arity at most
$\max_j k_j$.  In particular, the ranked-regular fragment $\RR$ (arity $1$)
is closed under all five constructions.
\end{theorem}

The ordering idea for combining outputs is clearest in the concatenation case.
For a nonempty input,
$f(w)\,\#\,g(w)$ has three consecutive blocks: $f(w)$, the
separator $\#$, and $g(w)$.  Every output letter in a $\WRP$ presentation is
carried by a selected atom.  We add a new first atom-rank coordinate with value $0$ on
the atoms producing $f(w)$, value $1$ on the new atom producing the separator
$\#$, and value $2$ on the atoms producing $g(w)$.  Lexicographic comparison
then places the three blocks in this order, while the original atom-rank
coordinates and tie-orders preserve the internal orders of $f(w)$ and $g(w)$.
The same block-number device handles source-tagged outputs.  The other cases
only restrict inputs, change or delete labels, or reverse the ordering data.
The complete constructions and proof of the arity bound appear in
Appendix~\ref{app:wrp-structural-proofs}.

\subsection{Logspace evaluation and expressive power}
\label{sec:wrp-evaluation}

A $\WRP$ presentation specifies which atoms are selected, their labels, and
their order, but is not an explicit evaluation algorithm.  In particular,
``sort the selected atoms by atom rank'' does not provide a stored list or a
cost-free sorting operation.

Let $n=|w|$.  A fixed arity-$k$ presentation has $O(n^k)$ potential atoms, so
one polynomial-time strategy materialises them, determines their selection,
labels, and atom ranks, sorts the selected atoms, and prints their labels.
This direct strategy may require polynomial working memory.

We show that this polynomial-size list need not be stored.  By a deterministic logarithmic-space
evaluator, we mean an algorithm with read-only access to the input, write-only
access to the output, and $O(\log n)$ bits of working memory.  Naming an input
position, or storing an atom rank or counter of polynomial magnitude, costs
$\Theta(\log n)$ bits.  Such workspace therefore holds only a fixed number of
these quantities, depending on the presentation but not on $n$.  Unlike a
finite-state transducer, it can maintain unbounded counters, but only a fixed
number of them, not the complete growing list of selected atoms.

The selection, label, atom rank, and relative order of any fixed atoms can be
recomputed from the input in logarithmic space.  By repeatedly enumerating the
potential atoms, the evaluator locates and emits one atom at a time in sorted
order while retaining only a constant number of atom names and counters.  The
repeated scans take polynomial time but reuse the same logarithmic workspace.

This tradeoff between time and space supplies the operational interpretation
of $\WRP$.  It places the declaratively defined class inside deterministic
logspace while showing that its global rank sort does not require polynomial
storage.  We also prove that not every deterministic-logspace transformation
admits such a one-layer rank-sort presentation.

\begin{theorem}[Logspace evaluation]\label{thm:wrp-logspace}
Every fixed $\WRP$ map $T$ of arity $k$ is computable in deterministic
\emph{logarithmic space}.  On an input $w$ of length $n$, an evaluator uses
only $O(\log n)$ bits of working memory beyond a read-only input and a
write-only output.  Its output has length $|T(w)|=O(n^k)$ and is produced in
polynomial time.
\end{theorem}

Space here and below is measured in bits.  For the precise time bound in the
following corollary, we use unit-cost operations on $O(\log n)$-bit words: an
input position, atom rank, or counter of that size fits in one machine word, and
arithmetic and comparison on one word count as one step.  If time is instead
counted bit by bit, each such word operation may require $O(\log n)$ steps.
Accordingly, the $O(n^2)$ time conclusion of the corollary becomes
$O(n^2\log n)$, while its $O(\log n)$ space bound and the polynomial-time
conclusion of Theorem~\ref{thm:wrp-logspace} are unchanged.

\begin{corollary}[Quadratic scan-order evaluation]
\label{cor:srr-quadratic}
Let $T$ be an $\SWR$ map whose selection and labelling are decided by a
single left-to-right finite-state pass, so the choice at each position is
determined by the prefix ending there.  Then, on an input $w$ of length $n$,
$T$ is computable in deterministic time $O(n^2)$ using $O(\log n)$ bits of
working memory.
\end{corollary}

The general algorithm favours space over time: it emits one output letter per
round, and each round recomputes, from the atom emitted last, its successor in
the output order.  Under the word-cost convention above, this takes
$O(n^{2k+1})$ time: there are $O(n^k)$ output rounds, each round scans the
$O(n^k)$ potential atoms once while retaining the best atom seen so far,
and each selection test or atom comparison may require an $O(n)$ scan of the
input.  Thus, even at arity $1$, this generic bound is $O(n^3)$.

The scan-order one-dimensional algorithm follows the same successor-round
strategy, but the additional hypotheses of Corollary~\ref{cor:srr-quadratic}
make each round a single forward pass.  There are only $O(n)$ potential atoms,
and one left-to-right scan maintains the running atom ranks, decides selection
and labels from the prefix, and retains the least selected atom following the
atom emitted in the previous round.  Hence $O(n)$ output rounds, each taking
$O(n)$ time, give the quadratic bound.  These are upper bounds for the two
evaluation procedures just described, not lower bounds for the maps
themselves.  The detailed algorithms and their resource analyses appear in
Appendix~\ref{app:wrp-structural-proofs}.

The evaluation theorem establishes the upper bound
$\WRP\subseteq\text{deterministic logspace}$.  The rank-sort layer nevertheless
adds genuine expressive power beyond $\polyreg$, for the central combinatorial
reason of the paper.

\begin{theorem}[Above $\polyreg$]\label{thm:wrp-strict-over-poly}
$\polyreg\subsetneq\WRP$ on the Dyck domain: $\zeta$ is
in $\WRP$ (Theorem~\ref{thm:zeta-wrp}) but is not in $\polyreg$
(Theorem~\ref{thm:zeta-not-polyregular}).
\end{theorem}

This finite-register interpretation explains the upper bound but does not
characterise $\WRP$.  A general logarithmic-space algorithm may use its fixed
number of position-sized registers adaptively, for example by allowing an
unbounded counter to decide whether any output is produced.  In a $\WRP$
presentation, by contrast, numerical atom ranks may reorder selected atoms but
cannot affect which atoms the $\MSO$ formulas select.  Thus, not every
deterministic-logspace transformation has a $\WRP$ presentation.

\begin{theorem}[Below logspace]
\label{thm:wrp-strict-below-logspace}
$\WRP$ is a proper subclass of the word-to-word maps computable in
deterministic logspace.
\end{theorem}

A concrete witness returns the input word when every prefix has nonnegative
height and returns the empty word otherwise.  A logarithmic-space evaluator
checks this condition with a height counter and, if it holds, makes a second
pass to copy the input to the output.

To explain why this map is not in $\WRP$, recall that a \emph{regular
condition} on input words is one whose satisfying words form a regular
language (Definition~\ref{def:regular-language}), or equivalently one that a
DFA can decide using a fixed finite amount of memory.  For every $\WRP$ map
$T$, the set of inputs $w\in\operatorname{dom}(T)$ for which
$T(w)\ne\varepsilon$ is regular (Lemma~\ref{lem:wrp-nonempty-regular}).
Indeed, the output is nonempty exactly when some $\MSO$ selection formula
selects an atom; the numerical atom ranks only reorder selected atoms and do not
affect their existence.  Write $L_{\ge 0}$ for the language of words all of
whose prefixes have nonnegative height.  This language is not regular: a DFA
with fixed finite memory cannot maintain an unbounded height counter.  The
nonempty-output preimage of the witness is not quite $L_{\ge 0}$, because the
empty input satisfies the prefix condition but is itself sent to the empty
word.  Rather, the preimage is $L_{\ge 0}\setminus\{\varepsilon\}$.  This
language is still nonregular, since otherwise adjoining the regular singleton
$\{\varepsilon\}$ would make $L_{\ge 0}$ regular.  Hence the witness is
computable in deterministic logspace but does not belong to $\WRP$.  The
formal argument appears in
Appendix~\ref{app:wrp-structural-proofs}.

\subsection{Two boundaries of the rank-sort layer}

The rank sort gives $\WRP$ its additional expressive power, but it also marks
a precise boundary.  For deterministic $\MSO$ and polyregular maps, a regular
condition on the output can be translated back into a regular condition on the
input; equivalently, inverse images of regular languages are regular.  This
backward translation need not be possible after selected atoms have been sorted
by unbounded numerical atom ranks.

\begin{theorem}[Closure failure]
\label{thm:wrp-not-closed}
There exist a $\WRP$ map $D$ and a regular language $K$ such
that $D^{-1}(K)$ is not regular.  Moreover, there exists a deterministic
2DFT $S$ whose input head moves only from left to right such that
$S\circ D\notin\WRP$.  Consequently, $\WRP$ is not closed under composition.
\end{theorem}

For the last conclusion, observe that $S$ is itself a $\WRP$ map.  Indeed,
every deterministic 2DFT is a deterministic $\MSO$ string transduction by
Theorem~\ref{thm:eh}, so it is polyregular and therefore belongs to $\WRP$ by
Proposition~\ref{prop:conservative}.  Thus, $D$ and $S$ both belong to $\WRP$
while their composite $S\circ D$ does not.
The construction of $D$ and $S$, and the proof of the theorem, are given in
Appendix~\ref{app:wrp-structural-proofs}.

The purpose of the theorem is to clarify the one-layer boundary of $\WRP$.
A $\WRP$ presentation selects and labels its atoms before performing its one
rank sort; after the sorted word has been emitted, the presentation has no
further stage that can inspect that word and decide what to output next.
Postcomposition by $S$ adds exactly such a stage: $S$ reads the sorted output
of $D$ and acts on what it sees.  The theorem shows that this later inspection
cannot always be absorbed into a new $\WRP$ presentation.  Thus, the failure of
composition is a structural consequence of allowing one rank-sort layer, not
an accidental defect of the definition.  Inversion under the realisation
convention is a separate question: later we prove that $\zeta\in\WRP$ but
$\zeta^{-1}\notin\WRP$ (Section~\ref{sec:inverse-zeta}).

There is a second, independent boundary, concerning the range of the prefix ranks
rather than the number of rank-sort layers.  Suppose that every prefix rank
takes values in a fixed finite set.  For each deterministic additive rank
source, augment its finite control state with its current prefix rank.
Because only finitely many such pairs can occur, the resulting device is an
ordinary finite automaton whose state records the current prefix rank.
Consequently, each possible atom rank, and hence every comparison between two
atom ranks, is
$\MSO$-definable.  The rank sort can then be absorbed into the polyregular
tie-order.

\begin{theorem}[Bounded-prefix-rank collapse]
\label{thm:bounded-rank-collapse}
Let $T$ be a $\WRP$ map whose rank sources are uniformly
\emph{bounded}.  Writing $A_{c,r}$ for the source attached to coordinate $r$
of copy name $c$, suppose that there is a constant $B$ such that
$\lVert\rho_{A_{c,r}}^w(i)\rVert_\infty\le B$ for all inputs
$w\in\mathrm{dom}(T)$, all coordinate sources $A_{c,r}$ in the presentation,
and all prefix positions
$i\in\{1,\ldots,|w|+1\}$.  Then, $T$ is a
polyregular map.
\end{theorem}

The proof makes the preceding finite-state augmentation precise.  In
particular, comparison of two atom ranks becomes a finite disjunction over
pairs of possible values.  Appendix~\ref{app:wrp-structural-proofs} gives the full
argument.

\begin{corollary}[Unbounded prefix rank needed]\label{cor:rank-necessary}
The map $\zeta$ admits no $\WRP$ realisation whose presentation has every rank
source uniformly bounded on the domain of the realising map.  Equivalently,
every $\WRP$ presentation realising $\zeta$ uses at least one rank source whose
prefix ranks are unbounded on that domain.
\end{corollary}

Indeed, if a $\WRP$ map realising $\zeta$ had a presentation in which every
rank source were uniformly bounded on its domain, that map would be
polyregular by the theorem, contradicting
Theorem~\ref{thm:zeta-not-polyregular}.  The usual height-rank presentation
visibly uses an unbounded source: heights on
$\D_n$ can reach $n$.  Thus, the finite-state case in which every numerical
source has a fixed finite range cannot separate $\WRP$ from polyregularity;
any $\WRP$ presentation realising $\zeta$ must have access, somewhere on its
domain, to an unbounded numerical source.

\subsection{A computational roadmap of the Catalan examples}
\label{sec:landscape}

Sections~\ref{sec:dyck} and~\ref{sec:models} introduced, respectively, the
Catalan maps studied in this paper and the established models of word-to-word
computation.  The present section has added the rank-sort model $\WRP$ and
located it in that hierarchy.  We now bring the two sides together before
turning to the classification proofs.  Table~\ref{tab:named-bijections}
described the principal maps in combinatorial terms;
Table~\ref{tab:landscape} revisits them from a computational perspective and
serves as a roadmap for the remaining sections.

The table records two kinds of information.  A membership statement, such as
$H\in\SWR$, is an upper bound: it names a mechanism that realises the map.  A
nonmembership statement, such as $\zeta^{-1}\notin\WRP$, is a lower bound: it
says that the indicated mechanism does not suffice.  In particular, ``not in
$\WRP$'' does not mean ``not computable.''  It means that the map cannot be
realised by the single rank-sort layer of Definition~\ref{def:wrp}.  For the
hypothetical area--dinv swap, it likewise does not say that no such bijection
exists; it says that no such bijection can have a $\WRP$ realisation.  All
classifications are relative to the step-word encoding
(Remark~\ref{rmk:step-word}).

\begin{table}[t]
\centering
\small
\renewcommand{\arraystretch}{1.25}
\begin{tabularx}{\textwidth}{@{}>{\raggedright\arraybackslash}p{0.18\textwidth}
                                   >{\raggedright\arraybackslash}X
                                   >{\raggedright\arraybackslash}p{0.20\textwidth}
                                   >{\raggedright\arraybackslash}X@{}}
\toprule
Map or problem & Computational status (step-word encoding) & Result &
  Combinatorial role \\
\midrule
$\id$ & Realised by a left-to-right 2DFT. & folklore &
  Preserves every statistic. \\
$\rc$ & Realised by a $\TDFT/\MSO$ transduction. &
  Example~\ref{ex:rc-two-ways} &
  A Catalan involution that anchors the two-way finite-state stratum. \\
Zeta $\zeta$ &
  $\zeta\in\SWR\subseteq\WRP$, but $\zeta\notin\polyreg$. &
  Section~\ref{sec:zeta}
  (Thm.~\ref{thm:zeta-wrp} and
  Thm.~\ref{thm:zeta-not-polyregular}) &
  $\area(\zeta(P))=\dinv(P)$. \\
Height sweep $H$ &
  $H\in\SWR\subseteq\WRP$, but $H\notin\polyreg$. &
  Section~\ref{sec:narayana-sweep}
  (Prop.~\ref{prop:alw-sweep-swr} and
  Thms.~\ref{thm:narayana-sweep} and~\ref{thm:H-not-polyregular}) &
  $\dr(H(P))=\val(P)$, $\val(H(P))=\dr(P)$. \\
Zeta inverse $\zeta^{-1}$ &
  $\zeta^{-1}\notin\WRP$; an explicit combinatorial inverse is known. &
  Section~\ref{sec:inverse-zeta}
  (Cor.~\ref{cor:inverse-zeta-not-wrp}) &
  Inverts $\zeta$; $\dinv(\zeta^{-1}(P))=\area(P)$. \\
A hypothetical area--dinv swap &
  No $\WRP$ map can realise such a swap. &
  Section~\ref{sec:main} (Thm.~\ref{thm:wrp-no-swap}) &
  Would witness the $q,t$-Catalan symmetry. \\
\bottomrule
\end{tabularx}
\caption{Computational status of the principal Catalan maps and bijective
problems considered in the paper.}
\label{tab:landscape}
\end{table}

The entries fall into three groups.
\begin{itemize}
\item The identity and reverse-complement anchor the finite-state end of the
hierarchy.  The identity can copy each input letter during a single rightward
scan.  To realise reverse-complement, the 2DFT of
Example~\ref{ex:rc-two-ways} first reaches the end of the input and then emits
letters while moving back from right to left.  Both maps therefore belong to
$\TDFT/\MSO$, and hence also to $\polyreg$ and $\WRP$.
\item Zeta and $H$ are the positive rank-sort examples.  Both belong to the
particularly simple fragment $\SWR$: they sort atoms by one additive integer
atom rank and use a scan order to break ties.  The lower bounds
$\zeta\notin\polyreg$ and $H\notin\polyreg$ show in each case that this rank
sort cannot be replaced by any polyregular mechanism.  Together with their
$\SWR$ presentations, either map independently witnesses the strict inclusion
$\polyreg\subsetneq\WRP$.  The two witnesses have different combinatorial
roles: zeta transports area and dinv, while the classical bijection $H$
exchanges valleys and double rises.
\item The last two entries are the negative results about what $\WRP$ itself
can realise.  The inverse
$\zeta^{-1}$ is a well-defined, explicitly computable bijection, but one $\WRP$
rank-sort layer cannot realise it
(Corollary~\ref{cor:inverse-zeta-not-wrp}).  The area--dinv entry is different:
the desired bijection is not known to exist, and
Theorem~\ref{thm:wrp-no-swap} proves that no $\WRP$ map can satisfy the two
required statistic identities.  This is the paper's main no-go theorem.
\end{itemize}

\section{Classifying the zeta map}\label{sec:zeta}

This section proves Theorem~\ref{thm:intro-zeta-classification}.  For the upper
bound, $\zeta$ creates two atoms from each up-step, ranked by the heights just
before and after it, and orders equal-rank atoms from left to right.  This is a
one-dimensional rank sort in $\SWR$.  For the lower bound, we prove that
$\zeta$ is not polyregular and hence not realisable by a deterministic 2DFT.
Together, the two results give $\zeta\in\WRP\setminus\polyreg$.

\subsection[Upper bound: zeta orders selected atoms by height]
{Upper bound: $\zeta$ orders selected atoms by height}

We make the positive half of Theorem~\ref{thm:intro-zeta-classification}
explicit in the $\WRP$ formalism.

\begin{theorem}[Zeta in $\SWR$]\label{thm:zeta-wrp}
The classical Haglund zeta map $\zeta\colon\D\to\D$ under the step-word
encoding belongs to the scan-order one-dimensional ranked-regular fragment
$\SWR$, hence to $\WRP$.
\end{theorem}

\begin{proof}
Define one additive rank source: the one-state automaton with weights
$\omega(U)=+1$ and $\omega(D)=-1$.  The prefix rank before position $i$
is then $\rho(i)=h(i-1)$, the height just before position~$i$.

At every input position $i$ labelled $U$ (selected by the $\MSO$ formula
$P_U(x)$ of Section~\ref{sec:mso}), create two selected atoms:
\[
   e_U(i)\colon\text{ label }U,\quad \kappa(e_U(i))=\rho(i),
   \qquad
   e_D(i)\colon\text{ label }D,\quad \kappa(e_D(i))=\rho(i)+1.
\]
Create no atoms at positions labelled $D$.  Use the left-to-right tie-order
on atoms of equal atom rank.  Formally, both copy names $e_U,e_D$ use the zero local
correction in Definition~\ref{def:prefix-additive-rank}; their constants $c_0$
are $0$ for $e_U$ and $1$ for $e_D$.

Now if $i=p_j$ is the $j$-th up-step of $P$, then $\rho(p_j)=a_j$.  Sorting
by atom rank visits the integer values $r=0,1,2,\ldots$ in turn.  At atom rank
$r$, the atoms
present are $e_U(p_j)$ for those $j$ with $a_j=r$ (label $U$), and
$e_D(p_j)$ for those $j$ with $a_j=r-1$ (label $D$).  Among these atoms, the
tie-order scans $j$ left to right.  This is exactly the area-sequence scan
of Definition~\ref{def:zeta}.
\end{proof}

The $\WRP$ formalism also realises an additive level-sorting operation inspired
by the sweep maps of Armstrong, Loehr, and Warrington
\cite{ArmstrongLoehrWarrington2015,ArmstrongLoehrWarrington2016}.  An
\emph{additive level sort} attaches one selected atom to each input step, assigns
it an additive level, and sorts the steps by that level.  The level is a scalar
atom rank, and the scan direction supplies the tie-order, so the resulting
presentation lies in $\SWR$.

\begin{definition}[Additive level sort]\label{def:additive-sweep}
Fix integer step weights $\nu\colon\{U,D\}\to\Z$ and a scan direction.  The
\emph{additive level sort} $\Phi_\nu$ assigns to each step of the input
word $w=w_1\cdots w_N$ the integer \emph{level} $\ell(i)=\sum_{j<i}\nu(w_j)$,
the running weight of the steps before it, and outputs the step labels
$w_1,\ldots,w_N$ listed by increasing level, ties broken by the scan
direction.
\end{definition}

\begin{proposition}[Additive level sorts in $\SWR$]
\label{prop:alw-sweep-swr}
Every additive level sort $\Phi_\nu$ is an $\SWR$ map.  Thus, it
lies in $\WRP$.
\end{proposition}

\begin{proof}
Use the single one-state additive rank source with weights $\omega(U)=\nu(U)$
and $\omega(D)=\nu(D)$ (Definition~\ref{def:rank-source}); by
Definition~\ref{def:prefix-rank}, its prefix rank
before position $i$ is exactly $\ell(i)$.  Select every input position by the
always-true $\MSO$ formula $x=x$, and give the atom at position $i$ the label
$w_i$.  Set its atom rank to the prefix rank $\rho(i)=\ell(i)$.  By
Definition~\ref{def:prefix-additive-rank}, this is a prefix-additive rank
function of dimension $d=1$, with $c_0=0$ and zero local correction.  Take the
tie-order $\chi$ to be the chosen scan direction.  This is an $\SWR$
presentation (Definition~\ref{def:wrp}), and its output is the step labels in
increasing level order with ties broken by scan direction, namely
$\Phi_\nu(w)$.
\end{proof}

We illustrate the choice $\nu(U)=+1$, $\nu(D)=-1$ with the left-to-right tie
convention on a small path.

\begin{example}[Sorting by starting height on $UUDUDD$]\label{ex:height-sweep}
Take $\nu(U)=+1$, $\nu(D)=-1$ and left-to-right ties.  The six steps of
$P=UUDUDD$ have levels (the height before each step)
\[
   \begin{array}{c|cccccc}
     \text{step}  & U & U & D & U & D & D\\\hline
     \text{level} & 0 & 1 & 2 & 1 & 2 & 1
   \end{array}
\]
Reading the steps from left to right within each level gives $U$ at level $0$,
$U,U,D$ at level $1$, and $D,D$ at level $2$.  Concatenating these lists gives
$\Phi_\nu(UUDUDD)=UUUDDD$.

This additive level sort sends $\D_n$ into itself for either tie-order.
Indeed, let $u_h$ be the number of up-steps starting at height $h$.  The number of
down-steps starting at height $h$ is $u_{h-1}$, where $u_{-1}=0$.  Immediately
before the output block at level $h$, the output path is at height $u_{h-1}$,
and that block contains exactly $u_{h-1}$ down-steps.  It therefore cannot take
the output below height zero.  The complete output is balanced, so it is a
Dyck path.

With the left-to-right tie-order, however, this endomap is already
noninjective on $\D_2$, since
$\Phi_\nu(UUDD)=\Phi_\nu(UDUD)=UUDD$.
It is therefore also nonsurjective on $\D_2$.  By contrast, the right-to-left
version, denoted $H$ below, is a bijection on every $\D_n$, as recalled in
Theorem~\ref{thm:narayana-sweep}.
\end{example}

\subsection[Lower bound: zeta is not polyregular]{Lower bound: $\zeta$ is not polyregular}\label{sec:zeta-not-polyregular}

We now prove the negative half of Theorem~\ref{thm:intro-zeta-classification}:
$\zeta$ is not polyregular.  Since every deterministic two-way finite-state
transduction is polyregular, this will also show that $\zeta$ is not realisable
by a deterministic 2DFT, or equivalently by a deterministic $\MSO$ string
transduction.  The argument uses a single two-pyramid probe.

\begin{definition}[Two-pyramid paths]
For $m,n\ge 1$, put $P_{m,n}=U^m D^m U^n D^n$.
\end{definition}

The following criterion isolates the lower-bound argument.  It will be used
for both $\zeta$ and the height sweep $H$.

\begin{proposition}[Two-pyramid criterion]
\label{prop:two-pyramid-criterion}
Let $f\colon\D\to\{U,D\}^*$ satisfy
$|f(P_{m,n})|=O(|P_{m,n}|)$, and let
$K\subseteq\{U,D\}^*$ be a regular language.  If
\[
   \{a^m\#b^n:m,n\ge1,\ f(P_{m,n})\in K\}
\]
is not regular, then no polyregular map realises $f$ on every Dyck path.
\end{proposition}

\begin{proof}
Suppose, for contradiction, that a polyregular map $T$ realises $f$ on
$\D$.  Let $a,b,\#$ be three distinct input symbols, with $\#$ serving as a
separator, and let
\[
   E\colon\{a,b,\#\}^*\longrightarrow\{U,D\}^*
\]
be the elementary 2DFT whose action on the inputs used below is
\[
   E(a^m\#b^n)=U^mD^mU^nD^n=P_{m,n},\qquad m,n\ge1.
\]
Such an $E$ is obtained by four finite-state passes: output the $a$-block as
$U$'s, output it again as $D$'s, and then do the same for the $b$-block.

Let $S=\{a^m\#b^n:m,n\ge1\}$, which is a regular language.  A 2DFT is
polyregular, and polyregular maps are closed under composition
\cite{Bojanczyk2018}.  Restricting $T\circ E$ to $S$ therefore gives a
polyregular map $g$ with domain $S$ and
\[
   g(a^m\#b^n)=f(P_{m,n}),\qquad
   |g(a^m\#b^n)|=O(m+n)=O(|a^m\#b^n|).
\]
The linear-growth collapse
(Theorem~\ref{thm:polyregular-linear-collapse}) makes $g$ a deterministic
$\MSO$ string transduction, equivalently a deterministic 2DFT.  A 2DFT has
regular inverse images of regular languages
\cite{EngelfrietHoogeboom2001}, so $g^{-1}(K)$ is regular.  By construction,
however,
\[
   g^{-1}(K)=\{a^m\#b^n:m,n\ge1,\ f(P_{m,n})\in K\},
\]
contrary to the hypothesis.
\end{proof}

\begin{lemma}[Two-pyramid formula]\label{lem:zeta-two-pyramid}
For $m,n\ge 1$,
\[
\zeta(P_{m,n})=
\begin{cases}
   UU(DU)^{2m-2}DD(UD)^{n-m}, & m\le n,\\[1mm]
   UU(DU)^{2n-1}DD(UD)^{m-n-1}, & m>n.
\end{cases}
\]
\end{lemma}

\begin{proof}
The area sequence of $P_{m,n}$ is
\(
   (0,1,\ldots,m-1,\,0,1,\ldots,n-1).
\)
Run the area-sequence scan of Definition~\ref{def:zeta}.

At scan level $0$, the two entries equal to $0$ produce the initial $UU$.

Suppose $m\le n$.  For each scan level $r=1,\ldots,m-1$, both increasing blocks
contain entries $r-1$ and $r$, and within each block the $r-1$ entry
appears first.  Hence each such scan level contributes $DUDU$, giving
$(DU)^{2m-2}$ in total.  At scan level $m$, the first block contributes only a
$D$ from level $m-1$, while the second block contributes a $D$ from level
$m-1$ and (if $m<n$) a $U$ from level $m$.  The two $D$'s form the central
$DD$.  The remaining scan levels of the second block contribute the tail
$(UD)^{n-m}$.

If $m>n$, the count is instead limited by the \emph{second} block, of length
$n$.  Scan levels $r=1,\ldots,n-1$ contribute $DUDU$ exactly as before, giving
$(DU)^{2n-2}$.  At scan level $n$ the first block still contains the entries $n-1$
and $n$ and contributes $DU$, after which the second block contributes its
final $D$ from the entry $n-1$; the scan levels $n+1,\ldots,m-1$ each contribute
$DU$ from the first block alone, and scan level $m$ contributes its final $D$.  After the
initial $UU(DU)^{2n-2}$, the output therefore continues
$DU\,D\,(DU)^{m-1-n}\,D$, and regrouping $D(DU)^{m-1-n}D=DD(UD)^{m-n-1}$ gives
the form $UU(DU)^{2n-1}DD(UD)^{m-n-1}$, with an empty tail when
$m=n+1$.  Note that this exponent $2n-1$ is \emph{odd}, whereas the
case $m\le n$ gave the \emph{even} $2m-2$; the probe of the next lemma turns on
exactly this even/odd distinction.
\end{proof}

\begin{lemma}[Regular probe]\label{lem:zeta-probe}
Let $R = \{\,UU(DU)^{2q}DD(UD)^s : q,s\ge 0\,\}\subseteq\{U,D\}^*$.
Then $R$ is regular, and for all $m,n\ge 1$,
$\zeta(P_{m,n})\in R$ if and only if $m\le n$.
\end{lemma}

\begin{proof}
Since $(DU)^{2q}=(DUDU)^q$, a DFA can recognise $R$ as follows.  After checking
the initial $UU$, it enters a boundary state.  From this state it reads a $D$.
If the next letter is $U$, it has read the first $DU$ of a $DUDU$ block; it
then requires one more $DU$ and returns to the boundary state.  If the next
letter is $D$, it has read the separator $DD$ and moves to a final phase, where
it reads zero or more $UD$ pairs.  It accepts only when the input ends at the
boundary between complete pairs.  On any unexpected letter, the automaton
enters a nonaccepting state and remains there for the rest of the input, so the
word is rejected.  Thus, $R$ is regular by
Definition~\ref{def:regular-language}.  This description also makes the
factorisation unique: the first $DD$ after the initial $UU$ is the central
separator, since every $D$ in the preceding $(DU)$ block is followed by a
$U$, and it determines the exponent of $DU$.  By
Lemma~\ref{lem:zeta-two-pyramid}, if $m\le n$, then the exponent of $DU$
before the central $DD$ is $2m-2$, which is even; if $m>n$, it is
$2n-1$, which is odd.  Hence $\zeta(P_{m,n})\in R$ if and only if the relevant
exponent is even, which happens if and only if $m\le n$.
\end{proof}

\begin{theorem}[Zeta beyond $\polyreg$]\label{thm:zeta-not-polyregular}
The map $\zeta$ is not realised by any polyregular map
under the step-word encoding.
\end{theorem}

\begin{proof}
Let
\[
   L=\{a^m\#b^n:m,n\ge 1,\ m\le n\}.
\]
This language is not regular by the standard pumping argument; see, for
example, Sipser~\cite[Chapter~1]{Sipser2013}.  Lemma~\ref{lem:zeta-probe}
identifies $L$ with
$\{a^m\#b^n:m,n\ge1,\ \zeta(P_{m,n})\in R\}$.  Since $\zeta$ preserves
length, Proposition~\ref{prop:two-pyramid-criterion}, applied with $f=\zeta$
and $K=R$, gives the result.
\end{proof}

\begin{corollary}[Zeta beyond $\MSO$]\label{cor:zeta-not-regular}
The classical Haglund zeta map $\zeta\colon\D\to\D$ under the ordinary
step-word encoding is not realisable by any deterministic two-way finite-state
transducer.  Equivalently, it is not a deterministic $\MSO$-definable
word-to-word map.
\end{corollary}

\begin{proof}
Every deterministic 2DFT is a deterministic $\MSO$ string transduction by
Theorem~\ref{thm:eh}, and hence an arity-$1$ polyregular map by
Definition~\ref{def:polyregular}.  The result therefore follows from
Theorem~\ref{thm:zeta-not-polyregular}.
\end{proof}

\begin{remark}[Why the collapse is applied to a slice, not to $T$ itself]
The linear-growth collapse cannot be applied to $T$ directly.  Under the
realisation convention (Definition~\ref{def:relative}), a polyregular $T$
realising $\zeta$ is unconstrained on non-Dyck inputs, where it may grow
superlinearly, so the hypothesis of
Theorem~\ref{thm:polyregular-linear-collapse} need not hold on
$\mathrm{dom}(T)$.  Proposition~\ref{prop:two-pyramid-criterion} sidesteps
this by collapsing the linear-growth restriction $g$, whose domain is
$\{a^m\#b^n:m,n\ge 1\}$; this language is regular.  On this domain, the
linear-growth hypothesis genuinely holds, so the conclusion is the strong
one, that \emph{no} polyregular map realises $\zeta$.
\end{remark}

\begin{remark}[First separation in the hierarchy]
Together, Theorem~\ref{thm:zeta-wrp} and
Theorem~\ref{thm:zeta-not-polyregular} give the first strict inclusion of the
paper: $\WRP\supsetneq\polyreg$.  The witness, $\zeta$, has \emph{linear}
output growth; the separation is not about output growth at all but about the
ability to sort by an unbounded integer rank.  Section~\ref{sec:narayana-sweep}
will give $H$ as a second witness with a different statistic exchange.
\end{remark}

\section{The height sweep and the Narayana symmetry}
\label{sec:narayana-sweep}

This section proves Theorem~\ref{thm:intro-narayana-sweep}.  As a second
classification example for $\WRP$, we consider a classical sweep map of
Armstrong, Loehr, and Warrington~\cite{ArmstrongLoehrWarrington2015,
ArmstrongLoehrWarrington2016}: the map that lists the steps of a Dyck path by
increasing starting height, breaking ties from right to left.  We denote it by
$H$.  General sweep-map bijectivity implies that $H$ is a bijection of each
$\D_n$~\cite{ThomasWilliams2018}; the corresponding tree and zeta descriptions
show that it exchanges valleys and double rises, defined
next~\cite{CeballosFangMuhle2020,Fang2024,SulzgruberThiel2018}.  What is new is
its computational status.
For the upper bound, $H$ is an additive level sort, hence an $\SWR$ map by
Proposition~\ref{prop:alw-sweep-swr}.  For the lower bound, we run the
two-pyramid probe of Section~\ref{sec:zeta-not-polyregular} a second time,
with the roles of the two parameters exchanged.  Together, the two results
give $H\in\SWR\setminus\polyreg$.

For a Dyck word $w=w_1\cdots w_{2n}$, let
$\val(w)=|\{i:1\le i<2n,\ w_iw_{i+1}=DU\}|$ be its number of
\emph{valleys} and let
$\dr(w)=|\{i:1\le i<2n,\ w_iw_{i+1}=UU\}|$ be its number of
\emph{double rises}.  Each counts occurrences of a fixed length-two factor, so
each is a \emph{local} statistic, in contrast to the global $\area$ and $\dinv$.
For $n\ge 1$, writing $w$ in its maximal runs
$U^{a_1}D^{b_1}\cdots U^{a_k}D^{b_k}$ with $k=\pk(w)$ peaks (the number of $UD$
factors) gives the elementary identities
\[
   \val(w)=\pk(w)-1,\qquad \dr(w)=n-\pk(w),\qquad
   \text{so}\quad \val(w)+\dr(w)=n-1.
\]
These identities are used in the proof in
Appendix~\ref{app:narayana-sweep-proof}.  They also show that $\val$ and $\dr$
determine each other on $\D_n$, so the symmetry below says exactly that the
distribution of $\val$ on $\D_n$ is symmetric about $(n-1)/2$.
Their joint distribution is the Narayana polynomial
\[
   \Nar_n(q,t)=\sum_{P\in\D_n}q^{\val(P)}t^{\dr(P)}
   =\sum_{k=1}^{n}\frac1n\binom nk\binom n{k-1}\,q^{k-1}t^{\,n-k},
\]
whose symmetry $\Nar_n(q,t)=\Nar_n(t,q)$ is the Narayana analogue of the
$q,t$-Catalan symmetry; for $n=0$ the empty path gives $\Nar_0(q,t)=1$.
Deutsch proved this symmetry bijectively, before the sweep maps, by a
recursive involution on Dyck paths~\cite{Deutsch1999}.  The map $H$ is a
different known realisation of the same symmetry, obtained by sorting steps by
their starting heights.

Let $H=\Phi_\nu$ be the additive level sort
(Definition~\ref{def:additive-sweep}) for the height rule $\nu(U)=+1$,
$\nu(D)=-1$, taken with the \emph{right-to-left} tie-order: $H(P)$ lists the
steps of $P$ by increasing starting height (the height of $P$ just before the
step), breaks ties between equal-height steps by decreasing input position, and
reads off their $U/D$ labels.  Thus, $H$ is the right-to-left version of the
construction in Example~\ref{ex:height-sweep}.

\begin{example}[The map $H$ on $UUUDDD$]\label{ex:narayana-sweep}
The single mountain $P=UUUDDD$ has $\val(P)=0$ and $\dr(P)=2$.  Its six steps,
with the height of $P$ just before each, are
\[
   \begin{array}{c|cccccc}
     \text{position}        & 1 & 2 & 3 & 4 & 5 & 6\\\hline
     \text{step}            & U & U & U & D & D & D\\
     \text{starting height} & 0 & 1 & 2 & 3 & 2 & 1
   \end{array}
\]
Listing the steps by increasing starting height, and within each height by
decreasing position (the right-to-left tie-order), gives
\[
   \begin{array}{r|l}
     \text{height }0: & U\ (1)\\
     \text{height }1: & D\ (6),\ U\ (2)\\
     \text{height }2: & D\ (5),\ U\ (3)\\
     \text{height }3: & D\ (4)
   \end{array}
\]
Reading off the labels gives $H(P)=U\cdot DU\cdot DU\cdot D=UDUDUD$.  The mountain, with two double
rises and no valleys, is sent to the zigzag $UDUDUD$, with $\val=2$ and $\dr=0$:
the two statistics are exchanged.  With the left-to-right convention of
Example~\ref{ex:height-sweep}, the same additive level sort would give
$UUDUDD$, which has $\val=\dr=1$.  It is the right-to-left reading that makes
$H$ exchange them.
\end{example}

Theorem~\ref{thm:narayana-sweep} below is not new.  After a translation of
conventions, it follows from the invertibility of the sweep maps together with
the statistic transport of the zeta map, as reviewed in
Section~\ref{sec:related-work}~\cite{ArmstrongLoehrWarrington2015,
ThomasWilliams2018,CeballosFangMuhle2020,SulzgruberThiel2018}.  We restate it
in the conventions of this paper and prove it directly in
Appendix~\ref{app:narayana-sweep-proof}, so that nothing here rests on an
unstated translation.  Combined with Proposition~\ref{prop:alw-sweep-swr}, it
says that a single $\SWR$ map realises a classical Catalan-type symmetry.

\begin{theorem}[The height sweep realises the Narayana symmetry]
\label{thm:narayana-sweep}
For every $n$, the map $H$ restricts to a bijection
$H\colon\D_n\to\D_n$ satisfying
\[
   \dr(H(P))=\val(P),\qquad \val(H(P))=\dr(P)
   \qquad (P\in\D_n).
\]
Consequently, $H$ realises the Narayana symmetry
$\Nar_n(q,t)=\Nar_n(t,q)$.
\end{theorem}

The two-pyramid criterion also gives a lower bound for $H$.  Recall the
two-pyramid paths $P_{m,n}=U^mD^mU^nD^n$ and the criterion of
Proposition~\ref{prop:two-pyramid-criterion}.  The same family and the same
regular probe $R$ that separated $\zeta$ from $\polyreg$ serve for $H$, with
the roles of the two parameters exchanged: the closed form below mirrors the
one in Lemma~\ref{lem:zeta-two-pyramid} with $m$ and $n$ interchanged, so the
probe detects $n\le m$ where Lemma~\ref{lem:zeta-probe} detected $m\le n$.

\begin{lemma}[Two-pyramid formula for the height sweep]
\label{lem:H-two-pyramid}
For $m,n\ge1$,
\[
H(P_{m,n})=
\begin{cases}
   UU(DU)^{2n-2}DD(UD)^{m-n}, & n\le m,\\[1mm]
   UU(DU)^{2m-1}DD(UD)^{n-m-1}, & n>m.
\end{cases}
\]
\end{lemma}

\begin{proof}
Index the $2m+2n$ steps of $P_{m,n}$ by their positions.  For the weights
$\nu(U)=+1$ and $\nu(D)=-1$, the level of a step is its starting height, and
we use the two terms interchangeably below.  For a level $h$, write $U_1$ and
$D_1$ for the steps of the runs $U^m$ and $D^m$ that start at height $h$, and
$U_2$ and $D_2$ for the corresponding steps of $U^n$ and $D^n$; each exists
exactly for the range of $h$ shown.  Reading off the starting heights of the
four runs gives the following table.

\begin{center}
\small
\renewcommand{\arraystretch}{1.25}
\begin{tabular}{@{}llll@{}}
\toprule
run & steps & starting heights & position of the step at height $h$\\
\midrule
$U^m$ & positions $1,\ldots,m$ &
   $0,1,\ldots,m-1$ & $U_1$ at $h+1$ \hfill($0\le h\le m-1$)\\
$D^m$ & positions $m+1,\ldots,2m$ &
   $m,m-1,\ldots,1$ & $D_1$ at $2m-h+1$ \hfill($1\le h\le m$)\\
$U^n$ & positions $2m+1,\ldots,2m+n$ &
   $0,1,\ldots,n-1$ & $U_2$ at $2m+h+1$ \hfill($0\le h\le n-1$)\\
$D^n$ & positions $2m+n+1,\ldots,2m+2n$ &
   $n,n-1,\ldots,1$ & $D_2$ at $2m+2n-h+1$ \hfill($1\le h\le n$)\\
\bottomrule
\end{tabular}
\end{center}

At level $0$, only $U_1$ at position $1$ and $U_2$ at position $2m+1$ start
at height $0$.  Decreasing position lists $U_2$ and then $U_1$, so this level
contributes $UU$.

For every level $h\ge1$, each run contributes at most one step.  Since the four
runs occur in $P_{m,n}$ in the order $U_1,D_1,U_2,D_2$, the right-to-left
tie-order lists the steps that exist in the order obtained by deleting absent
members from $D_2,U_2,D_1,U_1$.

Suppose first that $n\le m$.  The levels $1\le h\le n-1$ each contain all
four steps and together contribute $(DU)^{2n-2}$.  If $n<m$, level $n$
contributes $DDU$, the levels $n+1\le h\le m-1$ each contribute $DU$, and
level $m$ contributes $D$, so that the output is
\[
   UU(DU)^{2n-2}DD\,U(DU)^{m-1-n}D
   =UU(DU)^{2n-2}DD(UD)^{m-n}.
\]
If $n=m$, level $n$ contributes $DD$ and there is no further level, which is
the same formula with an empty tail.

Now suppose that $n>m$.  The levels $1\le h\le m-1$ together contribute
$(DU)^{2m-2}$.  Level $m$ contributes $DUD$, the levels
$m+1\le h\le n-1$ each contribute $DU$, and level $n$ contributes $D$.
Hence
\[
   H(P_{m,n})
   =UU(DU)^{2m-2}DUD(DU)^{n-1-m}D
   =UU(DU)^{2m-1}DD(UD)^{n-m-1},
\]
where the tail is empty when $n=m+1$.  This settles both cases.
\end{proof}

\begin{lemma}[Regular probe for the height sweep]
\label{lem:H-probe}
For all $m,n\ge1$, $H(P_{m,n})\in R$ if and only if $n\le m$, where $R$ is
the regular language of Lemma~\ref{lem:zeta-probe}.
\end{lemma}

\begin{proof}
By Lemma~\ref{lem:H-two-pyramid}, the exponent of $DU$ before the central
$DD$ is the even number $2n-2$ if $n\le m$ and the odd number $2m-1$ if
$n>m$.  Membership in $R$ is therefore equivalent to $n\le m$.
\end{proof}

\begin{theorem}[The height sweep is beyond $\polyreg$]
\label{thm:H-not-polyregular}
The map $H$ is not realised by any polyregular map under the step-word
encoding.  Together with Proposition~\ref{prop:alw-sweep-swr}, this gives
$H\in\SWR\setminus\polyreg$; in particular, $H$ is not realisable by any
deterministic 2DFT.
\end{theorem}

\begin{proof}
Let
\[
   L'=\{a^m\#b^n:m,n\ge1,\ n\le m\}.
\]
This language is not regular by the standard pumping argument: if $p$ were a
pumping length, pumping down a nonempty initial block of $a$'s in $a^p\#b^p$
would produce $a^{p-t}\#b^p$ for some $t>0$, which violates $n\le m$.
Lemma~\ref{lem:H-probe} identifies $L'$ with
$\{a^m\#b^n:m,n\ge1,\ H(P_{m,n})\in R\}$.  Since $H$ preserves length,
Proposition~\ref{prop:two-pyramid-criterion}, applied with $f=H$ and $K=R$,
shows that no polyregular map realises $H$ on every Dyck path.  The membership
$H\in\SWR$ follows from Proposition~\ref{prop:alw-sweep-swr}, and the final
claim follows because every deterministic 2DFT is polyregular.
\end{proof}

Proposition~\ref{prop:alw-sweep-swr} and
Theorems~\ref{thm:narayana-sweep} and~\ref{thm:H-not-polyregular} together
establish the introduction-level statement,
Theorem~\ref{thm:intro-narayana-sweep}.

The contrast with Theorem~\ref{thm:intro-no-swap} is the point of this
section.  A single $\WRP$ rank sort realises the Narayana symmetry, whereas
Section~\ref{sec:main} shows that no $\WRP$ map realises the $q,t$-Catalan
symmetry by an area--dinv swap.  The obstruction proved there is therefore not
the bare weakness of the model.

\section[Regular-slice semilinearity for WRP]{Regular-slice semilinearity for $\WRP$}\label{sec:slice-semilinearity}

Recall the no-swap theorem (Theorem~\ref{thm:intro-no-swap}, proved in
Section~\ref{sec:main}): no $\WRP$ map realises a semilength-preserving map on
Dyck paths that exchanges $\area$ and $\dinv$.  A bijection of this kind would settle
the $q,t$-Catalan symmetry directly.  The proof begins by restricting any
hypothetical swap to the explicit infinite family of Dyck paths
\[
   W_n=U(UD)^nD\qquad (n\ge 1),
\]
where $W_n$ repeats the block $UD$ exactly $n$ times between an
opening $U$ and a closing $D$.  The indexed family $(W_n)_{n\ge1}$ is a
particular \emph{regular slice}: a family
of the form $w_n=uv^nz$ for fixed words $u,v,z$, with only $n$ varying;
indeed, $\{W_n:n\ge 1\}$ is a regular language.  Its repeated-block form makes
finite-state behaviour easy to expose.  For any fixed finite automaton, reading
one more copy of $UD$ applies the same transition map to its current state.
Because the automaton has only finitely many states, repeated application of
this map eventually cycles.  Hence the state reached after reading $W_n$ is
eventually periodic in $n$.  This elementary observation is the starting point
of the regular-slice analysis.

For a fixed map $T$, let $a_n$ be the number of $U$-steps in
$T(W_n)$ before its first $D$, and let $b_n$ be the number of $U$-steps after
that first $D$.  If the output has no $D$, we let $a_n$ count all of its
$U$-steps and set $b_n=0$.  As $n$ varies, the pairs $(a_n,b_n)$ form a set of
integer points in the plane.  The goal of this section
(Theorem~\ref{thm:wrp-slice-semilinearity}) is to show that this set is
\emph{semilinear} whenever $T\in\WRP$ has output length $O(n)$ on this family.
Informally, a semilinear set is built from finitely many base points by
repeatedly adding fixed integer step vectors, so its geometry is governed by
finitely many linear patterns.  By contrast, a map exchanging $\area$ and
$\dinv$ would force the pairs $(a_n,b_n)$ to have a quadratic lower boundary.
This incompatibility is the obstruction used in the next section.

To prove the semilinearity statement, we use the logical language of
Presburger arithmetic.  Its formulas describe integer tuples using addition,
order comparisons, and congruence conditions, together with the usual logical
connectives and quantifiers.  We call a set of integer tuples
\emph{Presburger-definable} when membership in it can be expressed by such a
formula, and likewise call a condition on integer variables
\emph{Presburger-definable} when it can.  The Presburger-definable sets are
precisely the semilinear sets.  We
use two technical ingredients.  First, on a regular slice, the
conditions governing selection, labelling, atom-rank comparisons, and
tie-orders in a $\WRP$ presentation can be expressed by Presburger formulas in
$n$ and the relevant \emph{repetition indices}, which record the copies of the
repeated block containing the positions under consideration.  Second, if such a
formula describes finitely many solutions for each parameter tuple and their
number grows at most linearly in the parameters, then the set of
parameter/count pairs is semilinear.  The same conclusion holds when finitely
many such counts are recorded together.  The argument combines these
ingredients as follows.

\vspace{0.5em}
\noindent\emph{The argument in three steps.}
\begin{enumerate}[label=(\arabic*)]
\item \emph{Describe the output by Presburger formulas.}  The first ingredient
expresses which atoms are selected, whether they are labelled $U$ or $D$, and
how they are ordered in the output by Presburger formulas whose free variables
are $n$ and the relevant repetition indices.
\item \emph{Express $a_n$ and $b_n$ as solution counts.}  After treating
separately the case in which the output has no $D$, each quantity counts the
selected $U$-atoms that satisfy a Presburger-definable condition placing them
before or after the first selected $D$-atom.  Both counts are $O(n)$ because
the total output length is $O(n)$.
\item \emph{Apply bounded counting and project.}  The second ingredient shows
that the set $\{(n,a_n,b_n):n\ge 1\}$ is semilinear.  Since semilinear sets are
closed under projection, forgetting the coordinate $n$ leaves the desired
semilinear set of first-ascent pairs.
\end{enumerate}

\subsection{First ingredient: Presburger descriptions on a regular slice}

We now prove the first ingredient from the overview.  Once the input is
restricted to a regular slice, each position can be specified by finite data
recording its region and offset, together with a repetition index when it lies
in the repeated part.  With positions encoded in this way, the data determining
the $\WRP$ output can be described by Presburger formulas: which potential atoms
are selected, how the selected atoms are labelled, and their relative output
order.  The free variables of these formulas are $n$ and the relevant repetition
indices.  Thus, a fixed finite collection of Presburger formulas describes the
outputs across the entire slice.

To state it, fix words $u,v,z$ and consider the regular slice
$(w_n)_{n\ge0}$ given by
\[
   w_n=uv^nz\qquad(n\ge 0),
\]
in which only the number of copies of $v$ varies.  The family
$(W_n)_{n\ge1}$ is obtained by restricting the case $u=U$, $v=UD$, $z=D$ to
$n\ge1$.  Number the letters of $w_n$
from $1$ to $|u|+n|v|+|z|$, and encode a position $i$ by a triple
\[
   (\tau,\ s,\ j)
\]
read as follows:
\begin{itemize}
\item $\tau\in\{\mathsf u,\mathsf v,\mathsf z\}$ records which region $i$ lies
  in: the prefix $u$, one of the $n$ copies of $v$, or the suffix~$z$;
\item $s$ is the offset of $i$ within that region, so that $1\le s\le|u|$,
  $1\le s\le|v|$, or $1\le s\le|z|$ according to whether $\tau$ is
  $\mathsf u,\mathsf v$, or $\mathsf z$;
\item when $\tau=\mathsf v$, $j$ is the \emph{repetition index}: it records
  which copy of $v$ contains $i$, so $1\le j\le n$.  When
  $\tau\in\{\mathsf u,\mathsf z\}$, we set $j=0$, since it plays no role there.
\end{itemize}
The position $i$ is recovered from its triple by
\[
   i=
   \begin{cases}
     s, & \tau=\mathsf u,\\[2pt]
     |u|+(j-1)\,|v|+s, & \tau=\mathsf v,\\[2pt]
     |u|+n\,|v|+s, & \tau=\mathsf z,
   \end{cases}
\]
and this is a bijection between valid triples and positions of $w_n$.
Because $u,v,z$ are fixed words, the tag $\tau$ and the offset $s$ take only
finitely many values, while the repetition index $j$ may grow with $n$ and
occurs only for positions inside the repeated part.

For the family $(W_n)_{n\ge1}$, using $u=U$, $v=UD$, and $z=D$, this
specialises as follows: the opening $U$ is $(\mathsf u,1,0)$; inside the $j$-th
block $UD$, the up-step is $(\mathsf v,1,j)$ and the down-step is
$(\mathsf v,2,j)$; and the closing $D$ is $(\mathsf z,1,0)$.

\begin{example}[Coordinates and height for $W_n$]
\label{ex:slice-coordinates}
Take the height rank source of Example~\ref{ex:height-rank}.  On
$W_n=U(UD)^nD$, its prefix ranks (Definition~\ref{def:prefix-rank}) before the
four kinds of positions are
\[
   \rho(\mathsf u,1,0)=0,\qquad
   \rho(\mathsf v,1,j)=1,\qquad
   \rho(\mathsf v,2,j)=2,\qquad
   \rho(\mathsf z,1,0)=1.
\]
The two values inside the repeated part do not depend on $j$ because each
block $UD$ has total height change zero.  Thus, for this particular rank source
and this particular family, the tag and offset determine the prefix rank even
though the repetition index may grow with $n$.  This constancy is special to
the height source on $W_n$; prefix ranks need not generally be constant along
a regular slice.
\end{example}

A potential atom of the presentation consists of a copy name $c\in C$ (from
the finite set in Definition~\ref{def:polyregular}) together with a tuple of
$k_c$ positions $\bar\imath=(i_1,\ldots,i_{k_c})$; we call these positions the
\emph{coordinates} of the atom.  Thus, the copy name determines the number of
coordinates.  For each coordinate $r\in\{1,\ldots,k_c\}$, the triple
$(\tau_r,s_r,j_r)$ names its position.  The finite part
$(c,\tau_1,s_1,\ldots,\tau_{k_c},s_{k_c})$ ranges over a fixed finite set,
while the repetition indices
$(j_1,\ldots,j_{k_c})$ of the coordinates may grow with $n$.  (For the arity-one
fragments $\RR$ and $\SWR$, a potential atom sits at a single position, so it
has a single coordinate, with repetition index $j$.)

We recall the standard generator description of semilinear sets.  A
\emph{linear set} in $\N^k$ is a set
\[
   \{\beta+n_1 s_1+\cdots+n_m s_m:\ n_1,\ldots,n_m\in\N\},
\]
given by a base point $\beta\in\N^k$ and finitely many fixed step vectors
$s_1,\ldots,s_m\in\N^k$, each of which may be used any number of times; a
\emph{semilinear set} is a finite union of linear sets.  This is the ``base
points and constant integer steps'' form glossed in the introduction.
The Ginsburg--Spanier theorem~\cite{GinsburgSpanier1966} identifies these sets
with the Presburger-definable sets introduced at the beginning of this section.
Thus, the generator description above and the logical description by Presburger
formulas define the same class.  In what follows, we use the term
``semilinear'' whether a set is presented by generators or by a Presburger
formula.

The logical description gives the closure rules used below.  Semilinear sets
are closed under finite unions, intersections, complements, and projections.
Here, \emph{projection} means forgetting coordinates: if
$S\subseteq\N^{k+1}$ is semilinear, then so is
$\{x\in\N^k:(x,y)\in S\text{ for some }y\in\N\}$.
Finite unions, intersections, and complements correspond to disjunction,
conjunction, and negation in Presburger formulas, while projection corresponds
to existential quantification.  Universal quantification can be expressed
using negation and existential quantification.  We use these closure properties
below without further comment.  For example, ``no selected atom precedes a
given one'' is the negation of an existential condition, and retaining only two
chosen coordinates of each tuple is a projection; both operations preserve
semilinearity.

\paragraph{Marked words and eventual periodicity.}
The first two lemmas in this subsection analyse separately an $\MSO$ formula
and an additive rank source on a regular slice $(w_n)_{n\ge0}$ of the form
$w_n=uv^nz$.  The third
combines their conclusions for a fixed $\WRP$ presentation.  Both analyses use
the same mechanism: reading another copy of the fixed block $v$ updates a state
from a fixed finite set.  For a DFA
$M=(Q,q_0,\delta,F)$ (Definition~\ref{def:dfa}), let $\delta_v(q)$ be the state
reached after reading $v$ from $q$.  Reading $v$ repeatedly applies the same map
$\delta_v\colon Q\to Q$.  Since the set $Q^Q$ of maps from $Q$ to itself is
finite, the iterates $\delta_v^0,\delta_v^1,\delta_v^2,\ldots$ eventually
repeat.\footnote{For a function $f\colon X\to X$, we write $f^r$ for the
composition of $r$ copies of $f$ when $r\ge 1$, and set
$f^0=\mathrm{id}_X$.}  After a bounded initial segment, they cycle with a fixed
period.  An additive
rank source (Definition~\ref{def:rank-source}) has the same finite-state
transition structure; its accumulated weights will be handled separately.
This eventual periodicity is the finite-state mechanism behind the Presburger
descriptions below.

For an $\MSO$ formula $\varphi(x_1,\ldots,x_k)$ with free position variables,
we first encode an assignment to those variables in the input word.  Given a
word $w$ and chosen positions $i_1,\ldots,i_k$, attach $k$ bits to each letter,
with the $\ell$-th bit equal to $1$ exactly at position $i_\ell$.  The resulting
word over $\Sigma\times\{0,1\}^k$ records both $w$ and the chosen positions.  The
marked words for which $w\models\varphi(i_1,\ldots,i_k)$ form an
$\MSO$-definable language: the marker bits identify the positions assigned to
the free variables.  By the B\"uchi--Elgot--Trakhtenbrot theorem
(Section~\ref{sec:mso}), this language is regular and therefore recognised by a
DFA.  This is the form in which we use $\MSO$ below.

\begin{lemma}[MSO conditions on a regular slice]
\label{lem:one-loop-finite-state}
Fix a regular slice $(w_n)_{n\ge0}$ with $w_n=uv^nz$ and an $\MSO$ formula
$\varphi(x_1,\ldots,x_k)$.  For each variable $x_\ell$, fix its region
$\tau_\ell\in\{\mathsf u,\mathsf v,\mathsf z\}$: the prefix $u$, a copy of the
repeated block $v$, or the suffix $z$.  Also fix an offset $s_\ell$ valid in
that region, as in the position encoding preceding
Example~\ref{ex:slice-coordinates}.  If $\tau_\ell=\mathsf v$, let
$j_\ell\in\{1,\ldots,n\}$ be the repetition index; otherwise set $j_\ell=0$.
Let $i_\ell$ be the position encoded by $(\tau_\ell,s_\ell,j_\ell)$.  Then the
set
\[
   \{(n,j_1,\ldots,j_k):
      w_n\models\varphi(i_1,\ldots,i_k)\}
\]
is Presburger-definable.
\end{lemma}

\begin{proof}
We begin with the key idea of the proof; the remainder of the argument makes
this outline precise.  Inside the repeated part $v^n$, the chosen positions lie
in at most $k$ copies of $v$.  We first translate $\varphi$ into a DFA on marked
words.  We then partition the parameter tuples into cells according to which
repetition indices $j_\ell$ are equal and how the distinct indices are ordered.
On each cell, we therefore know which variables mark the same copy of $v$ and
the order in which the marked copies occur.  These copies are separated by
stretches of unmarked copies.  Eventual periodicity implies that the transition
map induced by such a stretch falls into one of finitely many cases as its
length varies, with each case described by a Presburger formula.  Refining
the cells by these cases fixes the entire sequence of transition maps applied
by the DFA.  Thus, every tuple in a refined cell leads to the same final state.
The accepting tuples are the union of the refined cells whose final state is
accepting and hence form a Presburger-definable set.

\vspace{0.5em}
\noindent\emph{The marked automaton and its transition maps.}
By the marking construction above and the
B\"uchi--Elgot--Trakhtenbrot theorem~\cite{Thomas1997}, fix a DFA
$M=(Q,q_0,\delta,F)$ over the marked alphabet $\Sigma\times\{0,1\}^k$
that accepts the marked encoding of $(w,i_1,\ldots,i_k)$ exactly when
$w\models\varphi(i_1,\ldots,i_k)$.  For a marked word $y$, define
$\delta_y\colon Q\to Q$ by $\delta_y(q)=\delta^*(q,y)$ in the notation of
Definition~\ref{def:dfa}.  Thus, $\delta_y(q)$ is the state $M$ reaches after
reading $y$ from $q$.  The map $\delta_y$ belongs to the finite set $Q^Q$ of
maps from $Q$ to itself, and reading a concatenation composes the corresponding
maps:
$\delta_{yy'}=\delta_{y'}\circ\delta_y$.  On this slice, the tuple
$(n,j_1,\ldots,j_k)$ determines all marker bits: the $\ell$-th bit is set to
$1$ at position $i_\ell$ and to $0$ elsewhere.  Write $\widetilde w_n$ for this
marked word.  Then,
$w_n\models\varphi(i_1,\ldots,i_k)$ holds if and only if
$\delta_{\widetilde w_n}(q_0)\in F$.

\vspace{0.5em}
\noindent\emph{The block map is eventually periodic.}
Let $g\colon Q\to Q$ be the map induced by reading one copy of $v$ with all $k$
marker bits equal to $0$.  Reading $L$ consecutive unmarked copies applies
$g^{L}$.  The powers $g^0,g^1,g^2,\ldots$ all lie in the finite set $Q^Q$, so
there are integers $k_0\ge0$ and $p\ge1$ for which
$g^{k_0}=g^{k_0+p}$.  Composing this equality with further powers of $g$ gives
\[
   g^{\,m+p}=g^{\,m}\quad(\text{as maps }Q\to Q)\qquad(m\ge k_0).
\]
Define the \emph{reduced exponent}
\[
   \widehat L=
   \begin{cases}
     L, & 0\le L<k_0,\\
     k_0+\bigl((L-k_0)\bmod p\bigr), & L\ge k_0,
   \end{cases}
\]
so that $g^{L}=g^{\widehat L}$ for every $L\ge 0$.  The value $\widehat L$
ranges over the finite set $\{0,1,\ldots,k_0+p-1\}$, and for each $e$ in it,
the condition $\widehat L=e$ is equivalent to ``$L=e$'' (if $e<k_0$) or to
``$L\ge k_0$ and $L\equiv e\pmod p$'' (if $e\ge k_0$).  In either case, it is
a Presburger-definable condition on $L$.  Thus, reading a stretch of $L$
unmarked copies applies the map $g^{\widehat L}$, which depends on $L$ only
through the finite-valued parameter $\widehat L$.

\vspace{0.5em}
\noindent\emph{The first partition: marked copies.}
We now carry out the equality-and-order partition from the proof outline.  The
regions and offsets of the variables are fixed in the statement, so the only
remaining integer data is $n$ together with the repetition indices $j_\ell$ of
the variables lying in $v$.  The valid values satisfy $n\ge0$ and
$1\le j_\ell\le n$ for these variables.  These restrictions are
Presburger-definable conditions.  Partition the valid tuples according to
the relative order and equalities among the repetition indices belonging to
variables in $v$.  Each comparison $j_a<j_b$, $j_a=j_b$, or $j_a>j_b$ is
expressed by a linear equality or inequality, so this produces finitely many
Presburger-definable cells.  Fix one nonempty cell
$C$.  On $C$, variables whose repetition indices are equal mark the same copy
of $v$.  These equality classes and their order are fixed.  For each class, we
know which variables lie in the corresponding copy of $v$, and the fixed
offsets determine the position assigned to each variable within that copy.
Only the numerical repetition indices of the marked copies may vary.  It
suffices to show that the accepting tuples in $C$ form a Presburger-definable
set, because the full parameter space is the finite union of these cells.

\vspace{0.5em}
\noindent\emph{Factoring the run on $C$.}
Choose one variable from each equality class on $C$, and list the repetition
indices of these representatives in their fixed order as
$j_{(1)}<\cdots<j_{(r)}$.  By the definition of $C$,
$1\le j_{(1)}<\cdots<j_{(r)}\le n$.  Suppose that $r\ge1$, and define the
$r+1$ unmarked stretch lengths by
\[
   L_0=j_{(1)}-1,\qquad
   L_t=j_{(t+1)}-j_{(t)}-1\ (1\le t<r),\qquad
   L_r=n-j_{(r)}.
\]
These are nonnegative affine expressions in the parameters on $C$.  Cutting the
word at the marked copies gives
\[
   \widetilde w_n=
   \widetilde u\;v^{L_0}\widetilde v_1\;v^{L_1}\cdots
   \widetilde v_r\;v^{L_r}\widetilde z,
\]
where $v^{L}$ denotes the word formed from $L$ consecutive unmarked copies,
$\widetilde u$ and $\widetilde z$ are the prefix and suffix carrying their
markers at the fixed offsets, and $\widetilde v_t$ is the marked copy at index
$j_{(t)}$.  The preceding partition therefore fixes the marked word
$\widetilde v_t$ for every $t$.  Set $M_u=\delta_{\widetilde u}$,
$M_t=\delta_{\widetilde v_t}$, and $M_z=\delta_{\widetilde z}$.  The
composition law gives a composite whose factors, in the order
in which they act, are
\[
   M_u,\ g^{L_0},\ M_1,\ g^{L_1},\ldots,\ M_r,\ g^{L_r},\ M_z.
\]
When $r=0$, set $L_0=n$; there are no marked copies of $v$, and the three
factors are $M_u,g^{L_0},M_z$.  In either case, the maps
$M_u,M_1,\ldots,M_r,M_z$ are fixed on $C$, and the only parameter dependence is
through the powers $g^{L_0},\ldots,g^{L_r}$.

\vspace{0.5em}
\noindent\emph{The second partition: unmarked stretches.}
Let $E=\{0,\ldots,k_0+p-1\}$, the set of possible reduced exponents.  For each
tuple $\mathbf e=(e_0,\ldots,e_r)\in E^{r+1}$, define
\[
   C_{\mathbf e}
   =\{(n,j_1,\ldots,j_k)\in C:
       \widehat L_t=e_t\text{ for }0\le t\le r\}.
\]
The nonempty sets $C_{\mathbf e}$ form a finite partition of $C$; this is the
refinement used below.  We verify that every $C_{\mathbf e}$ is
Presburger-definable.  The formulas above express every $L_t$
as a constant plus an integer linear combination of
$n,j_{(1)},\ldots,j_{(r)}$: for $r\ge1$, they are
$L_0=j_{(1)}-1$, $L_t=j_{(t+1)}-j_{(t)}-1$ for $1\le t<r$, and
$L_r=n-j_{(r)}$, while for $r=0$, we have $L_0=n$.  For fixed $t$ and
$e\in E$, the definition of the reduced exponent says that
$\widehat L_t=e$ is equivalent to $L_t=e$ when $e<k_0$, and to the conjunction
$L_t\ge k_0$ and $L_t\equiv e\pmod p$ when $e\ge k_0$.
Substituting the affine expression for $L_t$ gives linear equalities,
inequalities, and congruences in the parameters.  Here, these differences
denote nonnegative gaps, as guaranteed by the inequalities defining
$C$.  In the corresponding Presburger formulas over $\N$, each difference is
replaced by an equivalent relation using addition, so no convention for
subtraction on $\N$ is needed.  These conditions are therefore expressible by
Presburger formulas, and each
$C_{\mathbf e}$ is Presburger-definable.  On each nonempty $C_{\mathbf e}$,
every factor
$g^{L_t}=g^{\widehat L_t}$ is a fixed map.  The composite
described above is therefore a fixed map $\sigma\in Q^Q$, so
$\delta_{\widetilde w_n}(q_0)=\sigma(q_0)$ is constant on $C_{\mathbf e}$.
Thus, either every tuple in $C_{\mathbf e}$ is accepted or none is.  The
accepting tuples in $C$ are therefore the union of those sets $C_{\mathbf e}$
whose common final state lies in $F$, a finite union of Presburger-definable
sets.  Taking the finite union over all cells $C$ proves that the set
$\{(n,j_1,\ldots,j_k):w_n\models\varphi(i_1,\ldots,i_k)\}$ is
Presburger-definable.
\end{proof}

\begin{lemma}[Prefix-additive ranks are piecewise affine on a regular slice]
\label{lem:one-loop-rank-affine}
Fix a regular slice $(w_n)_{n\ge0}$ with $w_n=uv^nz$ and a $d$-dimensional prefix-additive rank
function $\kappa$ on $k$-tuples of positions.  For each coordinate
$\ell\in\{1,\ldots,k\}$, fix a region
$\tau_\ell\in\{\mathsf u,\mathsf v,\mathsf z\}$ and an offset $s_\ell$ valid in
that region.  If $\tau_\ell=\mathsf v$, let
$j_\ell\in\{1,\ldots,n\}$ be the repetition index; otherwise set $j_\ell=0$.
Let $i_\ell$ be the position encoded by $(\tau_\ell,s_\ell,j_\ell)$, and let
$J=\{\ell:\tau_\ell=\mathsf v\}$.  Then, the valid parameter tuples
$(n,(j_\ell)_{\ell\in J})$ admit a finite Presburger-definable partition such
that, on each part,
\[
   \kappa^{w_n}(i_1,\ldots,i_k)
   =b_0+b_n n+\sum_{\ell\in J} b_\ell j_\ell,
\]
where $b_0,b_n,b_\ell\in\mathbb Q^d$ are fixed on that part and the expression
takes values in $\Z^d$ on the part.
\end{lemma}

\begin{proof}
Recall (Definitions~\ref{def:rank-source} and~\ref{def:prefix-rank}) that a rank
source is a finite-state scanner that adds an integer vector
$\omega(q,a)\in\Z^d$ as it reads each letter $a$ in state $q$ and reports the
accumulated total.  Its \emph{prefix rank} before a position is that running
total, the abstract form of the height.  By
Definition~\ref{def:prefix-additive-rank}, $\kappa$ is the sum of a fixed
constant and one contribution from an additive rank source for each coordinate.
It is therefore enough to analyse one such source
$A=(Q,q_0,\delta,\omega)$; the same argument can then be applied coordinate by
coordinate, and the resulting contributions can be added.  Write
$q_u$ for the state of $A$ after reading $u$.  Suppose that the scanner is in
state $q$ immediately before reading one of the $n$ copies of $v$, and that its
accumulated weight at that point is $r\in\Z^d$.  After reading the entire copy,
its state is $\delta_v(q)$ and its accumulated weight is
\[
   r+\Omega_v(q),
\]
where $\delta_v$ is the transition induced by the fixed word $v$ and
$\Omega_v(q)\in\Z^d$ is the total weight added while scanning $v$ from state
$q$.  Both quantities depend only on the state $q$ held immediately before the
block is read, because every repeated block is the same fixed word $v$.

For each $m\ge 0$, the state $\delta_v^m(q_u)$ is the state of $A$ after reading
the prefix $uv^m$.  The sequence of these states,
\[
   q_u,\ \delta_v(q_u),\ \delta_v^2(q_u),\ldots,
\]
is eventually periodic: as $Q$ is finite, there are a threshold $k_0$ and a
period $p$ with
$\delta_v^{\,m+p}(q_u)=\delta_v^{\,m}(q_u)$ for all $m\ge k_0$.  Partition the
possible values of $m$ into finitely many cases: one case for each initial value
$m<k_0$, and, for $m\ge k_0$, one case for each \emph{residue class} modulo $p$.
A residue class is the set of integers having a fixed remainder modulo $p$;
it is not a state of $A$.  Each case is Presburger-definable using equalities,
inequalities, and congruences.  On each case, the state after reading $uv^m$ is
one fixed element of $Q$, and
the total weight of the $m$ completed copies is affine in $m$, with rational
coefficients: each whole period contributes a fixed weight increment, and the
bounded remainder is fixed on the case.  The value is integral for every $m$
in that residue class.

Now compute the prefix rank before a position $(\tau,s,j)$.
\begin{itemize}
\item If $\tau=\mathsf u$, the position lies in the fixed prefix $u$.  The
prefix before offset $s$ is a fixed word, so both the prefix rank and the state
just before the position are constant.
\item If $\tau=\mathsf v$, the position lies at offset $s$ in the $j$-th copy
of $v$.  The prefix consists of $u$, then $j-1$ whole copies of $v$, then the
fixed prefix of $v$ of length $s-1$.  By the previous paragraph the contribution
of the $j-1$ whole copies is affine in $j$ on residue classes; the final
within-copy contribution depends only on $s$ and on the state after reading
$uv^{j-1}$, so it is constant on the same residue case.  Thus, the prefix rank
is affine in $j$.
\item If $\tau=\mathsf z$, the position lies in the fixed suffix $z$.  The
prefix contains $u$, all $n$ copies of $v$, and the fixed prefix of $z$ of
length $s-1$.  The contribution of the $n$ whole copies is affine in $n$ on
residue classes, and the bounded contribution inside $z$ is fixed on the same
case.  Thus, the prefix rank is affine in $n$.
\end{itemize}

For coordinate $\ell$, apply this analysis with $A=A_\ell$ and $i=i_\ell$.
The same case split also fixes the local correction
$\beta_\ell(q^{A_\ell}_{i_\ell},a_{i_\ell})$.  Its two arguments are the
scanner state $q^{A_\ell}_{i_\ell}$ immediately before position $i_\ell$ and
the input letter $a_{i_\ell}$ at that position.  If
$\tau_\ell=\mathsf u$, both $q^{A_\ell}_{i_\ell}$ and $a_{i_\ell}$ are fixed.
If $\tau_\ell=\mathsf v$, the letter $a_{i_\ell}$ is determined by the fixed
offset $s_\ell$, while the state $q^{A_\ell}_{i_\ell}$ is determined by
$s_\ell$ and the residue class of $j_\ell$.  If
$\tau_\ell=\mathsf z$, the letter $a_{i_\ell}$ is again determined by
$s_\ell$, while the state $q^{A_\ell}_{i_\ell}$ is determined by $s_\ell$ and
the residue class of $n$.
Therefore, each coordinate contribution to $\kappa$ is affine in $n$ and the
relevant repetition index on a finite Presburger-definable partition.  Take a
common refinement of these finitely many partitions, one for each coordinate
contribution in $\kappa$.  On every part of the refinement, summing the
coordinate contributions and the fixed constant $c_0$ gives the affine
expression in the statement.
\end{proof}

For the application to a $\WRP$ presentation, a copy name $c$ determines its
arity $k_c$ and the relevant prefix-additive rank function $\kappa_c$; the copy
name is not an additional numerical argument of $\kappa_c$.  For a selected
atom $(c,\bar\imath)$, the vector $\kappa_c^w(\bar\imath)$ is its atom rank.
Hence, after fixing the copy names and the region and offset of every coordinate
of two selected atoms, the lemma applies separately to their two rank functions.
On a common refinement of the resulting finite partitions, both atom ranks are
affine vectors with rational coefficients.  Clearing the finitely many fixed
denominators turns their lexicographic comparison into a finite Boolean
combination of integer linear equalities and inequalities, so it is
Presburger-definable.

\begin{lemma}[WRP data on a regular slice]
\label{lem:one-loop-presburger}
Fix a $\WRP$ presentation and a regular slice $(w_n)_{n\ge0}$ with
$w_n=uv^nz$.  For every potential
atom under consideration, fix its copy name and the region and offset of each
coordinate.  Each of the following is then Presburger-definable in $n$ and the
remaining repetition indices:
\begin{enumerate}[label=(\alph*)]
\item whether a given potential atom is selected;
\item whether a given selected atom has a specified output label;
\item whether the first of two selected atoms has the smaller atom rank;
\item whether the first of two selected atoms of equal atom rank comes first in
the tie-order.
\end{enumerate}
\end{lemma}

\begin{proof}
The selection and label predicates, as well as the tie-order component of part
(d), are given by $\MSO$ formulas and are therefore covered by
Lemma~\ref{lem:one-loop-finite-state}.  Part (c), and the equal-rank condition
in part (d), follow from the atom-rank comparison described after
Lemma~\ref{lem:one-loop-rank-affine}.  Combining these conditions preserves
Presburger-definability.
\end{proof}

\subsection{Second ingredient: bounded Presburger counts are semilinear}

We now state the precise counting principle that forms the second ingredient
of the overview.

\begin{lemma}[Bounded counting]\label{lem:presburger-counting}
Let $p,q\ge 1$, and let $A\subseteq\N^p\times\N^q$ be a Presburger-definable relation
from $\N^p$ to $\N^q$, and for $x\in\N^p$, write
\[
   A_x=\{y\in\N^q:(x,y)\in A\}.
\]
Suppose every $A_x$ is finite and, for some constant
$C$ independent of $x$,
\[
   |A_x|\le C\bigl(\|x\|_\infty+1\bigr)
   \qquad\text{where}\
   \|x\|_\infty=\max_{1\le i\le p}x_i.
\]
Then,
\[
\{(x,\,|A_x|):x\in\N^p\}
\]
is semilinear.  The conclusion also holds jointly: if, for $1\le i\le r$,
$q_i\ge1$ and $A^i\subseteq\N^p\times\N^{q_i}$ is Presburger-definable, with
finite sets $A^i_x=\{y\in\N^{q_i}:(x,y)\in A^i\}$ satisfying a bound of the
same form, then
\[
   \{(x,|A^1_x|,\ldots,|A^r_x|):x\in\N^p\}
\]
is semilinear.
\end{lemma}

\begin{proof}
We use one standard input from the theory of integer-point counting, which we
state in full before applying it.  For each $x$, we count the vectors $y$ for
which $(x,y)\in A$.  The fact we need is that this count is \emph{piecewise
quasi-polynomial}: the space $\N^p$ splits into finitely many
Presburger-definable regions, and on each
region, the count $|A_x|$ agrees with a \emph{quasi-polynomial} in $x$.  A
quasi-polynomial is a polynomial in $x_1,\ldots,x_p$ whose coefficients depend on
$x$ only through the residues of its entries modulo a fixed integer $M$;
equivalently, it is a finite list of ordinary polynomials, one selected by the
residue class of $x$ modulo $M$.  This piecewise quasi-polynomiality is the
Presburger counting theorem of Woods~\cite[Theorem~1.10]{Woods2015}, which
generalises classical integer-point counting results such as Ehrhart's
theorem~\cite{Ehrhart1962}.
Since a quasi-polynomial agrees with an ordinary polynomial on each residue
class modulo $M$, we may subdivide these regions further so that, on each
resulting Presburger-definable region $R$, the count $|A_x|$ agrees with a
single ordinary polynomial $P_R(x)$.

Only finitely many such regions $R$ arise.  The Ginsburg--Spanier
characterisation recalled above shows that each $R$ is semilinear and hence a
finite union of linear sets.  Since semilinear sets are closed under finite
unions, it is enough to prove that the graph of the count over each of these
linear sets is semilinear.  Fix one such region $R$ and one linear set in its
decomposition:
\[
   L=\{\beta+t_1s_1+\cdots+t_ms_m:t_1,\ldots,t_m\in\N\}\subseteq R.
\]
Substituting this parametrisation into $P_R$ gives the polynomial
\[
   h(t_1,\ldots,t_m)
   =P_R(\beta+t_1s_1+\cdots+t_ms_m).
\]
It takes nonnegative integer values, and the assumed bound gives
\[
   0\le h(t_1,\ldots,t_m)
   \le C\bigl(\|\beta+t_1s_1+\cdots+t_ms_m\|_\infty+1\bigr)
   \le C'(1+t_1+\cdots+t_m)
\]
for a constant $C'$ depending only on $L$.  Hence $h$ has degree at most one.
Indeed, if its highest-degree homogeneous part $h_d$ had degree $d\ge2$, choose
$c\in\N_{>0}^m$ with $h_d(c)\ne0$.  Along $t=kc$ for $k\in\N$, the polynomial $h$ has the expansion
\[
   h(kc)=k^d h_d(c)+O(k^{d-1}).
\]
If $h_d(c)<0$, then $h(kc)<0$ for all sufficiently large $k$, contradicting
the nonnegativity of $h$.  If $h_d(c)>0$, then $h(kc)$ grows on the order of
$k^d$, contradicting the upper bound from above:
\[
   h(kc)\le C'\bigl(1+k(c_1+\cdots+c_m)\bigr),
\]
which is linear in $k$.  Thus, $d\ge2$ is impossible.

Consequently,
\[
   h(t_1,\ldots,t_m)=b+a_1t_1+\cdots+a_mt_m.
\]
Since $h(t_1,\ldots,t_m)\in\N$ for every $(t_1,\ldots,t_m)\in\N^m$, the coefficients
$b,a_1,\ldots,a_m$ are integers, and nonnegativity on all multiples of each
coordinate gives $a_i\ge0$.  The graph over $L$ is therefore the linear set
\[
   \{(\beta,b)+t_1(s_1,a_1)+\cdots+t_m(s_m,a_m):
       t_1,\ldots,t_m\in\N\}.
\]
Taking the finite union of these graphs over all linear sets and all regions
supplied by Woods proves that $\{(x,|A_x|):x\in\N^p\}$ is semilinear.  For the
tuple version, each individual graph
$\{(x,|A^i_x|)\}\subseteq\N^p\times\N$ is semilinear by this same argument.
Lift each of them to $\N^p\times\N^r$ by leaving the other $r-1$
count-coordinates unconstrained (a product of the graph with copies of $\N$,
still semilinear), and intersect the $r$ lifts: a tuple $(x,c_1,\ldots,c_r)$
lies in the intersection exactly when $c_i=|A^i_x|$ for every $i$, which is the
joint graph $\{(x,|A^1_x|,\ldots,|A^r_x|)\}$.  Semilinear sets are closed under
finite union, product, and intersection, so both the single graph and its tuple
version are semilinear.
\end{proof}

This conclusion accords with the general fact that a Presburger-definable
function is linear on each part of a finite definable
partition~\cite[Corollary~3.1.4]{CluckersHalupczok2018}; the argument above
also establishes the required definability of the counting graph.

\subsection{Combining the two ingredients: first-ascent semilinearity}

We define two statistics on output words.  For $y\in\{U,D\}^*$, let $\fas(y)$
be the number of $U$'s before its first $D$, and let $\tailU(y)$ be the number
after that $D$.
If $y$ has no $D$, set $\tailU(y)=0$ and $\fas(y)$ to the total number of
$U$'s.  For a Dyck path, $\fas$ is the length of the initial ascent and
$\tailU$ is the number of later up-steps.  We call
$(\fas(y),\tailU(y))$ the \emph{first-ascent pair} of $y$.  Thus, for the
counts introduced at the beginning of this section,
$a_n=\fas(T(W_n))$ and $b_n=\tailU(T(W_n))$.
The theorem below asserts that, if $T\in\WRP$ is defined on every $W_n$ and
$|T(W_n)|=O(n)$, these first-ascent pairs form a semilinear set as $n\ge1$
varies.

We now explain how the two preceding ingredients yield this conclusion.  Fix a
$\WRP$ map $T$ and the family $W_n=U(UD)^nD$ from the beginning of this
section.
Lemma~\ref{lem:one-loop-presburger} expresses the selection, labels, and output
order of the potential atoms of $T(W_n)$ by Presburger formulas.  The integer
tuples encoding $n$, the first selected $D$-atom, and a selected
$U$-atom before it therefore form a Presburger-definable set; the same holds
with ``before'' replaced by ``after.''  For each fixed encoding of $n$ and the
first $D$-atom, the two resulting sets of encoded $U$-atoms have sizes
$\fas(T(W_n))$ and $\tailU(T(W_n))$.  The case in which the output has no $D$ is treated
separately.  Both counts are bounded by $|T(W_n)|=O(n)$, so
Lemma~\ref{lem:presburger-counting} implies that their joint counting graphs
are semilinear.  Taking the required finite unions and projections then gives
the desired semilinear set of first-ascent pairs.

\begin{theorem}[First-ascent semilinearity]
\label{thm:wrp-slice-semilinearity}
Let $T\in\WRP$ with $W_n\in\operatorname{dom}(T)$ for every $n\ge1$, and
suppose $|T(W_n)|=O(n)$.  Then, the set
$S_T=\{(\fas(T(W_n)),\tailU(T(W_n))):n\ge 1\}\subseteq\N^2$
is semilinear.
\end{theorem}

\begin{proof}
Restrict the $\WRP$ presentation of $T$ to the family
$(W_n)_{n\ge1}$, where $W_n=U(UD)^nD$.
A \emph{potential atom} is a pair $(c,\bar i)$ consisting of a copy name $c$ and
a tuple $\bar i$ of input positions of $c$'s arity.  A potential atom becomes
a \emph{selected atom}, contributing one output letter, exactly when it
satisfies the selection formula $\varphi_c$.  On this family, each potential
atom is named by the finite data of its copy name, regions, and offsets,
together with $n$ and the repetition indices of its coordinates lying in $v$
(the position encoding preceding Example~\ref{ex:slice-coordinates}).  There
are only finitely many choices of the copy name, regions, and offsets.  Assign
a distinct integer tag to each choice, pad its list of repetition indices with
zeros to a fixed length, and thereby give every potential atom a unique
fixed-length numerical encoding.  For each fixed tag,
Lemma~\ref{lem:one-loop-presburger} makes the following data
Presburger-definable: whether the atom is selected, which letter labels it,
and the atom-rank and tie-order comparisons between two atoms.  Taking the
finite disjunction over all tags preserves Presburger-definability.  Thus, the
same predicates remain Presburger-definable when the atom variables below
range over all potential atoms, and counting their encodings counts each atom
exactly once.

\vspace{0.5em}
\noindent\emph{Splitting on the first descent.}
Whether $T(W_n)$ has any $D$-atom is, by Lemma~\ref{lem:one-loop-presburger}, a
Presburger-definable condition on $n$ (some potential atom is selected and labelled $D$); split
on it.  If there is no $D$-atom, every selected atom is a $U$, so
\[
   \fas(T(W_n))=|\{\text{selected $U$-atoms}\}|,\qquad
   \tailU(T(W_n))=0.
\]
The family $A=\{(n,u):u\text{ a potential atom, selected and labelled $U$ over }W_n\}$
is Presburger-definable in $n$ and the coordinates of $u$.  Its definition
conjoins the predicates ``$u$ selected'' and ``$u$ labelled $U$'' supplied by
Lemma~\ref{lem:one-loop-presburger}, and for each $n$, the corresponding set
$A_n$ has size
$\fas(T(W_n))\le|T(W_n)|=O(n)$, linear in $n$.  Thus,
Lemma~\ref{lem:presburger-counting} implies that
$\{(n,\fas(T(W_n))):T(W_n)\text{ has no }D\}$ is semilinear, and appending the
constant coordinate $\tailU=0$ gives a semilinear subset of $\N^3$.
Projecting away $n$ gives the no-descent contribution to $S_T\subseteq\N^2$.

\vspace{0.5em}
\noindent\emph{The two $U$-counts when a descent occurs.}
Otherwise, let $d$ be the first selected $D$-atom in the output order $\prec$.
Fix the finite case of $d$ (its copy name and the region and offset of each
coordinate), so that $d$ is named by the parameter
$x=(n,\text{repetition indices of }d)$.  There are only finitely many such cases, and
we take their union at the end.  The statement ``$d$ is the
first selected $D$-atom over $W_n$'' is Presburger-definable in $x$: by
Lemma~\ref{lem:one-loop-presburger}, ``$d$ selected'' and ``$d$ labelled $D$''
are Presburger-definable, and ``no selected $D$-atom precedes $d$ in $\prec$''
is the negation of an existential Presburger formula over a potential atom's
coordinates, and is therefore Presburger-definable as well.  Let $P$ be the
Presburger-definable set of encodings $x$ for
which $d$ is the first selected $D$-atom of $T(W_n)$.
Split the selected $U$-atoms by their side of $d$,
\[
   U_{<d}=\{u: u\text{ selected, labelled }U,\ u\prec d\},\qquad
   U_{>d}=\{u: u\text{ selected, labelled }U,\ d\prec u\}.
\]
Because $\prec$ totally orders the selected atoms and $d$ (a $D$-atom) lies in
neither set, every selected $U$-atom lands in exactly one of them.  Moreover,
$\fas(T(W_n))=|U_{<d}|$ and $\tailU(T(W_n))=|U_{>d}|$.  To apply
Lemma~\ref{lem:presburger-counting}, fix $x$ and count the coordinate tuples of
$U$-atoms $u$ in each set.  Each membership condition ``$u\in U_{<d}$'' and
``$u\in U_{>d}$'' is a Boolean combination of the predicates of
Lemma~\ref{lem:one-loop-presburger} (``selected'', ``labelled $U$'', and the
comparison $u\prec d$ or $d\prec u$ against the atom named by $x$), hence
Presburger-definable in $(x,u)$.  For each fixed $x$, both resulting sets of
$u$-values have size at most the number of selected atoms,
$|T(W_n)|=O(n)$, linear in $x$.  The tuple form of
Lemma~\ref{lem:presburger-counting} then implies that the set
$\{(x,|U_{<d}|,|U_{>d}|)\}$ is semilinear; intersecting with $P$ (closure under
intersection) restricts to the $x$ for which $d$ is genuinely the first $D$-atom,
\[
   \bigl\{(x,\,|U_{<d}|,\,|U_{>d}|):x\in P\bigr\}.
\]

\vspace{0.5em}
\noindent\emph{Assembling $S_T$.}
On $x\in P$, we have $(\fas(T(W_n)),\tailU(T(W_n)))=(|U_{<d}|,|U_{>d}|)$, so
projecting the last display onto its final two coordinates (that is, forgetting
$x=(n,\text{repetition indices of }d)$; projection preserves semilinearity)
gives exactly the pairs coming from inputs with a descent.  Taking the union
with the no-$D$ pairs from the previous step, and then over the finitely many
fixed cases of $d$, preserves semilinearity.  Therefore, $S_T\subseteq\N^2$ is
semilinear.
\end{proof}

\section{The family \texorpdfstring{$(W_n)_{n\ge1}$}{(Wn), n >= 1}
and the no-swap theorem}\label{sec:wrapped-flat}\label{sec:main}

We now turn from $\WRP$ to combinatorics.  The members of the family
$(W_n)_{n\ge1}$ have very simple statistics, and a hypothetical swap $F$
forces the first-ascent pairs of their images $F(W_n)$ to form a nonsemilinear
set.  Combining this fact with
Section~\ref{sec:slice-semilinearity} gives the no-swap theorem.

The argument is a short chain of three steps, which we preview here before
proving its links.
\begin{enumerate}[label=(\arabic*)]
\item \emph{The family's statistics.}  The paths $W_n=U(UD)^nD$ satisfy
$\area(W_n)=n$ and $\dinv(W_n)=\binom n2$ (Lemma~\ref{lem:wrapped-flat-stats}).
\item \emph{A swap sends it to a deficit-zero target.}  If $F$ exchanges
$\area$ and $\dinv$, then $Q_n:=F(W_n)$ has
$\area(Q_n)=\dinv(W_n)=\binom n2$ and
$\dinv(Q_n)=\area(W_n)=n$.  Recall the \emph{coarea}
$\coarea(P)=\binom N2-\area(P)$, the amount by which a semilength-$N$ path falls
short of the maximum area $\binom N2$ (Section~\ref{sec:dyck}).  As $Q_n$ has
semilength $N=n+1$, this gives $\coarea(Q_n)=\binom{n+1}2-\binom n2=n=\dinv(Q_n)$
and hence $\defc(Q_n)=0$.  This is the extreme case of the general inequality
$\dinv\le\coarea$ (Lemma~\ref{lem:dinv-coarea}), so any swap must land on
deficit-zero targets.
\item \emph{Deficit zero pins down the pair.}  Equality $\dinv=\coarea$ is
rigid: a deficit-zero path is forced to be a near-staircase
(Lemma~\ref{lem:deficit-zero-targets}), so the first-ascent pairs arising as
$n$ varies form exactly the explicit set
$S_{\mathrm{tri}}\subseteq\N^2$, the ``triangular'' region defined in
Corollary~\ref{cor:forced-triangular-pairs}.  As $b$ grows,
$S_{\mathrm{tri}}$ has a quadratically curving lower boundary and is \emph{not}
semilinear (Lemma~\ref{lem:triangular-not-semilinear}).
\end{enumerate}
This contradicts the semilinearity theorem of
Section~\ref{sec:slice-semilinearity} (which makes the set of first-ascent pairs
of any linear-growth $\WRP$ map semilinear), and that contradiction is the
no-swap theorem.  The substantive step is~(3), the rigidity of the
deficit-zero condition.

\begin{lemma}[Statistics of $W_n$]\label{lem:wrapped-flat-stats}
For every $n\ge1$, the path $W_n=U(UD)^nD$ has area sequence
$a(W_n)=(0,\underbrace{1,\ldots,1}_{n\text{ times}})$.  Thus,
$\area(W_n)=n$ and $\dinv(W_n)=\binom n2$.
\end{lemma}

\begin{proof}
The first up-step starts at height $0$, and each of the $n$ middle
up-steps starts at height $1$; this gives the stated area sequence, so
$\area(W_n)=n$.  For dinv, the leading entry $0$ forms no pair, since
$0-1=-1\notin\{0,1\}$, so the only contributing pairs are the equal pairs
among the $n$ entries equal to $1$, of which there are $\binom n2$.
\end{proof}

The inequality in the next lemma and the uniqueness of its equality cases
appear in Ammar's thesis~\cite[Lemma~4.13, Theorem~4.18, and
Remark~4.20]{Ammar2015}; the same equality cases are described through the
deficit statistic in work of Lee, Li, and Loehr~\cite[Lemma~2.10]{LeeLiLoehr2018}.
We state the equality cases in the explicit
area-sequence form needed for the lower bound; proofs of the next two lemmas
are given in Appendix~\ref{app:deficit-zero-proofs}.

\begin{lemma}[Dinv below coarea]\label{lem:dinv-coarea}
For every Dyck path $Q\in\D_N$, we have $\dinv(Q)\le\coarea(Q)$.
\end{lemma}

The equality cases of Lemma~\ref{lem:dinv-coarea} are precisely the
deficit-zero paths.  The next lemma restates their rigidity in the form used
below: for each semilength and each value of the common quantity
$c=\dinv(Q)=\coarea(Q)$, there is exactly one such path, an initial staircase
followed by a short two-level tail.\footnote{Here ``staircase'' refers to the
usual north/east drawing from Section~\ref{sec:dyck}, with $U$ drawn north and
$D$ east: the initial area-sequence segment $0,1,2,\ldots$ gives the successive
row lengths of a staircase-shaped area diagram.  In the up/down drawing of the
same word, this segment appears simply as the initial ascent.}

\begin{lemma}[Deficit-zero rigidity]\label{lem:deficit-zero-targets}
Fix $N\ge 1$ and $c\in\{0,\ldots,\binom N2\}$.  There is a unique Dyck path
$\Gamma_{N,c}\in\D_N$ satisfying
\[
   \dinv(\Gamma_{N,c})=\coarea(\Gamma_{N,c})=c.
\]
Letting $b$ be the unique integer with
$\binom{b+1}{2}\le c<\binom{b+2}{2}$ and putting $d=c-\binom{b+1}{2}$ and
$a=N-b$ (equivalently, $c=\binom{b+1}{2}+d$ with $0\le d\le b$), the area sequence
of $\Gamma_{N,c}$ is
\[
   (0,1,\ldots,a-1,
     \underbrace{a-1,\ldots,a-1}_{b-d\text{ times}},
     \underbrace{a-2,\ldots,a-2}_{d\text{ times}}),
\]
interpreted in the natural way when $b=0$: the plain staircase
$(0,1,\ldots,N-1)$, with $a=N$.  Its first-ascent length is $a$
and it has $b$ later up-steps.
\end{lemma}

\begin{example}[A deficit-zero target]
Take $N=6$ and $c=7$.  Since $\binom 42=6\le 7<10=\binom 52$, we have
$b=3$, $d=1$, and $a=N-b=3$.  The lemma gives area sequence
$(0,1,2,2,2,1)$ and path $\Gamma_{6,7}=UUUDUDUDDUDD$.
Its first ascent has length $3$ and it has $3$ later up-steps.  This example is
typical: after the initial area-sequence segment $0,1,\ldots,a-1$, every
remaining entry is either $a-1$ or $a-2$, with all the $a-1$ entries preceding
all the $a-2$ entries.
\end{example}

\begin{corollary}[Triangular set of first-ascent pairs]
\label{cor:forced-triangular-pairs}
Suppose $F\colon\D\to\D$ is a semilength-preserving area--dinv swap, that is,
\[
   \area(F(P))=\dinv(P),\qquad \dinv(F(P))=\area(P)
\]
for every Dyck path $P$.  Put $Q_n=F(W_n)$.  Then,
\[
   \{(\fas(Q_n),\tailU(Q_n)):n\ge 1\}=S_{\mathrm{tri}},
\]
where
\[
   S_{\mathrm{tri}}=
   \Bigl\{(a,b)\in\N^2: b\ge 1,\ \binom b2+1\le a\le\binom{b+1}{2}+1\Bigr\}.
\]
\end{corollary}

\begin{proof}
\emph{Identifying $Q_n$.}
Fix $n\ge 1$.  By Lemma~\ref{lem:wrapped-flat-stats},
$\area(Q_n)=\dinv(W_n)=\binom n2$ and $\dinv(Q_n)=\area(W_n)=n$; and since
$Q_n\in\D_{n+1}$, we have
$\coarea(Q_n)=\binom{n+1}{2}-\binom n2=n=\dinv(Q_n)$.  Thus, $Q_n$ is
deficit-zero, and by Lemma~\ref{lem:deficit-zero-targets} (with
$N=n+1$ and $c=n$), it is the unique deficit-zero target $\Gamma_{n+1,n}$: an
initial staircase followed by a short two-level tail.  Write $a=\fas(Q_n)$ for
its first-ascent length and
$b=\tailU(Q_n)$ for its number of later up-steps.  Since $Q_n$ has $n+1$
up-steps, $n=a+b-1$.
Lemma~\ref{lem:deficit-zero-targets} identifies $b$ as the unique integer with
\[
   \binom{b+1}{2}\le n<\binom{b+2}{2}.
\]
We also have $b\ge1$: otherwise $0\le n<\binom22=1$, so $n=0$, contradicting
$n\ge1$.
 
\vspace{0.5em}
\noindent\emph{The pairs range exactly over $S_{\mathrm{tri}}$.}
It remains to identify the image $\{(\fas(Q_n),\tailU(Q_n)):n\ge 1\}$ with
$S_{\mathrm{tri}}$.  Substituting $a=(n+1)-b$ and using the identities
$\binom{b+1}{2}-b=\binom b2$ and $\binom{b+2}{2}-b=\binom{b+1}{2}+1$, the range
that pins down $b$ turns into a range for $a$:
\[
   \binom{b+1}{2}\le n<\binom{b+2}{2}
   \quad\Longleftrightarrow\quad
   \binom b2+1\le a\le\binom{b+1}{2}+1
   \qquad(n=a+b-1).
\]
Read left to right, this says the pair $(a,b)$ of each $Q_n$ has $b\ge 1$ and
$\binom b2+1\le a\le\binom{b+1}{2}+1$, so it lies in $S_{\mathrm{tri}}$.  Read
right to left, it shows every point of $S_{\mathrm{tri}}$ is attained: given
$(a,b)$ with $b\ge 1$ and $\binom b2+1\le a\le\binom{b+1}{2}+1$, set $n=a+b-1$.
Then, $n\ge 1$ (because $a\ge\binom b2+1\ge 1$ and $b\ge 1$), so $n$ is one of the
indices the image ranges over.  By the equivalence, this $n$ lies in the range
$\binom{b+1}{2}\le n<\binom{b+2}{2}$, and $b$ is the unique integer placing it
there, so the description above gives $\tailU(Q_n)=b$ and
$\fas(Q_n)=(n+1)-b=a$.  Thus, the image is exactly $S_{\mathrm{tri}}$.
\end{proof}

\begin{lemma}[Semilinear envelopes]
\label{lem:semilinear-envelope}
Let $S\subseteq\N^2$ be semilinear, and suppose every vertical section
$S_b=\{a:(a,b)\in S\}$ is finite.  Then, there is a period $M\ge1$, determined
by $S$, such that on every residue class modulo $M$, exactly one of the following
holds:
\begin{enumerate}[label=(\roman*)]
\item $S_b$ is empty for all sufficiently large $b$ in the class;
\item $S_b$ is nonempty for all sufficiently large $b$ in the class, and its
lower envelope $m(b)=\min S_b$ is \emph{eventually affine}: there are constants
$\alpha,\gamma$, depending on the class, such that
$m(b)=\alpha b+\gamma$ for all sufficiently large $b$ in the class.
\end{enumerate}
\end{lemma}

\begin{proof}
The plan is to handle one linear set at a time, keeping track of both the
second coordinates at which it has points and its lower envelope there.  On
each residue class, a linear set will eventually have either no points with
second coordinate $b$, or at least one point for every sufficiently large $b$
in the class; in the latter case, its envelope will be affine.  We then combine
the finitely many linear sets and take a minimum of their envelopes.
In this proof, we work with the two coordinates of a point $(a,b)\in\N^2$ directly:
$a$ is its \emph{first} coordinate and $b$ its \emph{second}.  The section
$S_b=\{a:(a,b)\in S\}$ collects the first coordinates occurring with a given
second coordinate $b$.  When $S_b$ is nonempty, $m(b)=\min S_b$ is its smallest
element; when $S_b$ is empty, $m(b)$ is undefined.  (If one plots $b$ on the
horizontal axis and $a$ on the vertical axis, a nonempty $S_b$ is the vertical
slice of $S$ sitting above $b$, and $m(b)$ is its lowest point, the lower
envelope of $S$; that is the picture behind the names.)

\vspace{0.5em}
\noindent\emph{Reduction to one linear set.}  By semilinearity, $S$ is a finite
union of linear sets; write $S=L_1\cup\cdots\cup L_r$, where each $L_i$ is the
set of points reachable from its own base point by taking nonnegative integer
numbers of steps along its own finitely many step vectors.  A
point of $S$ has second coordinate $b$ exactly when it lies in some $L_i$ with
second coordinate $b$.  Write $(L_i)_b=\{a:(a,b)\in L_i\}$ for the section of
$L_i$ at second coordinate $b$ (the analogue of $S_b$ for the single linear
set $L_i$), so that $S_b=\bigcup_{i=1}^r (L_i)_b$.  When $(L_i)_b\neq\emptyset$,
that is, when $L_i$ has a point at second coordinate $b$, put
\[
   m_i(b)=\min (L_i)_b=\min\{a:(a,b)\in L_i\},
\]
the least first coordinate $L_i$ attains at second coordinate $b$; when
$(L_i)_b=\emptyset$, we leave $m_i(b)$ undefined, as $L_i$ then has no point
there.  Whenever $S_b$ is nonempty,
\[
   m(b)=\min_{i:\,(L_i)_b\neq\emptyset}\ m_i(b)
\]
is the minimum of a nonempty finite set of finite values.  It is enough to show
that, for each linear set $L_i$ and on each residue class modulo a suitable
period, one of two alternatives holds for all sufficiently large $b$: either
$(L_i)_b$ is always empty, or it is always nonempty and $m_i(b)$ is affine.
We establish this dichotomy separately for each $L_i$ in the four steps below.
We then return to the finite union $S=L_1\cup\cdots\cup L_r$ and pass to a
common period.  On each resulting residue class, either no component is
eventually present, in which case $S_b$ is eventually empty, or a fixed
collection of components is present, in which case $m(b)$ is the minimum of
their affine envelopes and is therefore eventually affine.

\vspace{0.5em}
\noindent\emph{One linear set.}  Fix $L=L_i$, with base
$\beta=(\beta_1,\beta_2)$ and step vectors $s_j=(s_{j,1},s_{j,2})$.

\vspace{0.5em}
\noindent\emph{Step 1: every step increases the second coordinate.}  First discard any
step vector equal to $(0,0)$, since taking it any number of times adds nothing to
a point.  No remaining step can leave the second coordinate
unchanged while moving the first: a step $s_j=(s_{j,1},0)$ with $s_{j,1}>0$ would
make $\beta,\,\beta+s_j,\,\beta+2s_j,\ldots$ infinitely many points all with the
same second coordinate $\beta_2$, so the section $S_{\beta_2}$ would be infinite,
contrary to hypothesis.  Thus, every remaining step has $s_{j,2}\ge 1$.  If no
step remains, then $L=\{\beta\}$ is a single point, whose section is nonempty only
at $b=\beta_2$.  Therefore, this component is eventually absent on every residue
class, as in alternative~(i) of the lemma, and we set it aside.  From now on,
$L$ has at least one step, each with $s_{j,2}\ge 1$; in particular, the minimum
in Step~3 below is over a nonempty set.

\vspace{0.5em}
\noindent
\emph{Step 2: reduction to one integer variable.}  A point of $L$ has second
coordinate $b$ exactly when its step counts $(n_j)$ satisfy
$\beta_2+\sum_j n_j s_{j,2}=b$, and its first coordinate is then
$\beta_1+\sum_j n_j s_{j,1}$.  Minimising the first coordinate over all such step
counts gives
\[
   m_L(b)=\beta_1+g(b-\beta_2),\qquad
   g(t)=\min\Bigl\{\textstyle\sum_j n_j s_{j,1}: n_j\in\N,\ \sum_j n_j s_{j,2}=t\Bigr\}.
\]
In words: among all ways of choosing how many of each step to take so that their
second coordinates add up to exactly $t$, $g(t)$ is the least achievable total of
their first coordinates (defined for those $t$ that arise as such a sum).

\vspace{0.5em}
\noindent\emph{Step 3: $g$ grows linearly, at a fixed rate.}  Let
$c=\min_j s_{j,1}/s_{j,2}$ be the smallest ratio of first coordinate to second
coordinate among the steps, attained by a step we call the \emph{shallowest}
(in the picture, the step of smallest slope $s_{j,1}/s_{j,2}$).  Since
$s_{j,1}\ge c\,s_{j,2}$ for every $j$, any feasible step counts give
$\sum_j n_j s_{j,1}\ge c\sum_j n_j s_{j,2}=c\,t$; hence $g(t)\ge c\,t$, with
equality whenever $t$ is a multiple of the shallowest step's second coordinate
(take only copies of that step).  So $g$ lies on or above the line of slope $c$,
meeting it at every multiple of the shallowest step's second coordinate.

\vspace{0.5em}
\noindent\emph{Step 4: each residue class is eventually absent or affine.}  Let
$d=s_{*,2}\ge 1$ be the shallowest step's second coordinate, so that
$c\,d=s_{*,1}$ is an integer (its first coordinate).  Taking one more copy of the
shallowest step turns any step counts summing to $t$ into step counts summing to
$t+d$, at an extra first-coordinate total of $s_{*,1}$; hence
\[
   g(t+d)\le g(t)+s_{*,1}.
\]
In particular, once $g$ is defined at an argument $t_0$, it is defined at every
$t_0+\nu d$ with $\nu\ge 0$.  Consequently, on each residue class modulo $d$,
either $g$ is nowhere defined, or it is defined at all sufficiently large
arguments in the class.  In the first case, the corresponding sections of $L$
are eventually empty.  In the second case, fix the residue $r$ and write
$t=r+\nu d$.  For all sufficiently large $\nu$, the value $g(r+\nu d)$ is
defined, and we set
\[
   G_\nu=g(r+\nu d)-\nu\,s_{*,1}\in\Z.
\]
For all sufficiently large $\nu$ the preceding inequality gives $G_{\nu+1}\le G_\nu$, while
$g(t)\ge c\,t$ gives $G_\nu\ge c\,r\ge 0$.  A non-increasing sequence of integers that stays
nonnegative is eventually constant, say $G_\nu=G_\infty$ for all sufficiently large $\nu$;
therefore, for all sufficiently large $\nu$,
\[
   g(r+\nu d)
   =\nu\,s_{*,1}+G_\infty
   =\nu c d+G_\infty
   =c(r+\nu d)+(G_\infty-c r).
\]
Thus, for all sufficiently large $t\equiv r\pmod d$,
$g(t)=c t+(G_\infty-c r)$, a single linear formula in $t$.
Since subtracting the fixed number $\beta_2$ merely permutes the residue classes
modulo $d$, the identity $m_L(b)=\beta_1+g(b-\beta_2)$ proves the required
dichotomy for $L$: on each residue class of $b$ modulo $d$, its sections are
eventually empty, or they are eventually nonempty and $m_L$ is affine.

\vspace{0.5em}
\noindent\emph{Combining the finitely many linear sets.}  The singleton
components set aside in Step~1 are eventually absent on every residue class.
If every $L_i$ is such a component, then $S$ is finite; taking $M=1$ gives
alternative~(i), and the proof is complete.  Otherwise, Step~4 attaches to each
remaining $L_i$ a period $d_i\ge1$ for which, on every residue class modulo
$d_i$, the sections of $L_i$ are eventually empty or are eventually nonempty
with $m_i$ affine.

Let $M$ be a common multiple of these periods.  Fix a residue class modulo
$M$.  For all sufficiently large $b$ in this class, the collection of $L_i$
with $(L_i)_b\neq\emptyset$ is fixed, and each member of that collection
contributes one affine function $m_i(b)=\alpha_i b+\gamma_i$.  If the collection
is empty, then $S_b$ is empty for all sufficiently large $b$ in the class,
which is alternative~(i).  If the collection is nonempty, then $S_b$ is
nonempty for all sufficiently large $b$ in the class and
\[
   m(b)=\min_{i:\,(L_i)_b\neq\emptyset} m_i(b)
\]
is the minimum of finitely many lines, and for large $b$, the line of least slope
(ties broken by least intercept) lies below the rest, so $m$ coincides with that
single line.  This is alternative~(ii), and proves the lemma.
\end{proof}

\begin{lemma}[Triangular obstruction]
\label{lem:triangular-not-semilinear}
The set $S_{\mathrm{tri}}\subseteq\N^2$ is not semilinear.
\end{lemma}

\begin{proof}
Recall $S_{\mathrm{tri}}=\{(a,b):b\ge 1,\ \binom b2+1\le a\le\binom{b+1}2+1\}$.
For each $b\ge 1$, its vertical section is the finite integer interval from
$\binom b2+1$ to $\binom{b+1}2+1$, so it is nonempty with smallest element
\[
   m(b)=\min\{a:(a,b)\in S_{\mathrm{tri}}\}=\tbinom b2+1.
\]
Suppose $S_{\mathrm{tri}}$ were semilinear.  Its sections being finite,
Lemma~\ref{lem:semilinear-envelope} supplies a period $M$ and the two
alternatives in that lemma.  Here every section is nonempty for $b\ge1$, so
alternative~(i) is impossible on every residue class.  Thus, the lower envelope
is eventually affine on every class; fix one, say $b\equiv r\pmod M$.  On it,
$m(b)=\alpha b+\gamma$ for all large $b$.  But
$m(b)=\binom b2+1=\tfrac12 b^2-\tfrac12 b+1$ is a genuine quadratic in $b$;
restricted to the progression $b=r+M\nu$, it is still quadratic in $\nu$
(leading term $\tfrac12 M^2\nu^2$), so it cannot equal a linear function of
$b$ for all sufficiently large such $b$.  This contradiction shows
$S_{\mathrm{tri}}$ is not semilinear.
\end{proof}

The combinatorial part of the obstruction can now be stated without reference
to a computational model.

\begin{corollary}[Model-free form of the obstruction]
\label{cor:model-free-obstruction}
Let $F\colon\D\to\D$ be semilength-preserving and satisfy
$\area(F(P))=\dinv(P)$ and $\dinv(F(P))=\area(P)$ for every $P\in\D$.  Then,
\[
   \{(\fas(F(W_n)),\tailU(F(W_n))):n\ge1\}
\]
is not semilinear.
\end{corollary}

\begin{proof}
Corollary~\ref{cor:forced-triangular-pairs} identifies the displayed set with
$S_{\mathrm{tri}}$, which is not semilinear by
Lemma~\ref{lem:triangular-not-semilinear}.
\end{proof}

The no-swap theorem follows by combining this model-free obstruction with the
regular-slice semilinearity theorem.

\begin{theorem}[No $\WRP$ area--dinv swap]\label{thm:wrp-no-swap}
There is no $\WRP$ map $T$ that realises a
semilength-preserving map $F\colon\D\to\D$ satisfying
$\area(F(P))=\dinv(P)$ and $\dinv(F(P))=\area(P)$
for every Dyck path $P$.  In particular, no $\WRP$ map realises a
semilength-preserving \emph{bijection} of $\D$ swapping $\area$ and $\dinv$.
\end{theorem}

\begin{proof}
Assume such a $T$ exists.  For the Dyck input $W_n=U(UD)^nD$, the output
is $F(W_n)\in\D_{n+1}$.  Since $T$ realises $F$, every $W_n$ lies in the
domain of $T$, and $|T(W_n)|=2(n+1)=O(n)$.  By
Theorem~\ref{thm:wrp-slice-semilinearity}, the set of first-ascent pairs
\[
   \{(\fas(T(W_n)),\tailU(T(W_n))):n\ge 1\}
\]
is semilinear.  But $T(W_n)=F(W_n)$, so
Corollary~\ref{cor:model-free-obstruction} says that the same set is not
semilinear.  Contradiction.
\end{proof}

This proves the introduction-level statement, Theorem~\ref{thm:intro-no-swap}.
In particular, $\WRP$ contains the zeta map (Theorem~\ref{thm:zeta-wrp}) but no
semilength-preserving Catalan bijection that swaps area and dinv: it captures the
forward sweep without the full symmetry.

Theorem~\ref{thm:wrp-no-swap} does \emph{not} assert that no area--dinv
swap exists.  The symmetry of the $q,t$-Catalan polynomial implies that, for
each fixed $n$, some bijection of $\D_n$ exchanges the two statistics.  The
theorem says that no single $\WRP$ map can realise such bijections
simultaneously for all $n$.

\section[Inverse zeta lies outside WRP]
{A second separation: inverse zeta lies outside $\WRP$}
\label{sec:inverse-zeta}

The no-swap theorem shows that $\WRP$ cannot supply the full
$q,t$-symmetry.  We now turn to a different structural question under the
realisation convention: is $\WRP$-realisability of a Dyck-path bijection
preserved under inversion?  The zeta map provides a counterexample: it is a
semilength-preserving bijection in $\WRP$ (Theorem~\ref{thm:zeta-wrp}), whereas we
prove below that $\zeta^{-1}$ does not belong to $\WRP$.  Thus,
$\WRP$-realisability is not preserved under inversion, even among
semilength-preserving bijections of Dyck paths.

The proof follows the three-step pattern of the no-swap theorem, now on the
two-parameter family
\[
   W_{m,n}=U^m(UD)^nD^m\qquad(m\ge 1,\ n\ge 0)
\]
of Definition~\ref{def:two-parameter-family}.  Its subfamily with $m=1$ and $n\ge1$ is
the family $(W_n)_{n\ge1}$ studied in
Section~\ref{sec:wrapped-flat}.  Suppose, for a contradiction, that
$\zeta^{-1}\in\WRP$.  Then, it is defined on every $W_{m,n}$.  Since
$\zeta^{-1}$ is semilength-preserving,
Theorem~\ref{thm:two-parameter-semilinearity} implies that
\[
   S_{\zeta^{-1}}
   =\Bigl\{\bigl(m,n,\fas(\zeta^{-1}(W_{m,n})),
                    \tailU(\zeta^{-1}(W_{m,n}))\bigr):
                    m\ge1,\ n\ge0\Bigr\}
\]
is semilinear.  Second, Lemma~\ref{lem:inverse-zeta-fas} computes the first
ascent of $\zeta^{-1}(W_{m,n})$:
\[
   \fas\bigl(\zeta^{-1}(W_{m,n})\bigr)=\left\lceil\frac{m+n}{m+1}\right\rceil.
\]
Third, projecting $S_{\zeta^{-1}}$ away from the $\tailU$-coordinate would make
the set
\[
   G=\Bigl\{(m,n,f):m\ge1,\ n\ge0,\
          f=\Bigl\lceil\frac{m+n}{m+1}\Bigr\rceil\Bigr\}\subseteq\N^3
\]
semilinear.  Lemma~\ref{lem:inverse-zeta-not-semilinear} shows that this is impossible:
intersecting $G$ with the condition $f=m$ and projecting onto $(m,n)$ produces
the band $m^2-m\le n\le m^2$, whose quadratic lower boundary is impossible for
a semilinear set by Lemma~\ref{lem:semilinear-envelope}.  Therefore,
$\zeta^{-1}\notin\WRP$.

\subsection{The two-parameter family and its Presburger control}

This subsection proves the semilinearity constraint used above.  We introduce
the two-parameter family $(W_{m,n})_{m\ge1,\,n\ge0}$ and show that, as in
Section~\ref{sec:slice-semilinearity}, the selection, labelling, and ordering
conditions of a fixed $\WRP$ presentation are Presburger-definable on this
family.  Bounded counting then implies that,
for any $T\in\WRP$ defined on every $W_{m,n}$ with
$|T(W_{m,n})|=O(m+n)$, its first-ascent pairs, recorded together with $m$ and
$n$, form a semilinear set
(Theorem~\ref{thm:two-parameter-semilinearity}).

\begin{definition}[The family $W_{m,n}$]\label{def:two-parameter-family}
For $m\ge 1$ and $n\ge 0$, let
\[
   W_{m,n}=U^m(UD)^nD^m,
\]
which is a Dyck path of semilength $m+n$.  When $m=1$ and $n\ge1$, we have
$W_{1,n}=W_n$, so the family $(W_n)_{n\ge1}$ studied in
Section~\ref{sec:wrapped-flat} is a subfamily.  Allowing $n=0$ includes the
boundary paths $W_{m,0}=U^mD^m$, in particular $W_{1,0}=UD$.
\end{definition}

The area sequence of $W_{m,n}$ consists of $0,1,\ldots,m-1$, followed by
$n$ copies of $m$.  The first $m$ up-steps climb the initial staircase, and the
remaining $n$ up-steps all start at height $m$.  For the Presburger argument,
we encode positions by a triple
$(\tau,s,j)$, as in Section~\ref{sec:slice-semilinearity}.  The region tag
$\tau$ records whether the position lies in the initial stretch $U^m$, the
middle stretch $(UD)^n$, or the final stretch $D^m$; thus,
$\tau\in\{\mathsf i,\mathsf m,\mathsf f\}$.  The offset $s$ specifies a letter
within one copy of the corresponding block $U$, $UD$, or $D$, and the
repetition index $j$ identifies that copy.  Unlike a regular slice $uv^nz$,
whose prefix and suffix are fixed, this family has three repeated stretches:
the initial and final stretches contain $m$ copies of their blocks, and the
middle stretch contains $n$ copies.  The position is recovered from
$(\tau,s,j)$ as follows:
\begin{itemize}
\item $\tau=\mathsf i$: block $U$, so $s=1$; repetition index $1\le j\le m$; and
$i=(j-1)\cdot 1+s=j$;
\item $\tau=\mathsf m$: block $UD$, so $s\in\{1,2\}$ ($s=1$ for its $U$, $s=2$ for
its $D$); repetition index $1\le j\le n$; and $i=m+(j-1)\cdot 2+s$;
\item $\tau=\mathsf f$: block $D$, so $s=1$; repetition index $1\le j\le m$; and
$i=m+2n+(j-1)\cdot 1+s=m+2n+j$.
\end{itemize}
These mirror the regular-slice formulas
$i=s$, $i=|u|+(j-1)|v|+s$, and $i=|u|+n|v|+s$ from
Section~\ref{sec:slice-semilinearity}, with the fixed prefix and suffix there
replaced here by the repeated boundary stretches $U^m$ and $D^m$.  As before, $\tau$ and $s$
are finite tags fixed in advance, while the repetition index $j$ and the
parameters $m,n$ are the unbounded data.  The initial and final stretches are
governed by the same parameter $m$, while the middle stretch is governed by
$n$, so on each fixed choice of tags, the position $i$ is affine in $j,m,n$.

\begin{lemma}[Arithmetic on the two-parameter family]
\label{lem:two-parameter-presburger}
Fix a $\WRP$ presentation and consider the family
$\{W_{m,n}:m\ge 1,\ n\ge 0\}$.  For every potential atom under consideration,
fix its copy name and the region and offset of each coordinate.  Each of the
following is then Presburger-definable in $m,n$ and the remaining repetition
indices:
\begin{enumerate}[label=(\alph*)]
\item whether a given potential atom is selected;
\item whether a given selected atom has a specified output label;
\item whether the first of two selected atoms has the smaller atom rank;
\item whether the first of two selected atoms of equal atom rank comes first in
the tie-order.
\end{enumerate}
\end{lemma}

\begin{proof}
This is the two-parameter analogue of
Lemma~\ref{lem:one-loop-presburger}, and the proof follows the regular-slice
arguments of Lemmas~\ref{lem:one-loop-finite-state}
and~\ref{lem:one-loop-rank-affine}; the only change is that $W_{m,n}$ has
\emph{three} repeated stretches and \emph{two} unbounded parameters $m,n$, in
place of the single block $v$ and parameter $n$ there.  We use the position
encoding fixed just above.

\vspace{0.5em}
\noindent\emph{Overview.}  Each path $W_{m,n}$ is made of long runs of
identical steps, governed by just the two numbers $m$ and $n$.  Whatever a finite
automaton computes while scanning such a run eventually falls into a periodic
pattern, and any running integer total it keeps grows at a constant rate.  So
every decision a $\WRP$ presentation makes on $W_{m,n}$ depends on $m$, $n$, and
the repetition indices of the atoms involved only through linear expressions and
congruences.  In the language introduced in
Section~\ref{sec:slice-semilinearity}, these decisions are
Presburger-definable.  Turning
that intuition into the lemma's four conditions (a)--(d) is the whole proof.
These conditions are of two kinds: the finite-state conditions (a), (b), (d) ask
whether an $\MSO$ formula holds at the chosen positions, a question about
automaton acceptance, while the atom-rank condition (c) asks which of two atom
ranks is lexicographically smaller.  We first isolate the ingredient shared by
both kinds, the eventual
periodicity of an automaton along a repeated block, and then treat the two kinds
in turn, as parts (a),(b),(d) and part (c) below.

\vspace{0.5em}
\noindent\emph{Three block maps, each eventually periodic.}
We use the following repetition fact for both kinds.  The two parts below
scan $W_{m,n}$ with different machines: the $\MSO$ parts
(a), (b), (d) use a DFA over the \emph{marked} alphabet
$\Sigma\times\{0,1\}^k$ (the marking construction recalled in
Section~\ref{sec:slice-semilinearity}), while the atom-rank part~(c) uses an additive
rank source over the plain alphabet $\Sigma$.  In either case, let $Q$ be its
state set, $\delta^*$ its transition extension, and $\delta_y(q)=\delta^*(q,y)$
the state reached on reading $y$ from $q$.  For the marked DFA, call an
occurrence of a repeated block \emph{marked} if at least one chosen position
lies in it, equivalently if some marker bit on that occurrence is $1$; call it
\emph{unmarked} otherwise.  Reading one unmarked occurrence of each repeated
block induces three fixed transformations of $Q$.  Every letter in such an
occurrence carries marker bits $0$, so $\delta_U$ abbreviates reading the
marked letter $(U,0^k)$, and likewise for $\delta_{UD}$ and $\delta_D$.  For
the rank source, each subscript simply denotes the corresponding plain block.
The three transformations are
\[
   g_{\mathsf i}=\delta_U,\qquad g_{\mathsf m}=\delta_{UD},\qquad
   g_{\mathsf f}=\delta_D,
\]
one per region $\tau\in\{\mathsf i,\mathsf m,\mathsf f\}$, so reading $L$
consecutive unmarked block occurrences in a stretch applies the single map
$g_\bullet^{\,L}$.  There are at most $k$ marked block occurrences in the
$\MSO$ run.  Their transformations need not be powers of $g_\bullet$; as in
Lemma~\ref{lem:one-loop-finite-state}, they are absorbed into the fixed
transformations of the factorisation in parts (a), (b), (d) below.  Each
$g_\bullet$ lies in the finite monoid $Q^Q$, hence its powers are
eventually periodic; taking the larger threshold and a common multiple of the
periods gives a single $k_0$ and $p$ with $g_\bullet^{\,L+p}=g_\bullet^{\,L}$
for $L\ge k_0$ and all three maps.  The reduced exponent
$\widehat L\in\{0,\ldots,k_0+p-1\}$ of Lemma~\ref{lem:one-loop-finite-state}
then satisfies $g_\bullet^{\,L}=g_\bullet^{\,\widehat L}$, and ``$\widehat L=c$''
is a Presburger-definable condition on $L$.  This eventual periodicity, captured by the
Presburger-definable reduced exponent, is the only way the machines' state
behaviour along the repeated stretches enters either part below.

\vspace{0.5em}
\noindent\emph{Parts (a), (b), (d): the finite-state conditions.}
Selection, labelling, and the tie-order are each an $\MSO$ formula
$\varphi(x_1,\ldots,x_k)$ with $k$ free position variables, and we must show that
the tuples of parameters for which $\varphi$ holds at the chosen positions form a
Presburger-definable set.  The marked DFA $M$ over $\Sigma\times\{0,1\}^k$ of the setup
above accepts, by the marking construction, exactly when
$W_{m,n}\models\varphi(i_1,\ldots,i_k)$: it reads $W_{m,n}$ with the letters at
the chosen positions $i_1,\ldots,i_k$ flagged by the $k$ marker bits.  Write
$\widetilde W_{m,n}$ for this \emph{marked word}.  On this word, the state
reached by $M$ from its start state $q_0$ is
$\delta_{\widetilde W_{m,n}}(q_0)$.  The formula $\varphi$ holds at the chosen
positions if and only if this state belongs to $F$.  By the encoding, each marked
variable $x_\ell$ has a fixed region $\tau_\ell$ and offset $s_\ell$ and an
unbounded repetition index $j_\ell$ (in $\{1,\ldots,m\}$, $\{1,\ldots,n\}$, or
$\{1,\ldots,m\}$ according to $\tau_\ell$), and its position $i_\ell$ is the
affine form in $j_\ell,m,n$ recorded there; the unbounded data is the tuple
$(m,n,(j_\ell)_\ell)$.

The goal is now concrete: show that the tuples $(m,n,(j_\ell)_\ell)$ with
$\delta_{\widetilde W_{m,n}}(q_0)\in F$ form a Presburger-definable set.  We
proceed in three steps.  First, for every pair of variables assigned to the same
stretch, we distinguish the cases $j_\ell<j_{\ell'}$,
$j_\ell=j_{\ell'}$, and $j_\ell>j_{\ell'}$.  These comparisons determine
whether the two chosen positions lie in the same occurrence of the repeated
block and, if not, which occurrence is read first.  Second, in each resulting
case, we factor the marked word into the block occurrences containing chosen
positions and the intervening runs of unmarked block occurrences.  The
transformations associated with the marked occurrences are fixed, while only
the lengths of the unmarked runs vary.  Third, for each intervening run of $L$
unmarked block occurrences, we subdivide the case according to the bounded
value $\widehat L$ defined above, which satisfies
$g_\bullet^{\,L}=g_\bullet^{\,\widehat L}$.  After this subdivision, the
composite transition, and hence the truth of $\varphi$, is fixed on each
subcase.

\vspace{0.5em}
\noindent\emph{Splitting into cells.}  Recall that the regions $\tau_\ell$ and
offsets $s_\ell$ are finite tags fixed in advance, so the only unbounded data is
$(m,n,(j_\ell)_\ell)$.  For each pair of variables in the same stretch, split
the parameter space according to which of $j_\ell<j_{\ell'}$,
$j_\ell=j_{\ell'}$, and $j_\ell>j_{\ell'}$ holds.  For example, if two
variables lie in the middle stretch $(UD)^n$, equality means that their chosen
positions lie in the same occurrence of $UD$, while $j_\ell<j_{\ell'}$ means
that the occurrence containing $x_\ell$ is read before the occurrence
containing $x_{\ell'}$.  The order between different stretches is already
fixed: the initial stretch $U^m$ is read first, then $(UD)^n$, and then $D^m$.
Each consistent choice of the pairwise comparisons is given by a conjunction
of linear (in)equalities and therefore defines a Presburger-definable cell.  On
such a cell, equal repetition indices identify the variables whose marker bits
occur in the same block occurrence, and strict inequalities fix the order of
the marked block occurrences.  Because the offsets are also fixed, we know
exactly which marker bits are attached to each letter in every such block
occurrence.  The cells cover the whole parameter space, so it suffices to show
that the accepting tuples inside one cell $C$ form a Presburger-definable set.

\vspace{0.5em}
\noindent\emph{Factoring the run on $C$.}  By the previous step, the order of
the block occurrences containing chosen positions and the marker bits attached
to their letters are fixed on $C$.  Reading $\widetilde W_{m,n}$ from left to right therefore alternates
between maximal runs of unmarked block occurrences and individual marked block
occurrences.  A run of $L$ unmarked occurrences of one block applies the
corresponding power $g_{\mathsf i}^{\,L}$, $g_{\mathsf m}^{\,L}$, or
$g_{\mathsf f}^{\,L}$ of a block map.  A single marked block occurrence applies
one fixed transformation $M\in Q^Q$: the cell determines which chosen
positions lie in that occurrence, and their offsets $s_\ell$ are fixed, so the
letters and marker bits in the occurrence are completely determined.  Writing
$t\le k$ for the number of marked block occurrences, composing these maps in
the order the input is scanned
gives
\[
   \delta_{\widetilde W_{m,n}}
   = g_\bullet^{\,L_t}\circ M_t\circ\cdots\circ M_1\circ g_\bullet^{\,L_0}
\]
(later factors on the left, since reading a concatenation composes the maps),
where $M_1,\ldots,M_t$ are the fixed transformations associated with the
marked block occurrences and
$g_\bullet^{\,L_0},\ldots,g_\bullet^{\,L_t}$ are the unmarked runs surrounding
them; any of the exponents $L$ may be $0$.  Each exponent $L$ is the length of one such run, a
nonnegative linear form in $m,n$ and the repetition indices: within each stretch, the
runs are the gaps before, between, and after its marked block occurrences, summing to the
stretch length ($m$ for the initial and final stretches, $n$ for the middle)
minus the number of marked block occurrences in it.

\vspace{0.5em}
\noindent\emph{Making the run constant.}  Refine $C$ by the reduced exponent $\widehat L$
of each of these finitely many run-lengths, adjoining the Presburger-definable condition
``$\widehat L=c$'' for each.  This splits $C$ into finitely many sub-cells on
each of which every $g_\bullet^{\,L}=g_\bullet^{\,\widehat L}$ is one fixed power
of a block map, so the composite above is a single fixed map $\sigma\in Q^Q$.
Then, $\delta_{\widetilde W_{m,n}}(q_0)=\sigma(q_0)$ is constant on the sub-cell,
and the acceptance test $\sigma(q_0)\in F$ has the same answer
everywhere on it.  The accepting tuples are thus the union of the sub-cells on
which $\sigma(q_0)\in F$, a finite union of Presburger-definable sets.  This settles (a),
(b), and (d).

\vspace{0.5em}
\noindent\emph{Part (c): atom-rank comparison.}
The aim is to decide which of two atoms receives the smaller atom rank.  The
point is that on $W_{m,n}$, every atom rank is an affine vector-valued function
of $m$, $n$, and the
repetition indices, once the parameters are split into finitely many cases by threshold
tests and residues; comparing two affine forms is then expressible 
by a Presburger formula.

\vspace{0.5em}
\noindent\emph{One coordinate source at a time.}  Recall
(Definitions~\ref{def:rank-source} and~\ref{def:prefix-rank})
that a rank source is a finite-state scanner that, as it reads each letter $a$
in a state $q$, also adds an integer vector $\omega(q,a)\in\Z^d$ and reports
the running total;
the \emph{prefix rank} before a position is that total accumulated over the
letters strictly before it, the abstract form of the height.  A
prefix-additive rank function
(Definition~\ref{def:prefix-additive-rank}) adds one such contribution, with a
bounded local correction, for each tuple coordinate.  It is therefore enough
to take one coordinate source $A=(Q,q_0,\delta,\omega)$ and show that its
prefix rank before a position is affine in $m$, $n$, and that position's
repetition index, after a finite split into Presburger-definable cases; the finitely many coordinate
contributions can then be summed.

\vspace{0.5em}
\noindent\emph{Weight accumulated along a stretch is affine.}  Reading one occurrence of a block
from a state $q$ moves $q$ by the corresponding block map ($g_{\mathsf i}$,
$g_{\mathsf m}$, or $g_{\mathsf f}$) and adds a weight $\Omega_\tau(q)\in\Z^d$, the
total of $\omega$ over that block occurrence.  So if a stretch is entered in a state $q$, then
after $\nu$ completed block occurrences, the state is
$g_\bullet^{\,\nu}(q)$, which is eventually periodic in $\nu$: there are a threshold and a period so that, once $\nu$ passes the
threshold, the entering state depends only on $\nu$ modulo the period, while the
finitely many $\nu$ below the threshold are separate.  Call each possibility a
\emph{case}; it is either a value below the threshold or a residue modulo the
period above it.  Fix one such case.  Each completed period then adds the same
fixed weight increment, and the bounded leftover is fixed.  Hence the weight
of the first $\nu$ block occurrences is an affine function of $\nu$ on the
case, with rational coefficients and integer values.  (This is
the additive counterpart of the eventual periodicity used for the $\MSO$
parts: the states cycle, and the running weight grows by a fixed amount per
cycle.)

\vspace{0.5em}
\noindent\emph{Prefix rank by stretch.}  Applying this to each stretch, the prefix rank
before a position $(\tau,s,j)$ is, after splitting $m,n,j$ into finitely many such
cases:
\begin{itemize}
\item $\tau=\mathsf i$, the $j$th block occurrence: the prefix is $U^{j-1}$, so the prefix rank is the
weight of $j-1$ occurrences of $U$ read from $q_0$, affine in $j$ on each case of $j$
(a residue modulo the period of $g_{\mathsf i}$, above a threshold);
\item $\tau=\mathsf m$, offset $s$ in the $j$th block occurrence: the prefix is $U^m$, then
$(UD)^{j-1}$, then the fixed length-$(s-1)$ start of $UD$; the $U^m$ part is affine
in $m$, the $(UD)^{j-1}$ part affine in $j$ (read from the state after $U^m$, which
is fixed on the case of $m$), and the within-block part depends only on $s$
and that entering state, so is bounded and fixed on the case; the prefix rank
is thus affine in $m$ and $j$;
\item $\tau=\mathsf f$, the $j$th block occurrence: the prefix is $U^m$, then $(UD)^n$, then
$D^{j-1}$, whose three completed stretches contribute weights affine in $m$, in
$n$, and in $j$, so the prefix rank is affine in $m,n,j$.
\end{itemize}
The bounded local corrections of Definition~\ref{def:prefix-additive-rank} are
likewise constant on each such case, depending only on the fixed offset,
letter, and entering state.  Summing the finitely many coordinate contributions
and the fixed constant $c_0$ as in that definition, the atom rank is one affine
form in $m,n$ and the relevant repetition indices on each case.

\vspace{0.5em}
\noindent\emph{Comparing two atom ranks.}  For two atoms, refine to a common case; both
atom ranks are then affine vectors in $m,n$ and the repetition indices $j_\ell$.
Whether one is lexicographically smaller than the other is a Boolean
combination of affine equalities and inequalities and the congruences defining
the case.  Clearing the fixed denominators makes all coefficients integral, so
the comparison is expressible by a Presburger formula.  This
settles (c).
\end{proof}

\begin{theorem}[Semilinearity on the two-parameter family]
\label{thm:two-parameter-semilinearity}
Let $T\in\WRP$ be fixed, with
$W_{m,n}\in\operatorname{dom}(T)$ for all $m\ge1$ and $n\ge0$, and suppose
$|T(W_{m,n})|=O(m+n)$.  Then,
\[
   S_T=\bigl\{(m,n,\fas(T(W_{m,n})),\tailU(T(W_{m,n}))):m\ge 1,\ n\ge 0\bigr\}
   \subseteq\N^4
\]
is semilinear.
\end{theorem}

\begin{proof}
This is the two-parameter analogue of
Theorem~\ref{thm:wrp-slice-semilinearity}; here is the outline before the
details.  We must show that the first-ascent pairs
$(\fas(T(W_{m,n})),\tailU(T(W_{m,n})))$, recorded together with the parameters
$(m,n)$ they come from, form a \emph{semilinear} set $S_T$.
Two features of the family make this happen.  First, $W_{m,n}=U^m(UD)^nD^m$ is so
repetitive that every decision the $\WRP$ presentation makes on it (which atoms
it selects, how it labels them, how it orders them by atom rank) is governed by plain
linear arithmetic in $m$, $n$, and the repetition indices of the atoms: after
the finite tags of the atoms are fixed, each such condition is
Presburger-definable by Lemma~\ref{lem:two-parameter-presburger}.  The finitely
many choices of those tags can then be combined by finite unions.  Second,
the two numbers $\fas$ and $\tailU$
are just \emph{counts} of selected atoms, and the size hypothesis
$|T(W_{m,n})|=O(m+n)$ keeps each count below a fixed multiple of $m+n$.  Now
counting the solutions of such a Presburger-definable family when it has only linearly many
solutions yields a semilinear set (the bounded counting principle,
Lemma~\ref{lem:presburger-counting}), so assembling the two counts, while keeping
$m,n$, makes $S_T$ semilinear.  The labelled steps below carry this out, with the
two parameters $m,n$ in place of the single $n$ of
Theorem~\ref{thm:wrp-slice-semilinearity}.

\vspace{0.5em}
\noindent\emph{Atoms on the family.}  We begin by naming the output
items the counts range over and stating the one fact we use about them on the
family.  A \emph{potential atom} is a pair
$(c,\bar\imath)$ of a copy name $c$ and a tuple $\bar\imath$ of input positions of
$c$'s arity.  It becomes a \emph{selected atom}, contributing one output
letter, exactly when it satisfies the corresponding selection formula.  On
the family, a potential atom is named by the
finite data of its copy name and the region and offset of each coordinate, together
with the parameters $m,n$ and the repetition indices of its coordinates (the encoding
above).

For quantification and counting below, we use one common encoding for all
potential atoms.  Let $K$ be the largest arity among the finitely many copy
names.  A bounded tag records the copy name and the region and offset of each
coordinate; the repetition indices occupy $K$ further coordinates, with the
unused coordinates set to $0$.  Thus, every potential atom has a unique code of
the same fixed dimension.  The validity of such a code is expressed by the
bounds $1\le j\le m$ or $1\le j\le n$ appropriate to each region, together
with the requirement that unused coordinates equal $0$, and is therefore
Presburger-definable.  For each fixed tag, Lemma~\ref{lem:two-parameter-presburger}
makes selection, labelling, and the two ordering comparisons
Presburger-definable in $m,n$ and the repetition indices.  Taking the finite
union over all tags therefore gives the same Presburger-definable predicates on
the common codes.  This is the form of Lemma~\ref{lem:two-parameter-presburger}
used in the remainder of the proof.

\vspace{0.5em}
\noindent\emph{The two counts.}  Concretely, the two first-ascent
numbers are separated by the output's first $D$-atom: $\fas(T(W_{m,n}))$ counts
the selected $U$-atoms emitted before it, and $\tailU(T(W_{m,n}))$ those emitted
after.  We realise each as a \emph{linearly bounded Presburger count}: the number
of selected atoms cut out by a Presburger-definable condition, of which there are only
$O(m+n)$.  Then, we treat in turn the case where the output has no $D$-atom and the
case where it has one.

\vspace{0.5em}
\noindent\emph{Splitting on the first descent.}  We first observe that having
a descent is a Presburger-definable condition in $m,n$, and then handle the
no-descent case, where $\tailU=0$ and only the count $\fas$ remains.  Whether $T(W_{m,n})$
has any $D$-atom is a Presburger-definable condition on $(m,n)$: using the
common encoding above, it is the existential projection of the set of
potential atoms that are selected and labelled $D$.  Split on this condition.  If there
is none, every selected atom is a $U$, so
$\fas(T(W_{m,n}))=|\{\text{selected }U\text{-atoms}\}|$ and
$\tailU(T(W_{m,n}))=0$.  The family
\[
   A=\{(m,n,u):u\text{ a potential atom, selected and labelled }U\text{ over }W_{m,n}\}
\]
is Presburger-definable in $(m,n)$ and the common code $u$.  For each fixed
$(m,n)$, the corresponding set of $u$-values has size
$\fas(T(W_{m,n}))\le|T(W_{m,n})|=O(m+n)$, linear in the
parameters.  Lemma~\ref{lem:presburger-counting} thus makes
$\{(m,n,\fas(T(W_{m,n}))):T(W_{m,n})\text{ has no }D\}$ semilinear, and appending
the constant coordinate $\tailU=0$ keeps it semilinear in $\N^4$.

\vspace{0.5em}
\noindent\emph{The two $U$-counts when a descent occurs.}  Suppose that
$T(W_{m,n})$ has a descent.  Its first $D$-atom separates the selected
$U$-atoms into those occurring before it, counted by $\fas$, and those
occurring after it, counted by $\tailU$.  We express both quantities as
linearly bounded Presburger counts.  To parameterize the first $D$-atom, fix
the finite data of a potential atom $d$: its copy name and the region and offset
of each coordinate.  With these data fixed, each tuple
$x=(m,n,\text{repetition indices of }d)$ specifies one potential atom $d$ over
$W_{m,n}$.  Let $P$ consist of the tuples whose specified atom is the first
selected $D$-atom in $T(W_{m,n})$.  By
Lemma~\ref{lem:two-parameter-presburger}, the conditions that $d$ is selected
and labelled $D$ are expressible by Presburger formulas.  The condition that
no selected $D$-atom precedes $d$ is obtained by negating an existential
Presburger formula over the common code of a potential atom.  Since Presburger
formulas are closed under conjunction, existential quantification, and
negation, $P$ is Presburger-definable.  There are only finitely many choices for the finite data
of $d$, and we take their union at the end.  Split the selected $U$-atoms by
their side of $d$,
\[
   U_{<d}=\{u:u\text{ selected, labelled }U,\ u\prec d\},\qquad
   U_{>d}=\{u:u\text{ selected, labelled }U,\ d\prec u\}.
\]
Since $\prec$ totally orders the selected atoms and the $D$-atom $d$ lies in
neither set, every selected $U$-atom falls in exactly one of them, and
$\fas(T(W_{m,n}))=|U_{<d}|$, $\tailU(T(W_{m,n}))=|U_{>d}|$.  Apply
Lemma~\ref{lem:presburger-counting} with parameters $x$ and the common code of
a $U$-atom $u$ as the unknown tuple: each membership ``$u\in U_{<d}$'' and
``$u\in U_{>d}$'' is a Boolean combination of the predicates of
Lemma~\ref{lem:two-parameter-presburger} (``selected'', ``labelled $U$'', and the
comparison $u\prec d$ or $d\prec u$ against the atom named by $x$), hence
Presburger-definable in $(x,u)$.  For each fixed $x$, both resulting sets of
$u$-values have size at most the number of selected atoms, $O(m+n)$, linear in
$x$.  The tuple form of
Lemma~\ref{lem:presburger-counting} then makes $\{(x,|U_{<d}|,|U_{>d}|)\}$
semilinear; intersecting with $P$ restricts this set to the parameter tuples in
$P$ and preserves semilinearity.

\vspace{0.5em}
\noindent\emph{Assembling $S_T$.}  Finally, we combine the no-descent and
descent cases to form $S_T$.  The previous step produced the semilinear set
$\{(x,|U_{<d}|,|U_{>d}|):x\in P\}$.  For every tuple in this set, $x$
contains the parameters $m,n$ and specifies the first selected $D$-atom $d$.
Hence
\[
   \fas(T(W_{m,n}))=|U_{<d}|,\qquad
   \tailU(T(W_{m,n}))=|U_{>d}|.
\]
Projecting away the repetition indices of $d$ from $x$ therefore gives the
semilinear set of quadruples
$(m,n,\fas(T(W_{m,n})),\tailU(T(W_{m,n})))$ whose first selected $D$-atom $d$ has
the fixed finite data chosen above.  Taking the finite union over all possible
finite data of $d$ gives all such quadruples when the output has a descent.
Finally, taking the union with the no-descent case gives $S_T$, which is
semilinear.
\end{proof}

\subsection{A nonsemilinear set of first-ascent pairs for inverse zeta}

The previous subsection put a semilinear constraint on every linear-growth
$\WRP$ map: its set of first-ascent pairs on the family $(W_{m,n})$ has to be
semilinear.  To separate $\zeta^{-1}$ from the class, we now compute this set
for $\zeta^{-1}$ and show that it violates the constraint.

Ceballos, Denton, and Hanusa give a new method for inverting classical zeta,
explicit inverses for certain rational families, and inductive methods for
additional families~\cite{CeballosDentonHanusa2016}.  Pons gives an iterative
description of classical inverse zeta directly on area
sequences~\cite[Theorem~6 and Section~2.3]{Pons2022}.  For the structured
targets $W_{m,n}$ used here, the inverse takes the particularly simple
balanced-pyramid form described below.

Inverting $\zeta$ head-on looks forbidding, but the \emph{preimage} of $W_{m,n}$
is transparent when read backwards: we guess a path $Q$ and check
$\zeta(Q)=W_{m,n}$, which pins down $Q=\zeta^{-1}(W_{m,n})$ because $\zeta$ is a
bijection.  The guess is a row of $m+1$ \emph{pyramids}
\[
   Q=U^{\ell_1}D^{\ell_1}\;U^{\ell_2}D^{\ell_2}\cdots U^{\ell_{m+1}}D^{\ell_{m+1}},
\]
of heights $\ell_1\ge\cdots\ge\ell_{m+1}$ chosen as equal as possible (so any two
differ by at most one).  Each pyramid $U^{\ell}D^{\ell}$ of $Q$ contributes the
staircase $(0,1,\ldots,\ell-1)$ to the area sequence, so the area sequence of $Q$
is these $m+1$ staircases laid end to end.  The zeta scan then rebuilds $W_{m,n}$
from $Q$ one rank at a time: rank $0$ catches the foot of every nonempty
pyramid and lays down the initial run of $U$-steps, the intermediate ranks emit
the repeated $DU$ pairs of the middle, and the top ranks emit the closing
down-steps.  Because the pyramids
are balanced, the first (and tallest) has height
$\ell_1=\lceil(m+n)/(m+1)\rceil$; its ascent $U^{\ell_1}$ is the initial ascent of
$Q=\zeta^{-1}(W_{m,n})$, so $\fas(\zeta^{-1}(W_{m,n}))=\lceil(m+n)/(m+1)\rceil$.
As Lemma~\ref{lem:inverse-zeta-not-semilinear} will show, this quotient is the
source of the nonsemilinearity.

\begin{example}[The balanced-block preimage]
Let $m=2$ and $n=5$, so $N=7$ and $\lceil N/(m+1)\rceil=\lceil 7/3\rceil=3$.  The
$m+1=3$ pyramids then have balanced heights $3,2,2$, so
\[
   Q=\zeta^{-1}(W_{2,5})=U^3D^3\;U^2D^2\;U^2D^2
   =UUUDDD\,UUDD\,UUDD,
\]
with area sequence $(0,1,2,\ 0,1,\ 0,1)$ and initial ascent $U^3$ of length $3$.
Running the zeta scan on this area sequence (the values reach $2$, so the scan
levels are $0,1,2,3$): level $0$ emits a $U$ at each $0$, giving $U^3$; level $1$ emits
$DU$ at each $0,1$, giving $(DU)^3$; level $2$ emits $DU$ from the first pyramid
(its values $1,2$) and a lone $D$ from each shorter pyramid (its value $1$),
giving $DUD^2$; and level $3$ emits a $D$ at the single $2$, giving $D$.
Concatenating,
\[
   \zeta(Q)=U^3\,(DU)^3\,DUD^2\,D=U^3(DU)^4D^3=U^2(UD)^5D^2=W_{2,5},
\]
which confirms $Q=\zeta^{-1}(W_{2,5})$ and that its first ascent has length $3$.
\end{example}

\begin{lemma}[Inverse-zeta ascent]
\label{lem:inverse-zeta-fas}
For $m\ge 1$ and $n\ge 0$,
\[
   \fas\bigl(\zeta^{-1}(W_{m,n})\bigr)=\left\lceil\frac{m+n}{m+1}\right\rceil.
\]
\end{lemma}

\begin{proof}
\emph{Constructing the preimage.}
Write $N=m+n$ and $b=\lceil N/(m+1)\rceil$.  As in the balanced-pyramid
description above, distribute $N$ as evenly as possible among $m+1$ pyramid
heights.  Set
\[
   s:=N-(m+1)(b-1),
   \qquad
   \ell_t=
   \begin{cases}
      b,   & 1\le t\le s,\\
      b-1, & s<t\le m+1,
   \end{cases}
\]
where $1\le s\le m+1$, and define
\[
   Q=\prod_{t=1}^{m+1}U^{\ell_t}D^{\ell_t}
    =(U^bD^b)^s(U^{b-1}D^{b-1})^{m+1-s}.
\]
The sum of the $\ell_t$ is $N$, and each factor
$U^{\ell_t}D^{\ell_t}$ is a Dyck path.  Thus, $Q\in\D_N$.  Its area sequence
is the concatenation
\[
   B_1B_2\cdots B_{m+1},\qquad B_t=(0,1,\ldots,\ell_t-1),
\]
where $B_t$ is empty when $\ell_t=0$.  Since the first pyramid has height
$\ell_1=b$, the initial ascent of $Q$ has length $b$.  It
remains to verify, scan level by scan level, that $\zeta(Q)=W_{m,n}$.  Since $\zeta$ is a
bijection on $\D_N$, this verification will identify
$Q=\zeta^{-1}(W_{m,n})$ and prove the claimed formula for $\fas$.

If $n\le 1$ then $N\le m+1$, so $b=1$; every block then has length at most $1$,
the area sequence is all zeros, and $Q=(UD)^N$ with $\zeta(Q)=U^ND^N$.  This is
$W_{m,n}$ (namely $U^mD^m$ when $n=0$ and $U^{m+1}D^{m+1}$ when $n=1$).

\vspace{0.5em}
\noindent\emph{Checking the zeta scan.}
Assume now $n\ge 2$, so $b\ge 2$ and every block is nonempty.  Recall the scan
(Definition~\ref{def:zeta}): for each scan level $r=0,1,\ldots$, read the area sequence
left to right, emitting a $U$ at every entry equal to $r$ and a $D$ at every entry
equal to $r-1$.  The largest area value here is $b-1$, so the scan levels run
$r=0,1,\ldots,b$.  Within a block, the values $0,1,\ldots,\ell_t-1$ appear in
increasing order.  Thus, at a fixed scan level $r\ge 1$, a block $B_t$ contributes, in scan
order, a $D$ for value $r-1$ (present if and only if $\ell_t\ge r$) immediately
followed by a $U$ for value $r$ (present if and only if $\ell_t\ge r+1$).  Recall
that the $s$ length-$b$ blocks come first, followed by the $m+1-s$
length-$(b-1)$ blocks.
Hence:
\begin{itemize}
\item scan level $0$ contributes a $U$ from every block: $U^{m+1}$;
\item each scan level $r$ with $1\le r\le b-2$ has $\ell_t\ge b-1\ge r+1$ for
every block, contributing $(DU)^{m+1}$;
\item scan level $b-1$ contributes $DU$ from each length-$b$ block (value $b-1$
present) and a lone $D$ from each length-$(b-1)$ block (value $b-1$
absent): $(DU)^sD^{\,m+1-s}$;
\item scan level $b$ contributes a lone $D$ from each length-$b$ block: $D^s$.
\end{itemize}
Concatenating,
\[
   \zeta(Q)=U^{m+1}\,(DU)^{(m+1)(b-2)}\,(DU)^sD^{\,m+1-s}\,D^s
          =U^{m+1}(DU)^{k}D^{m+1},\qquad k=(m+1)(b-2)+s.
\]
Using $U^{m+1}(DU)^k=U^m(UD)^kU$ and absorbing one trailing $D$,
\[
   \zeta(Q)=U^m(UD)^k\,UD\,D^m=U^m(UD)^{k+1}D^m.
\]
Finally, substituting $s=N-(m+1)(b-1)$ and using $N=m+n$,
\[
   k+1=(m+1)(b-2)+s+1
      =(m+1)(b-2)-(m+1)(b-1)+N+1
      =N-(m+1)+1=n.
\]
Thus, $\zeta(Q)=U^m(UD)^nD^m=W_{m,n}$.
\end{proof}

\begin{lemma}[Nonsemilinear ascent graph]
\label{lem:inverse-zeta-not-semilinear}
The set
\[
   G=\Bigl\{(m,n,f):m\ge 1,\ n\ge 0,\
                    f=\bigl\lceil\tfrac{m+n}{m+1}\bigr\rceil\Bigr\}\subseteq\N^3
\]
is not semilinear.
\end{lemma}

\begin{proof}
Suppose $G$ were semilinear.  Semilinear sets are closed under intersection
with the linear set $\{f=m\}$ and under coordinate projection, so
\[
   K=\Bigl\{(m,n):\bigl\lceil\tfrac{m+n}{m+1}\bigr\rceil=m\Bigr\}\subseteq\N^2
\]
would be semilinear.  For $m\ge 1$, the equation
$\lceil (m+n)/(m+1)\rceil=m$ is equivalent to
$(m-1)(m+1)<m+n\le m(m+1)$, i.e.\ to $m^2-m\le n\le m^2$.  Hence each
vertical section $K_m=\{m^2-m,\ldots,m^2\}$ is a finite interval, and the
lower envelope is
\[
   \ell(m)=\min\{n:(m,n)\in K\}=m^2-m.
\]
By applying Lemma~\ref{lem:semilinear-envelope} with the two coordinates
interchanged, a semilinear subset of $\N^2$ with finite sections $K_m$ would
satisfy the two alternatives in that lemma.  Since every $K_m$ is nonempty for
$m\ge1$, alternative~(i) is impossible on every residue class.  Thus, the lower
envelope would be eventually affine on every residue class modulo some period
$M$.  But $\ell(m)=m^2-m$ is quadratic, so it is not eventually affine on any
infinite arithmetic progression.  Contradiction.
\end{proof}

\begin{corollary}[Inverse zeta outside $\WRP$]
\label{cor:inverse-zeta-not-wrp}
The inverse zeta map $\zeta^{-1}\colon\D\to\D$, under the step-word
encoding, does not belong to $\WRP$.
\end{corollary}

\begin{proof}
Since $\zeta$ restricts to a bijection $\D_N\to\D_N$ for every $N$, its inverse
preserves semilength.  Hence $|\zeta^{-1}(W_{m,n})|=2(m+n)=O(m+n)$.  If
$\zeta^{-1}\in\WRP$, then it is defined on every $W_{m,n}$, and
Theorem~\ref{thm:two-parameter-semilinearity} makes
\[
   S=\bigl\{(m,n,\fas(\zeta^{-1}(W_{m,n})),\tailU(\zeta^{-1}(W_{m,n}))):
              m\ge 1,\ n\ge 0\bigr\}
\]
semilinear; projecting away the $\tailU$-coordinate, the set
$\{(m,n,\fas(\zeta^{-1}(W_{m,n}))):m\ge1,\ n\ge0\}$ is semilinear.  By
Lemma~\ref{lem:inverse-zeta-fas}, however, that set is exactly the set $G$ of
Lemma~\ref{lem:inverse-zeta-not-semilinear}, which is not semilinear.
Contradiction.
\end{proof}

This proves the introduction-level statement,
Theorem~\ref{thm:intro-inverse-zeta}.
Together with Theorem~\ref{thm:zeta-wrp}, the corollary also shows that
$\WRP$-realisability under the realisation convention is not preserved by
inversion, even for bijections of Dyck paths that preserve semilength.

\section{Machine-checked formalisation in Lean}\label{sec:lean}

The results of this paper have been formalised and machine-checked in the
Lean~4 proof assistant~\cite{DeMouraUllrich2021}, building on the Mathlib
library~\cite{MathlibCommunity2020}.  The development was produced
automatically by the AI models Claude, Codex, and Aristotle, as recorded in
the declaration of AI use at the end of the paper.  The proofs are checked by
Lean regardless of how they were found.  The formalisation is distributed with
the paper.\footnote{\url{https://github.com/hongseok-yang/automata-catalan-symmetry-release}}

The development comprises roughly $82{,}000$ lines of Lean in $162$ source
files.  It pins an exact toolchain (Lean~4 v4.33.1 and the matching Mathlib
release) so that the build is reproducible, and it contains no \texttt{sorry},
Lean's placeholder for an unproved statement.  Every mathematical result is
proved from Lean's kernel rules and the short list of admitted axioms described
below.  The numerical worked examples involving $\area$, $\dinv$, $\zeta$, and
the sweep maps are instead discharged by compiled evaluation
(\texttt{native\_decide}) and therefore additionally trust the Lean compiler;
none of the results in Table~\ref{tab:lean} relies on compiled evaluation.  This
section describes the scope of the formalisation, its trust base, and the
places where its statements differ from those in the paper.

\paragraph{Scope.}
The development formalises the theory of the paper from the ground up rather
than verifying an isolated statement.  On the combinatorial side, it defines
Dyck paths, area sequences, the statistics $\area$, $\dinv$, and $\coarea$,
the zeta map of
Definition~\ref{def:zeta}, the height sweep $H$ of
Section~\ref{sec:narayana-sweep}, and the families $W_n$, $W_{m,n}$, and
$P_{m,n}$.
On the computational side, it defines $\MSO$ on words together with the
marked-word encodings of free variables (Section~\ref{sec:mso}); polyregular
presentations (Definition~\ref{def:polyregular}) and deterministic two-way
finite-state transducers (Definition~\ref{def:2dft}); rank sources,
prefix-additive rank functions, and the class $\WRP$ with its fragments $\RR$
and $\SWR$ (Definitions~\ref{def:rank-source}--\ref{def:wrp}); and two explicit
machine models for logspace transductions, the deterministic multihead
bounded-counter transducer and the deterministic worktape transducer.
On this base, the development proves formal counterparts of the paper's
headline results, subject to the qualifications under ``Statement fidelity''
below.  Table~\ref{tab:lean} gives their Lean names, and the supporting lemmas
are formalised as well.  In particular, the deficit-zero analysis of
Section~\ref{sec:main}, the
first-ascent computations of Section~\ref{sec:inverse-zeta}, the two-pyramid
closed forms behind the $\zeta$ and $H$ lower bounds
(Lemmas~\ref{lem:zeta-two-pyramid} and~\ref{lem:H-two-pyramid}), and the
regular-slice lemmas of Section~\ref{sec:slice-semilinearity} all have formal
proofs.  The regular-slice lemmas are proved both for the family $W_n$ and in
the general form used in the paper, over an arbitrary slice $u\,v^n z$.

Within the Lean development, background results quoted from the literature are
handled in three ways: some are admitted as axioms, some are proved in
Lean or Mathlib, and others are bypassed because the formal proof takes a
different route.  The paragraphs ``Trust base'' and ``What is proved or
avoided'' explain these three cases.

Here an ``admitted axiom'' means an assumption encoded inside Lean and reported
by the \texttt{\#print axioms} command.  Theorem~\ref{thm:eh} is not admitted
in this sense: the Lean declarations prove the relevant lower bounds for
deterministic $\MSO$ string transductions.  To obtain the paper's formulations
for deterministic 2DFTs, we invoke Theorem~\ref{thm:eh} outside Lean.  Thus,
the Lean proofs use only the three admitted axioms listed below, whereas the
paper's 2DFT conclusions additionally rely on the external equivalence in
Theorem~\ref{thm:eh}.

Structural examples, such as the two presentations of reverse-complement as a
2DFT and an $\MSO$ string transduction, are not formalised.

\begin{table}[t]
\centering
\footnotesize
\renewcommand{\arraystretch}{1.25}
\begin{tabularx}{\textwidth}{@{}>{\raggedright\arraybackslash}p{0.31\textwidth}
                                   >{\raggedright\arraybackslash}X
                                   >{\raggedright\arraybackslash}p{0.11\textwidth}@{}}
\toprule
Result & Lean theorem(s) & Axioms \\
\midrule
Thm.~\ref{thm:wrp-closures} (basic closures) &
  \texttt{isWRP\_relabel}, \texttt{isWRP\_restrict},
  \texttt{isWRP\_reverse}, \texttt{isWRP\_disjointUnion},
  \texttt{isWRP\_concat} & none \\
Thm.~\ref{thm:wrp-strict-over-poly} (above $\polyreg$) &
  \texttt{polyreg\_strict\_subset\_wrp} & (iii) \\
Thm.~\ref{thm:wrp-logspace} (logspace evaluation) &
  \texttt{wrp\_isLogspaceMH}, \texttt{wrp\_isLogspaceTM},
  \texttt{wrp\_logspace\_polytime}, \texttt{wrp\_logspaceTM\_polytime} &
  (i) \\
Cor.~\ref{cor:srr-quadratic} (quadratic evaluation) &
  \texttt{srr\_quadratic} & (i) \\
Thm.~\ref{thm:wrp-strict-below-logspace} (below logspace) &
  \texttt{wrp\_strict\_below\_logspace},
  \texttt{wrp\_strict\_below\_logspaceTM} & (i) \\
Thm.~\ref{thm:wrp-not-closed} (closure failure, regular preimage) &
  \texttt{wrp\_not\_closed\_preimage\_comp} & none \\
Thm.~\ref{thm:wrp-not-closed} (closure failure, composition) &
  \texttt{wrp\_not\_closed\_composition} & (i) \\
Thm.~\ref{thm:bounded-rank-collapse} (bounded-rank collapse) &
  \texttt{bounded\_rank\_collapse} & none \\
Cor.~\ref{cor:rank-necessary} (unbounded rank needed) &
  \texttt{rank\_necessary} & (iii) \\
Thm.~\ref{thm:zeta-wrp} (zeta in $\SWR$) &
  \texttt{zetaMap\_realisedByWRP}, \texttt{zetaSweep\_isSRR1} & none \\
Prop.~\ref{prop:alw-sweep-swr} (additive level sorts in $\SWR$) &
  \texttt{additiveSweep\_isSRR1} & none \\
Prop.~\ref{prop:two-pyramid-criterion} (two-pyramid criterion) &
  \texttt{two\_pyramid\_criterion} & (iii) \\
Thm.~\ref{thm:zeta-not-polyregular} (zeta beyond $\polyreg$) &
  \texttt{zetaMap\_not\_polyregular} & (iii) \\
Cor.~\ref{cor:zeta-not-regular} (zeta beyond $\MSO$) &
  \texttt{zetaMap\_not\_regular} & (iii) \\
Thm.~\ref{thm:narayana-sweep} (Narayana sweep) &
  \texttt{heightSweep\_bijOn},
  \texttt{valleys\_heightSweep},
  \texttt{doubleRises\_heightSweep} & none \\
Thm.~\ref{thm:H-not-polyregular} (height sweep beyond $\polyreg$) &
  \texttt{heightSweep\_not\_polyregular},
  \texttt{heightSweep\_not\_regular} & (iii) \\
Thm.~\ref{thm:wrp-slice-semilinearity} (first-ascent semilinearity) &
  \texttt{wrp\_slice\_profile\_semilinear} & (i) \\
Cor.~\ref{cor:model-free-obstruction} (model-free obstruction) &
  \texttt{model\_free\_obstruction} & none \\
Thm.~\ref{thm:wrp-no-swap} (no area--dinv swap) &
  \texttt{wrp\_no\_area\_dinv\_swap} & (i) \\
Thm.~\ref{thm:two-parameter-semilinearity} (two-parameter semilinearity) &
  \texttt{two\_param\_profile\_semilinear\_unconditional} &
  (ii) \\
Cor.~\ref{cor:inverse-zeta-not-wrp} (arity one) &
  \texttt{inverse\_zeta\_not\_wrp\_arity1} & (i) \\
Cor.~\ref{cor:inverse-zeta-not-wrp} (general arity) &
  \texttt{inverse\_zeta\_not\_wrp} & (i), (ii) \\
\bottomrule
\end{tabularx}
\caption{The headline results and their Lean counterparts.  The last column
lists the admitted axioms (i)--(iii), described in the text, on which each
Lean proof depends beyond the kernel; ``none'' marks results proved with no
admitted axiom.  Lean names are cited without namespace prefixes; the
repository documentation maps the numbered statements of the paper, including
the intermediate lemmas, to their formal counterparts.}
\label{tab:lean}
\end{table}

\paragraph{Trust base.}
Every mathematical result in the development is built from Lean's kernel and
its three standard axioms (\texttt{propext}, \texttt{Classical.choice}, and
\texttt{Quot.sound}), together with exactly three admitted axioms of our own.
What a machine-checked development assumes is part of what it certifies, so we
state all three in full.  Below, $w\models\varphi(i_1,\ldots,i_k)$ means that
the word $w$ satisfies the $\MSO$ formula $\varphi$ when its free first-order
variables are interpreted by the positions $i_1,\ldots,i_k$.
\begin{enumerate}[label=(\roman*)]
\item \texttt{buchi}.  For every finite alphabet $\Sigma$ and every $\MSO$
sentence $\varphi$ over $\Sigma$, there is a DFA accepting exactly the words
$w\in\Sigma^*$ with $w\models\varphi$.  This is the logic-to-automata direction
of the B\"uchi--Elgot--Trakhtenbrot theorem~\cite{Thomas1997}.

\item \texttt{msoDefinableRel2\_semilinear\_general}.  Fix a finite alphabet
$\Sigma$, and call a map $F\colon\N^2\to\Sigma^*$ \emph{block-linear} if there
are words $b_1,\ldots,b_r\in\Sigma^*$ and constants
$\alpha_i,\beta_i,c_i\in\N$ with
\[
   F(m,n)=b_1^{\,\alpha_1m+\beta_1n+c_1}\,
          b_2^{\,\alpha_2m+\beta_2n+c_2}\cdots
          b_r^{\,\alpha_rm+\beta_rn+c_r}.
\]
For every block-linear $F$ and every $\MSO$ formula $\varphi(x_1,\ldots,x_k)$
over $\Sigma$ with $k$ free first-order variables, the set
\[
   \bigl\{(m,n,i_1,\ldots,i_k)\in\N^{k+2}\;:\;
          F(m,n)\models\varphi(i_1,\ldots,i_k)\bigr\}
\]
is semilinear.  The family $W_{m,n}$ of
Definition~\ref{def:two-parameter-family} is the block-linear map with blocks
$U$, $UD$, $D$ and coefficients $(1,0,0)$, $(0,1,0)$, $(1,0,0)$, and the slice
$W_n$ of Section~\ref{sec:wrapped-flat} is its row $m=1$.  The axiom packages
the consequence of (i) and the Ginsburg--Spanier
theorem~\cite{GinsburgSpanier1966} needed for these block-linear families.

\item \texttt{polyreg\_regular\_preimage}.  For finite alphabets
$\Sigma$ and $\Gamma$, a polyregular partial map
$f\colon\Sigma^*\rightharpoonup\Gamma^*$, and a regular language
$L\subseteq\Gamma^*$, the language
$\{w\in\Sigma^*: f(w)\text{ is defined and }f(w)\in L\}$ is
regular~\cite{Bojanczyk2018}.
\end{enumerate}

\paragraph{What is proved or avoided.}
The arguments in the body of this paper quote several further results from the
literature and also use the signed companion of (ii) described below.  In Lean,
these results are proved or bypassed by different arguments; none adds to the
trust base.
\begin{itemize}
\item \emph{The automata-to-logic direction of B\"uchi--Elgot--Trakhtenbrot},
that every language recognised by a DFA is definable by an $\MSO$ sentence.
The paper uses it for the restriction clause
of Theorem~\ref{thm:wrp-closures} and in the proof of
Theorem~\ref{thm:bounded-rank-collapse}; the development proves it, so both of
those results are axiom-free.

\item \emph{The Ginsburg--Spanier theorem}, quoted in
Section~\ref{sec:slice-semilinearity} to identify the semilinear sets with the
Presburger-definable ones.  It is a proved theorem in Mathlib and is used as
such.

\item \emph{Woods' Presburger counting theorem}~\cite[Theorem~1.10]{Woods2015}
and the Ehrhart theory behind it~\cite{Ehrhart1962}.
Section~\ref{sec:slice-semilinearity} uses them to derive
Lemma~\ref{lem:presburger-counting}.  The Lean development instead proves
directly the single-count case needed here, by decomposing the semilinear
relation into simple pieces and reducing the fibre count to elementary
counting of arithmetic progressions.  Thus, the formal proof of
Lemma~\ref{lem:presburger-counting} uses neither Woods' theorem nor Ehrhart
theory and requires no admitted axiom.

\item \emph{A signed companion of} (ii).  Axiom (ii) is about sets of
positions, whereas the rank comparisons of
Sections~\ref{sec:slice-semilinearity} and~\ref{sec:inverse-zeta} need the
analogous statement for the $\Z^d$-valued atom ranks produced by a $k$-ary
prefix-additive rank function $\kappa$: that
\[
   \bigl\{(m,n,i_1,\ldots,i_k,v)\;:\;
          \kappa^{F(m,n)}(i_1,\ldots,i_k)=v\bigr\}
\]
is semilinear, with $v\in\Z^d$ encoded by a pair of vectors in $\N^d$.  This
follows from axiom~(ii), the automata-to-logic direction of
B\"uchi--Elgot--Trakhtenbrot proved in the development, and
Lemma~\ref{lem:presburger-counting}.

\item \emph{The linear-growth collapse}
(Theorem~\ref{thm:polyregular-linear-collapse}) and \emph{closure of polyregular
maps under composition}~\cite{Bojanczyk2018}, both used in the proof of
Proposition~\ref{prop:two-pyramid-criterion}.  The formal proof needs neither;
see the discussion of statement fidelity below.
\end{itemize}

The dependencies in Table~\ref{tab:lean} can be audited mechanically with
Lean's \texttt{\#print axioms} command.  Among the supporting results not
listed separately there, the two-pyramid closed forms, the deficit-zero lemmas,
and the two stages of the simulation from the multihead bounded-counter model
to the worktape model are axiom-free.  The evaluator theorem itself uses~(i) to
turn the $\MSO$ data of a presentation into finite automata.

\paragraph{Statement fidelity.}
A machine-checked proof certifies exactly its formal statement.  The following
differences between the statements in the paper and their Lean counterparts
therefore matter when interpreting the headline results.
\begin{itemize}[leftmargin=*,itemsep=0.5em]
\item \emph{Two versions of $\WRP$.}
The formalisation contains two versions of $\WRP$.  The $\WRP$ class of
Definition~\ref{def:wrp} requires copy names to have positive arity and $\chi$
to be a strict total order on selected atoms.  A relaxed $\WRP$ class permits
arity zero and requires only the combined output order to be total on selected
atoms.  The relaxed class is a syntactic superset, so negative theorems proved
for it are stronger.  Each negative headline theorem is also restated for the
$\WRP$ class of Definition~\ref{def:wrp} with the same trust base.  The positive
memberships of $\zeta$, $H$, and the additive level sorts are proved directly
for that class.

\item \emph{Closure properties.}
Among the headline results, Theorem~\ref{thm:wrp-closures} has the most
substantial differences between its paper and Lean statements.  There are two
kinds of difference.  First, the Lean statements do not include the paper's
arity-preservation claim: if the input maps have presentations of arity at most
$k$, then so does the resulting map.  They also omit definition by cases and
letter-deleting relabellings.  Restriction is formulated using an $\MSO$
sentence, and tagging and concatenation are proved only for two maps; the
concatenation theorem permits only a one-letter separator.  Second, all five
Lean clauses are proved for the relaxed $\WRP$ class.  A closure theorem for a
larger class does not imply that a subclass is closed under the same operation.
Separate closure theorems for the $\WRP$ class of Definition~\ref{def:wrp} are
currently available only for relabelling and concatenation.

\item \emph{Two-pyramid criterion.}
The formal version of Proposition~\ref{prop:two-pyramid-criterion} is stronger:
it omits the growth hypothesis
$|f(P_{m,n})|=O(|P_{m,n}|)$ because its proof applies axiom~(iii) directly.

\item \emph{Machine-model statements.}
For Corollary~\ref{cor:zeta-not-regular} and the 2DFT clause of
Theorem~\ref{thm:H-not-polyregular}, Lean proves the corresponding
non-realisability statements for deterministic $\MSO$ string transductions.
The conclusions for deterministic 2DFTs in the paper then use
Theorem~\ref{thm:eh} externally; the equivalence itself is not formalised in
Lean.
For Theorem~\ref{thm:wrp-logspace}, Lean first constructs an evaluator in the
multihead bounded-counter model.  Such a machine has a fixed number of two-way
input heads and linearly bounded counters; their positions and values require
$O(\log n)$ bits.  Lean proves that this evaluator computes the given $\WRP$
map and has a polynomial upper bound on its number of steps.  A formally proved
simulation then yields a deterministic worktape transducer with read-only
input, write-only output, and an $O(\log n)$-space read-write work tape.  This
second model directly matches the resource convention in
Theorem~\ref{thm:wrp-logspace}, and Lean also proves a polynomial step bound for
the simulated evaluator.  These results use the fixed input alphabet
$\{U,D\}$; the output-length conclusion $|T(w)|=O(n^k)$ is not formalised.
For Corollary~\ref{cor:srr-quadratic}, Lean constructs an evaluator in the
multihead bounded-counter model and proves that every halting run on an input
of length $n$ has at most $D(n+1)^2$ steps, for a fixed constant $D$.  The
machine implements the position and atom-rank comparisons rather than treating
them as an external operation, and their cost is included in the quadratic
step bound.  This accounts for the operations treated as constant-time word
comparisons in the paper's $O(n^2)$ analysis.
\end{itemize}

\section{Related work and open problems}
\label{sec:discussion}\label{sec:related-work}

Garsia and Haiman introduced the algebraic $q,t$-Catalan
sequence~\cite{GarsiaHaiman1996}.  Haglund proposed the Catalan-word formula that
became the area--bounce model~\cite{Haglund2003}, and Haiman and Haglund
developed the zeta map relating the area--bounce and area--dinv
formulations~\cite[Section~2]{ThomasWilliams2018}.  Garsia and Haglund
announced a proof of the area--bounce formula in 2001~\cite{GarsiaHaglund2001}
and published the full proof in
2002~\cite{GarsiaHaglund2002}.  Our statistics and unlabelled formulation of
zeta follow Haglund's monograph~\cite{Haglund2008}.
An inverse form had appeared earlier, without the later zeta terminology, in
work of Andrews, Krattenthaler, Orsina, and Papi~\cite{AndrewsEtAl2002};
Armstrong, Loehr, and Warrington identify it in the sweep-map
framework~\cite{ArmstrongLoehrWarrington2015}.  Haglund and Loehr extended the
framework to labelled parking functions~\cite{HaglundLoehr2005}.

Armstrong, Loehr, and Warrington subsequently placed zeta in a broader
framework of maps that assign integer levels to steps and reorder them by level
\cite{ArmstrongLoehrWarrington2015,ArmstrongLoehrWarrington2016}.  The
starting-point levels are called ranks in part of the rational Dyck-path
literature~\cite{Xin2015}.  Ceballos, Denton, and Hanusa developed inversion
methods for rational zeta maps and explicit inverses for several
families~\cite{CeballosDentonHanusa2016}.  Thomas and Williams proved the
general sweep maps bijective via an inverse for the modular sweep
map~\cite{ThomasWilliams2018}; Pons later gave an iterative inverse for
classical zeta directly on area sequences~\cite{Pons2022}.

The map $H$ of Section~\ref{sec:narayana-sweep} is the starting-height,
increasing-level, right-to-left member of the Armstrong--Loehr--Warrington
sweep family, so its bijectivity follows from known sweep-map
theory~\cite{ThomasWilliams2018}.  After
translating path conventions, its breadth-first tree description agrees with
the plane-tree case of the Ceballos--Fang--M\"uhle encoding
$\Xi_{\mathrm{bounce}}$~\cite[Section~3.3]{CeballosFangMuhle2020}; Fang gives an
explicit breadth-first description~\cite[Construction~3.6 and
Proposition~3.7]{Fang2024}.  The exchange of valleys and double
rises is a classical Narayana symmetry.  Deutsch proved it using a recursive
involution~\cite{Deutsch1999}, and the corresponding rise--valley transport
under the classical zeta map is reviewed by Sulzgruber and
Thiel~\cite[Section~2.8]{SulzgruberThiel2018}.  Thus,
Theorem~\ref{thm:narayana-sweep} recalls a known combinatorial result; its role
here is to show that the small class $\SWR$, and hence $\WRP$, already realises
this symmetry.  Theorem~\ref{thm:H-not-polyregular} is the new computational
lower bound: it shows that the additive level sort used by this classical
bijection cannot be replaced by a polyregular mechanism.

The lower bounds also use the inequality $\dinv\le\coarea$ and uniqueness of
its equality case at fixed semilength and common value.  Both appear in
Ammar's thesis~\cite{Ammar2015}, and the same cases are encoded by the
deficit-zero framework of Lee, Li, and Loehr~\cite{LeeLiLoehr2018}.  We give
the explicit area-sequence form needed in Section~\ref{sec:main}.

To our knowledge, no explicit combinatorial, semilength-preserving bijection on
$\D$ that exchanges $\area$ and $\dinv$ is known~\cite{Pons2022}.  For the
closely related exchange of $\area$ and $\bounce$, Ayyer and Sundaravaradan construct an
explicit bijection on an exponentially large subset of the Dyck paths, while
the full problem remains open~\cite{AyyerSundaravaradan2025}.  Our question is
more constrained: can a statistic exchange be realised by a prescribed
word-transduction mechanism?  The class $\WRP$ isolates the single global
additive rank sort used by zeta and $H$.  It contains both of these known
bijections, both of which lie outside $\polyreg$, but the lower bounds show
that it cannot realise a full area--dinv exchange and, under the realisation
convention, that realisability is not preserved by inversion.

These Catalan statistics and sweep maps also extend beyond ordinary Dyck
paths.  The rational $q,t$-symmetry follows from Mellit's proof of the rational
shuffle theorem~\cite{Mellit2021}, while rational sweep maps use
slope-dependent integer levels~\cite{ArmstrongLoehrWarrington2015,
ArmstrongLoehrWarrington2016,ThomasWilliams2018}.  Xin and Zhang extend
$\dinv$, $\area$, and $\bounce$ to vector-$\vec k$ Dyck
paths~\cite{XinZhang2023}; constant vectors recover the Fuss--Catalan setting.
After the standard endpoint and step-convention identifications, these
frameworks include versions of the ordinary and Fuss--Catalan families, but
they generalise the classical setting in different directions.

The automata-theoretic history supplies the computational hierarchy used in
the paper.  Engelfriet and Hoogeboom proved that $\MSO$-definable string
transductions coincide with deterministic two-way finite-state
transducers~\cite{EngelfrietHoogeboom2001}.  Alur and \v{C}ern\'y later gave
the equivalent model of deterministic copyless streaming string
transducers~\cite{AlurCerny2010}.  Polyregular maps extend this regular class
and admit equivalent descriptions by pebble transducers, \emph{for}-programs,
and tuple interpretations~\cite{Bojanczyk2018,BojanczykKieferLhote2019}; their
possible output-growth rates are studied in~\cite{Bojanczyk2023Growth}.
Equivalence for the regular class is decidable, with a complete delay-based
characterisation~\cite{FiliotJeckerLodingWinter2023}.  Against this background,
$\WRP$ adds one operation to polyregular presentations: globally sorting
selected atoms by unbounded integer ranks.

The closest comparison is ranked $\MSO$ enumeration.  An $\MSO$ query selects
satisfying assignments, a weighted $\MSO$ formula assigns each a cost in an
ordered abelian group, and the assignments are enumerated without repetition
in nondecreasing cost order.  Bourhis, Grez, Jachiet, and Riveros give an
algorithm with linear preprocessing and logarithmic delay between
answers~\cite{BourhisEtAl2021}.

The comparison with $\WRP$ is precise but not an equivalence.  Every
prefix-additive rank function is a special case of a cost function defined by
a weighted $\MSO$ formula.  A selected atom $(c,\bar i)$ then plays the role of
a satisfying assignment: its selection formula is the query, and its atom rank
is the cost.  The rank sort orders atoms by cost, while $\chi$ orders atoms of
equal cost; their labels are then concatenated into one output word.  Ranked
$\MSO$ enumeration instead returns assignments and studies the delay between
answers, without prescribing their order within a cost class.  Thus, $\WRP$
uses a deterministic prefix-additive fragment of the cost formalism for a
different purpose.  This restriction supports the logspace evaluation and
semilinear lower bounds in Sections~\ref{sec:wrp}
and~\ref{sec:slice-semilinearity}.

Filiot, Lhote, and Reynier introduced lexicographic
transductions~\cite{FiliotLhoteReynier2025}.  This class contains polyregular
transductions, admits exponential output growth, and preserves regular
languages under inverse image.  Its
lexicographic enumeration of finite-alphabet annotations differs from sorting
selected atoms by unbounded additive integer atom ranks.  The two classes are
incomparable.  Exponential-growth lexicographic transductions cannot be
$\WRP$, since every fixed-arity $\WRP$ presentation has polynomial output
growth.  Conversely, the $\WRP$ map $D$ from
Theorem~\ref{thm:wrp-not-closed} has a nonregular inverse image of a regular
language, whereas lexicographic transductions preserve regular languages
under inverse image.

Thus, $\WRP$ is a deliberately narrow extension of the standard transducer
hierarchy, isolating one global sort by unbounded additive integer atom ranks.

A broader computational perspective on algebraic combinatorics comes from
work seeking to make the notion of a \emph{combinatorial interpretation}
precise.  Pak proposes membership in $\#\mathrm{P}$ as a workable criterion for
a nonnegative integer-valued counting function to have such an interpretation,
illustrating the question with quantities that include Kronecker and Schubert
coefficients~\cite{Pak2024WhatIs}.  Ikenmeyer and Pak develop the corresponding
study of membership and nonmembership in $\#\mathrm{P}$
\cite{IkenmeyerPak2022}.  In this programme, one asks whether the quantity can
be represented as the number of polynomially verifiable witnesses.  Our
question is adjacent but different: when two known families, or two statistics
on one known family, are already equidistributed, what resources are needed to
compute an explicit bijection?  The transducer hierarchy distinguishes
finite-state logic, polyregular copying, and the unbounded atom ranks used by
zeta.  Our lower bounds therefore constrain a mechanism for bijective proofs,
not the existence of a combinatorial interpretation in Pak's sense.

Other Catalan bijections point to mechanisms that are not explicit primitives
of $\WRP$.  Deutsch's involution is mirror reflection under the first-return
binary-tree encoding, while Hopkins and Joseph identify the
Lalanne--Kreweras involution with rowvacuation on the type~$A$ root poset and
hence with a composition of toggles~\cite{Deutsch1999,HopkinsJoseph2022}.  The
word-based models developed here do not formalise first-return recursion or
toggle operations, and we do not classify these maps.

\paragraph{Open problems.}
The results leave several natural questions about the rank sort and its
combinatorial scope.
\begin{enumerate}[label=(\arabic*)]
\item The class $\WRP$ permits one rank sort.  Is the hierarchy obtained by
allowing a fixed number of successive rank sorts strict?  In particular, does
a two-layer presentation suffice for $\zeta^{-1}$?
\item Is equivalence of two $\WRP$ presentations decidable?  More generally,
after fixing a finitely presented ambient transducer model, is it decidable
whether a presented map admits a $\WRP$ presentation?
\item Does the deficit-zero rigidity argument have an analogue for $\area$
and $\bounce$, leading to a computational obstruction for an area--bounce
exchange?
\item Which additive level sorts are polyregular?  Unbounded levels alone do
not force a lower bound: if $\nu(U)=\nu(D)=1$, then the level of position $i$
is $i-1$, and $\Phi_\nu$ is the identity map.  A classification by the step
weights and tie-order would separate such degenerate cases from examples such
as $H$ and other mixed-sign sweep maps.
\end{enumerate}

\section*{Declaration of AI use}
The mathematical theory of this paper was developed jointly by the human
authors and the AI models GPT~Pro (OpenAI), Codex (OpenAI), and Claude
(Anthropic).  The Lean formalisation described in Section~\ref{sec:lean} was
carried out fully by AI models: Claude, Codex, and Aristotle (Harmonic).  The
writing of the paper was assisted by Claude and Codex.

\section*{Acknowledgments}
We would like to thank Xiaoyu Huang for the inverse of the zeta map.
This work was supported by the National Research Foundation of Korea (NRF)
grant funded by the Korean Government (MSIT) (No.~RS-2023-00279680).

\bibliographystyle{plain}
\bibliography{paper}

\clearpage
\appendix
\section[Proofs of the structural properties of WRP]{Proofs of the structural properties of $\WRP$}
\label{app:wrp-structural-proofs}

This appendix supplies the detailed constructions and proofs for the
structural results stated in Section~\ref{sec:wrp}.  The main text retains the
definitions, theorem statements, and conceptual interpretation; here we give
the witnesses and verify the technical details.

\subsection{Closure constructions}

\begin{proof}[Proof of Theorem~\ref{thm:wrp-closures}]
Definition~\ref{def:wrp} provides only two kinds of ordering data for a map:
one prefix-additive rank function $\kappa_c$ for each copy name $c$
(Definition~\ref{def:prefix-additive-rank}) and one global $\MSO$ tie-order
$\chi$.
Several of the constructions below would be
easiest to describe by grouping the selected atoms into contiguous \emph{blocks}
laid out one after another (concatenation~(iii), for instance, wants all of
$f(w)$ as one block, then the separator, then all of $g(w)$); but
Definition~\ref{def:wrp} provides no primitive for such grouping, only the
functions $\kappa_c$ and the single order $\chi$.  Each construction must
therefore be realised \emph{through} legitimate functions $\kappa_c$ and a
single $\MSO$-definable $\chi$, with any block structure encoded in their atom
ranks by a dedicated \emph{output-block-tag} coordinate (as in~(ii) and~(iii)
below).  We first note two robustness facts that make this possible, both
immediate from the definitions, and then treat the five constructions.

\vspace{0.5em}
\noindent\emph{(R1) Robustness of prefix-additive rank functions.}
Definition~\ref{def:prefix-additive-rank} is closed under the elementary
changes used below.  A constant vector is prefix-additive (put it in $c_0$ and
use zero-weight sources and zero corrections).  The negation $-\kappa$ is
prefix-additive after negating $c_0$, every transition weight, and every
correction table.  A $d'$-dimensional prefix-additive rank function embeds,
with its coordinate order preserved, into a larger $\Z^d$ by applying that
coordinate embedding to $c_0$, all transition weights, and all corrections,
leaving the other coordinates $0$.  The rank
dimension $d$ is an unconstrained parameter of Definition~\ref{def:wrp}, so
raising it this way keeps the map in the class; only the
position-tuple \emph{arity} is bounded, and none of the constructions below
widens a tuple.  When two maps of arities $k\le k'$ are combined, each
atom of the arity-$k$ part is padded to arity $k'$ by adjoining dummy position
variables pinned by equalities (say $x_{k+1}=\cdots=x_{k'}=x_1$) and assigning
the new coordinates zero-weight sources and zero corrections.  This changes
neither the selected atoms, their labels, nor their atom ranks, so the combination
has arity $\max(k,k')$.

\vspace{0.5em}
\noindent\emph{(R2) Robustness of the tie-order.}
Suppose the set $C$ of copy names is partitioned into finitely many classes by
a fixed function of the copy name, each class carrying its own $\MSO$-definable
strict total order on its atoms, and fix a linear order on the classes.  Then
``earlier class first, and within a class that class's order'' is again a single
$\MSO$-definable strict total order on all selected atoms: the class comparison
is a finite case distinction on the copy names, and a finite combination of
$\MSO$ orders is $\MSO$.  This supplies the global $\chi$ whenever sets of copy
names are merged.

\vspace{0.5em}
\noindent\emph{(i) Restriction and definition by cases.}
A regular language is $\MSO$-definable (Section~\ref{sec:mso}).  To restrict a
map to a regular language $L$, conjoin the presentation's domain sentence
$\varphi_{\mathrm{dom}}$ with an $\MSO$ sentence defining $L$; all other data
are unchanged.  This gives precisely the domain
$L\cap\operatorname{dom}(T)$ and the original output $T(w)$ there.

For a finite case distinction, let $\lambda_j$ be an $\MSO$ sentence defining
$L_j$, and let $\varphi_{\mathrm{dom},j}$ be the domain sentence of the chosen
presentation of $T_j$.  Take the disjoint union of the sets of copy names in the
per-case maps, and replace the selection formula $\varphi_{j,c}$ for every copy
name $c$ from case $j$ by $\lambda_j\wedge\varphi_{j,c}$.  Use the domain sentence
\[
   \bigvee_{j=1}^r
   \bigl(\lambda_j\wedge\varphi_{\mathrm{dom},j}\bigr).
\]
Embed the atom ranks of all cases into a common dimension by (R1), and obtain a
single tie-order $\chi$ from (R2), indexed by case.  Since the $L_j$ are pairwise
disjoint, exactly one case contributes atoms on any word satisfying this
domain sentence.
Their labels, atom ranks, and relative order are those of that case, so the
output is exactly $T_j(w)$.  Padding position tuples as in (R1) gives arity at
most $\max_j k_j$.

\vspace{0.5em}
\noindent\emph{(ii) Combining outputs with source tags.}
Take the disjoint union of the sets of copy names of all $T_1,\ldots,T_r$, use
the conjunction of their domain sentences, and replace every label $a$ from the
$j$-th map by $(j,a)$.  Let $d_j$ be the rank dimension of $T_j$ and put
$d=1+\max_j d_j$.  For an atom from $T_j$, give its new atom rank the leading
coordinate $j-1$ and embed its original atom rank into the remaining $d-1$
coordinates by (R1).  Use (R2) to combine the original tie-orders into one
$\MSO$ tie-order.  The leading coordinate puts all atoms from $T_1$ first,
then all atoms from $T_2$, and so on.  Within the $j$-th block, comparison
reduces to the original atom ranks, with the tie-order of $T_j$ breaking ties.
Consequently the output
on the common domain is exactly
\[
   \operatorname{tag}_1(T_1(w))\cdots
   \operatorname{tag}_r(T_r(w)),
\]
over the flat tagged alphabet stated in the theorem.

\vspace{0.5em}
\noindent\emph{(iii) Concatenation with fixed separators on nonempty inputs.}
Let $f,g$ have rank dimensions $d_f,d_g$ and put $d=1+\max(d_f,d_g)$.  Form the
map whose set of copy names is the disjoint union of those of $f$ and $g$ with
a fresh arity-$1$ separator copy name $s$, selected exactly at the first input
position and labelled $\#$.  Its domain is
$\operatorname{dom}(f)\cap\operatorname{dom}(g)\setminus\{\varepsilon\}$,
which is specified by conjoining the two original domain sentences with the
$\MSO$ sentence asserting that an input position exists.
Define the atom rank in $\Z^d$ by (R1): coordinate $0$ is a constant
\emph{output-block tag} ($0$ on the copy names from $f$, $1$ on $s$, and $2$
on those from $g$); coordinates
$1,\ldots,d-1$ carry $\kappa_f$ (resp.\ $\kappa_g$) embedded by (R1), and are
$0$ on $s$.  Take $\chi$ from (R2) with the blocks ordered $f<s<g$.  In the
lexicographic order the output-block tag dominates, so every $f$-atom precedes
$s$, which precedes every $g$-atom; within the $f$-block the tags and the
padding coordinates agree, so the comparison reduces to $\kappa_f$ with $f$'s
tie-order and emits $f(w)$, and likewise $g$ emits $g(w)$.  The output is
$f(w)\,\#\,g(w)$.  More generally, gluing $r$ outputs by $r-1$ constant separator
strings uses the same idea with $2r-1$ blocks (indices $0,\ldots,2r-2$): the even
blocks carry the outputs.  If $s_j=b_1\cdots b_\ell$, its odd block has $\ell$
fresh arity-$1$ copy names, all selected at the first input position, labelled
$b_1,\ldots,b_\ell$, and ordered in that order by $\chi$.  An empty separator
uses no copy names.  The domain sentence is the conjunction of the $r$ original
domain sentences and the assertion that the input is nonempty.  The
output-block-tag coordinate therefore gives exactly
$T_1(w)s_1\cdots s_{r-1}T_r(w)$ on precisely the domain stated in the theorem.
The nonempty-domain condition supplies the first position to which inserted
separator atoms are attached; it is needed whenever some separator is
nonempty.

\vspace{0.5em}
\noindent\emph{(iv) Replacing or deleting output letters.}
Fix $h\colon\Gamma\to\Delta^*$ with $|h(a)|\le 1$.  For a copy name $c$, let
$\varphi_c$ be its selection formula and let $\psi_{c,a}$ be the formula saying
that a selected atom has label $a\in\Gamma$.  Replace its selection formula by
\[
   \varphi_c\ \wedge\!
   \bigvee_{\substack{a\in\Gamma\\ h(a)\ne\varepsilon}}\psi_{c,a},
\]
and, for each $b\in\Delta$, use the new label formula
\[
   \bigvee_{\substack{a\in\Gamma\\ h(a)=b}}\psi_{c,a}.
\]
These are finite $\MSO$ formulas.  Thus, an atom is removed exactly when its
old label $a$ satisfies $h(a)=\varepsilon$; otherwise it remains in the same
position in the output order and receives the unique letter $h(a)$.  The
domain, prefix-additive rank functions, tie-order, and arities are unchanged,
so the resulting output
is exactly $\widehat h(T(w))$.

\vspace{0.5em}
\noindent\emph{(v) Output reversal.}
Replace every prefix-additive rank function $\kappa$ by $-\kappa$, which is
again prefix-additive by
(R1), and $\chi$ by its converse, which is $\MSO$-definable (swap the two atom
arguments in the defining formula).  Negating every coordinate reverses the lexicographic
order on $\Z^d$, and the converse tie-order reverses the order among atoms with
equal atom rank.  The domain, selected atoms, and labels are unchanged, so the
emission order is exactly reversed and the output is
$\operatorname{rev}(T(w))$ on $\operatorname{dom}(T)$.

\vspace{0.5em}
The constructions in (i) and (ii) use only copy names from the original maps,
padded when necessary to the maximum input-tuple arity.  Construction (iii)
adds only arity-$1$ separator copy names, while (iv) and (v) do not change the
set of copy names.  Hence
the resulting arity is at most $\max_j k_j$, and arity $1$ stays arity $1$;
thus the ranked-regular fragment $\RR$ is closed under all five constructions.
The rank dimension may grow in (i)--(iii), but
Definition~\ref{def:wrp} bounds only the arity.
\end{proof}

\subsection{Evaluation algorithms}

\begin{proof}[Proof of Theorem~\ref{thm:wrp-logspace}]
Fix $T$ and write $n=|w|$.  Before producing any output, evaluate the fixed
$\MSO$ domain sentence of the presentation.  By the
B\"uchi--Elgot--Trakhtenbrot theorem this is a regular condition, so a
finite-state pass suffices.  If the condition fails, report that $T(w)$ is
undefined; henceforth assume $w\in\operatorname{dom}(T)$.

By Definition~\ref{def:wrp} the output is the
labels of the selected atoms listed in the order $\prec$.  There can be up to
$O(n^k)$ selected atoms (arity at most $k$, finitely many copy names), far too many
to store or sort in $O(\log n)$ space.  The key observation is that, for a
\emph{single}
potential atom, whether it is selected, its label, and its $\prec$-position
relative to any other atom are each recomputable from the input with only
$O(\log n)$ bits of bookkeeping; so we print the output one letter at a time,
in $\prec$-order.

\vspace{0.5em}
\noindent\emph{Recomputable primitives.}
A potential atom $\alpha=(c,\bar i)$ is named by its copy name $c$ (one of
finitely many) and a tuple $\bar i\in\{1,\ldots,n\}^{k_c}$, where $k_c\le k$; it
therefore uses $O(k\log n)$ bits.  The total number of potential atoms is
$\sum_{c\in C}n^{k_c}=O(n^k)$.  With $\alpha$ on the work tape:
\begin{itemize}
\item \emph{selection and label} are decided by evaluating the fixed $\MSO$
selection and label formulas on the input annotated with $\bar i$.  During a
scan, the stored indices in $\bar i$ determine the finite bit-vector marking
which free variables occupy the current position.  The resulting marked word
is tested by a fixed finite automaton, in constant space beyond $\bar i$.
\item \emph{the atom rank} $\kappa_c(\bar i)\in\Z^d$ is a fixed constant plus a
fixed sum of prefix-additive coordinate contributions
(Definitions~\ref{def:prefix-rank} and~\ref{def:prefix-additive-rank});
rescanning the finitely many fixed sources while accumulating their weights,
then adding the bounded local corrections and the constant, computes it.  Each
coordinate is bounded in absolute value by $Cn$ for a constant $C$ of $T$, and
so is held in an $O(\log n)$-bit counter.
\item \emph{the comparison} $\alpha\prec\beta$ is the lexicographic comparison
of the two atom ranks, with ties broken by the fixed $\MSO$ tie-order $\chi$ on the
input annotated with both tuples.
\end{itemize}
Thus, ``$\alpha$ is selected'' and the total order $\prec$ are available on
demand in $O(\log n)$ space, with no stored list of atoms.

\vspace{0.5em}
\noindent\emph{Emitting the output.}
Since $\prec$ is a total order on the selected atoms, the output is the list
of their labels in increasing $\prec$-order, and it suffices to produce, from
each emitted atom, its $\prec$-successor.  Maintain the atom $L$ emitted last,
initialised to a sentinel with $\bot\prec\alpha$ for every atom $\alpha$.  In
each round, scan all $O(n^k)$ potential atoms once, maintaining a best-so-far
atom $B$: the $\prec$-least selected atom seen so far with $L\prec B$.  For
each potential atom $\alpha$, test whether $\alpha$ is selected and, if so,
compare it against $L$ and $B$ with the recomputable comparison above,
updating $B$ to $\alpha$ when $L\prec\alpha$ and $\alpha\prec B$ (or when no
eligible selected atom has been recorded yet); when $B$ is updated, record its
label as well.  At the end of the scan, if no selected atom with
$L\prec\alpha$ was found, every selected atom has been emitted and the machine
halts.  Otherwise, emit the recorded label of $B$ and set $L:=B$.  Because
$\prec$ is a strict total
order on the finitely many selected atoms, round $j$ emits exactly the
$j$-th selected atom in $\prec$-order, so the concatenated labels are $T(w)$.

\vspace{0.5em}
\noindent\emph{Resources.}
At every moment the work tape holds a constant number of atoms ($O(k\log n)$
bits each: $L$, $B$, and the atom under examination), a constant number of
$O(\log n)$-bit counters (the source accumulators used by the primitives), and
finite control for the fixed formulas; total space $O(\log n)$.  For the time
bound, there is one round per output letter and one final round, hence
$O(n^k)$ rounds; each round scans $O(n^k)$ potential atoms; and each selection
test, atom-rank computation, or order test requires at most a constant number
of length-$n$ input scans.  Thus, under the unit-cost word convention used
above, the algorithm described above takes $O(n^{2k+1})$ time; at arity $1$
this is $O(n^3)$.  With bit-level time accounting, these bounds acquire an
additional factor of $O(\log n)$.  Finally there are at most $O(n^k)$ selected
atoms, each contributing one letter, so $|T(w)|=O(n^k)$.
\end{proof}

\begin{proof}[Proof of Corollary~\ref{cor:srr-quadratic}]
Write $n=|w|$.  As in the proof of Theorem~\ref{thm:wrp-logspace}, first use a
finite-state pass to test the presentation's domain sentence and report an
undefined value if it fails.  Assume below that $w\in\operatorname{dom}(T)$.
By Definition~\ref{def:wrp} a map in $\SWR$ has arity
$1$, so its potential atoms are the pairs $(c,i)$ with copy name $c\in C$ and
position $i\in\{1,\ldots,n\}$, $O(n)$ in all; its atom rank is a single
integer $\kappa_c(i)\in\Z$ (dimension $d=1$); and its tie-order $\chi$ is a scan
order.  Thus, $\prec$ is ``smaller atom rank first, ties broken by scan order''.

One point must be kept straight.  The atom rank is, by
Definitions~\ref{def:prefix-rank} and~\ref{def:prefix-additive-rank}, an
additive \emph{prefix} sum
$\rho_A^w(i)=\sum_{j<i}\omega(\cdot)$ of an automaton that reads $w$ left to
right, together with a constant and a bounded local correction.  This forward
accumulation is intrinsic to the atom rank and is unrelated to the scan order's
direction, which may be left-to-right or its
reverse.  A forward automaton is not reversible, so the atom rank cannot be
maintained by a right-to-left walk; the algorithm below therefore always
computes the atom rank in the forward direction, and lets the (possibly reversed)
scan direction enter \emph{only} through the positional tie-comparison, which
is read off the two indices with no accumulation.

\vspace{0.5em}
\noindent\emph{The algorithm.}
The procedure is a streaming stable sort of the selected atoms by the
weighted-automaton key $\kappa$, carried out without ever materialising the
sorted list.  As in Theorem~\ref{thm:wrp-logspace} we enumerate the selected
atoms in $\prec$-order without storing them, one output letter per round; here
the arity-one structure makes each round a single forward scan.  Since $\prec$ is a
strict total order, it suffices to find, for the atom $L$ emitted last (for the
first round a sentinel $\bot\prec$ everything), its $\prec$-successor in one
forward scan.  Scan $w$ once left to right, advancing the finitely many fixed
additive sources, so that on reaching position $i$ the atom rank $\kappa_c(i)$
of every $c\in C$ is available in $O(1)$.  Maintain a best-so-far atom $B$: the
$\prec$-least selected atom seen so far that is still $\succ L$.  At each
position $i$ and each $c\in C$, evaluate the selection formula; by the hypothesis on
$T$, both selection and the label, when selected, are decided in $O(1)$ by the
finite-state pass running left to right alongside the source accumulations.  If $(c,i)$ is
selected, compare it (using its freshly computed atom rank, the cached atom ranks of
$L$ and $B$, and the tie-order) against $L$ and $B$, updating $B$ when
$(c,i)\succ L$ and $(c,i)\prec B$; when $B$ is updated, cache its label as
well.  The tie-comparison asks which of two positions comes first in the fixed
scan direction and, when the positions coincide, uses the fixed order on
the set $C$.  At the end of the scan, if no selected atom was $\succ L$, halt:
the output is complete.  Otherwise, emit the cached label of $B$ and set
$L:=B$.  Because $\prec$ is a strict total order
on the finitely many selected atoms, the rounds emit exactly the selected atoms
in increasing $\prec$-order, which is the output of Definition~\ref{def:wrp}.

\vspace{0.5em}
\noindent\emph{Resources.}
There are at most $|C|\,n=O(n)$ selected atoms, hence $O(n)$ rounds, each a
single length-$n$ forward scan doing $O(1)$ work per position, so $O(n^2)$ time
in all.  The work tape holds the stored forms of $L$ and $B$ (a copy name, an
$O(\log n)$-bit position, and one $O(\log n)$-bit cached atom rank for each),
the cached
label of $B$ (constant space), the forward accumulators, and a position index,
all $O(\log n)$ bits; the space is
$O(\log n)$.  Both bounds hold whether the scan order runs left-to-right or
right-to-left.
\end{proof}

\subsection{Regular nonemptiness and the logspace separation}

The following elementary property is the technical ingredient used to separate
$\WRP$ from deterministic logspace.

\begin{lemma}[Regular nonemptiness]
\label{lem:wrp-nonempty-regular}
For every $\WRP$ map $T$, the language
$\{w\in\mathrm{dom}(T):|T(w)|\ge 1\}$ of inputs with nonempty output is
regular.
\end{lemma}

\begin{proof}
By Definition~\ref{def:wrp} the output $T(w)$ is the concatenation of the
labels of the selected atoms, and by Definition~\ref{def:polyregular}(iv)
each selected atom contributes exactly one letter of $\Gamma$; hence
$|T(w)|$ equals the number of selected atoms, and $|T(w)|\ge 1$ if and
only if at least one atom is selected.  ``At least one atom is selected''
is expressed by the $\MSO$ sentence
\[
   \bigvee_{c\in C}\exists x_1\cdots\exists x_{k_c}\,
       \varphi_c(x_1,\ldots,x_{k_c}),
\]
the finite disjunction over the copy names $c\in C$ of the existential closure
of the selection formula $\varphi_c$ (Definition~\ref{def:polyregular}(iii));
the atom-rank order plays no role in the mere existence of a selected atom.
Conjoining with the $\MSO$ domain sentence $\varphi_{\mathrm{dom}}$ and
applying the B\"uchi--Elgot--Trakhtenbrot theorem
(Section~\ref{sec:mso}), the
language $\{w\in\mathrm{dom}(T):|T(w)|\ge 1\}$ is $\MSO$-definable, hence
regular.
\end{proof}

\begin{proof}[Proof of Theorem~\ref{thm:wrp-strict-below-logspace}]
Containment is Theorem~\ref{thm:wrp-logspace}.  For strictness, let
\[
   L_{\ge 0}=\{w\in\{U,D\}^*:\text{every prefix of }w\text{ has nonnegative height}\}
\]
and define $F_{\ge 0}(w)=w$ if $w\in L_{\ge 0}$ and
$F_{\ge 0}(w)=\varepsilon$ otherwise.
Write $n=|w|$.

\vspace{0.5em}
\noindent\emph{$F_{\ge 0}$ is in deterministic logspace.}  One left-to-right scan
over the read-only input of length $n$ maintains a single height counter
(in $[-n,n]$, hence $O(\log n)$ bits) and a single Boolean flag
recording ``some prefix has gone negative''.  At the end of the scan, if the
flag is $0$, re-read the input and copy it to the output;
otherwise emit $\varepsilon$.  The work tape holds the counter and one bit,
i.e.\ $O(\log n)$.

\vspace{0.5em}
\noindent\emph{$F_{\ge 0}$ is not $\WRP$.}  By
Lemma~\ref{lem:wrp-nonempty-regular}, every $\WRP$ map has a
\emph{regular} preimage of the nonempty-output language $\{y:|y|\ge 1\}$.
Were $F_{\ge 0}$ in $\WRP$, that preimage would be
$L_{\ge 0}\setminus\{\varepsilon\}$, forcing $L_{\ge 0}$ to be regular.  But
$L_{\ge 0}$ is not regular.  Indeed, if $p$ were a pumping length, then
$U^pD^p\in L_{\ge 0}$ would admit a decomposition $xyz$ with
$|xy|\le p$ and $|y|>0$.  Necessarily $y=U^t$ for some $t>0$, but pumping
down would give $xz=U^{p-t}D^p$, whose final prefixes have negative height.
This contradicts the pumping lemma.  Hence $F_{\ge 0}$ is not in $\WRP$.
\end{proof}

\subsection{Failure of composition closure}

\begin{proof}[Proof of Theorem~\ref{thm:wrp-not-closed}]
We construct the three objects asserted in the theorem: a $\WRP$ map $D$, a
regular language $K$, and a left-to-right deterministic 2DFT $S$.

\vspace{0.5em}
\noindent\emph{The map $D$.}
Take $D$ to have domain $\{U,D\}^*$ and output alphabet
$\Gamma_D=\{G,B,\#,U,D\}$.  On input
$w=w_1\cdots w_n$, it produces three groups of selected atoms, kept apart by a
leading atom-rank coordinate exactly as in the concatenation construction of
Theorem~\ref{thm:wrp-closures}: every atom rank is a pair
$(b_0,v)\in\Z^2$ whose first component $b_0\in\{0,1,2\}$ is a constant
\emph{output-block tag}, so that in the lexicographic order all block-$0$ atoms
precede the block-$1$ atom, which precedes all block-$2$ atoms.
\begin{itemize}
\item \emph{Block $0$} (the diagnostic block): one \emph{sentinel} atom labelled
$G$ with atom rank $(0,0)$, and, for each position $i$, one atom labelled $B$
with atom rank
$(0,h_w(i))$, where $h_w(i)$ is the height of the length-$i$ prefix of $w$.
For the copy name producing the $B$-atoms, use the two-dimensional height
source with transition weights
$(0,+1)$ on $U$ and $(0,-1)$ on $D$, take $c_0=(0,0)$, and use the local
correction $(0,+1)$ at a $U$-position and $(0,-1)$ at a $D$-position.  Its
prefix-additive rank function therefore assigns the atom rank $(0,h_w(i))$,
including the step at $i$
(Definition~\ref{def:prefix-additive-rank}).
\item \emph{Block $1$}: a single atom labelled $\#$ with atom rank $(1,0)$.
\item \emph{Block $2$}: a verbatim copy of the input, position $i$ emitting
$w_i$ with atom rank $(2,0)$, in input order.
\end{itemize}
Use the $\MSO$ tie-order that orders the three blocks by their indices, places
the sentinel before the $B$-atoms within block $0$, orders those $B$-atoms by
input position, and orders block $2$ by input position.  This completes a
strict total tie-order; in particular, among atoms with atom rank $(0,0)$ the
sentinel precedes every $B$-atom.
Each atom is named by a single position (the sentinel and the $\#$ by the
first input position).  The copy names producing the sentinel, separator, and
input letters have constant prefix-additive rank functions, and the preceding
construction gives the copy name producing the $B$-atoms a prefix-additive rank
function.  Thus, $D$ is a $\WRP$ map
(Definition~\ref{def:wrp}) of
arity $1$.  For nonempty $w$, its output is
$\bigl(\text{block-}0\text{ labels in }\prec\text{-order}\bigr)\,\#\,w_1\cdots w_n$,
the block-$0$ part being the $G$ and $B$ labels listed by increasing prefix
height.

\vspace{0.5em}
\noindent\emph{The first output letter detects $L_{\ge 0}$.}
Recall $L_{\ge 0}$ from the proof of
Theorem~\ref{thm:wrp-strict-below-logspace}.  All block-$0$ atoms share block
index $0$, so among themselves they are ordered by the second atom-rank
coordinate, the height value, with ties broken by the tie-order; the whole output begins
with the label of the $\prec$-least block-$0$ atom.  The sentinel sits at height
$0$.  If every prefix height is $\ge 0$, the least height present is $\ge 0$,
the least atom rank is $(0,0)$, and the tie-order puts the sentinel first there, so
the first output letter is $G$.  If some prefix height is negative, that
position's $B$-atom has atom rank $(0,\text{negative})\prec(0,0)$ and precedes the
sentinel, so the first output letter is $B$.  Hence, for nonempty $w$, the
first output letter is
$G$ if and only if $w\in L_{\ge 0}$.  (On the empty input there are no
positions, hence no atoms, and $D(\varepsilon)=\varepsilon$.)

\vspace{0.5em}
\noindent\emph{Regular inverse image fails.}
Let $K=G\,\Gamma_D^*$, the regular language of strings beginning with $G$.  By the
previous step $D^{-1}(K)=L_{\ge 0}\setminus\{\varepsilon\}$, which is not
regular: adding the single word $\varepsilon$ would make $L_{\ge 0}$ regular,
contradicting the pumping
argument in the proof of Theorem~\ref{thm:wrp-strict-below-logspace}.  Thus, $D$
is a $\WRP$ map and $K$ is regular, yet $D^{-1}(K)$ is not regular,
which proves the first claim.

\vspace{0.5em}
\noindent\emph{The left-to-right 2DFT and composition failure.}
This is the same failure of regularity preservation under inverse image,
expressed as a composition failure.
Let $S$ be the deterministic 2DFT that reads $D(w)$ in one left-to-right pass.
It initialises a Boolean bit to $0$, records in that bit whether the first
letter is $G$, skips the rest of block $0$ and the separator (advancing past
the first $\#$), and then copies the remaining block-$2$ letters to its output
if and only if the bit is $1$.  Its finite control also records whether it is
before or after the separator; every defined head transition moves right, and
it halts acceptingly at the right end marker.  On the empty input it therefore
halts immediately with output $\varepsilon$; its behaviour on other malformed
inputs, which are outside the range of $D$, may be fixed arbitrarily.  Hence,
on input $w$ the composite $S\circ D$ outputs $w$ when $w\in L_{\ge 0}$ (first
letter $G$, bit $1$, copy emitted) and $\varepsilon$ otherwise; that is,
$S\circ D=F_{\ge 0}$, the map of
Theorem~\ref{thm:wrp-strict-below-logspace}, which is not a $\WRP$
map.  By Theorem~\ref{thm:eh} and
Proposition~\ref{prop:conservative}, $S$ itself belongs to $\WRP$.  Hence two
$\WRP$ maps, $D$ and $S$, have a composite outside $\WRP$, proving that the
class is not closed under composition.
\end{proof}

\begin{remark}[Why this is a structural feature, not a defect]
Both $\TDFT/\MSO$ maps and $\polyreg$ maps have regular inverse
images of regular languages.  For deterministic $\MSO$ string transductions,
this follows from ordinary backward translation.  For polyregular maps it is
a nontrivial consequence of their equivalent pebble-transducer and
string-to-string $\MSO$-interpretation characterisations
\cite{Bojanczyk2018,BojanczykKieferLhote2019}.  Thus, in both classes, ``the
output lies in $K$'' pulls back to an $\MSO$ sentence on the input.  $\WRP$
adds a rank-sort layer whose
output order is governed by an unbounded integer key, and ``the atom rank of
one atom is less than that of another'' need not be $\MSO$-definable
on the input.  This is precisely what breaks backward $\MSO$
translation, and it is precisely what gives $\WRP$ the power to
realise $\zeta$.  The closure loss is therefore the sharp boundary that
identifies the additional computational resource, not an oversight in
the model design.

The loss is sharp, not total: by Lemma~\ref{lem:wrp-nonempty-regular} the
preimage of ``the output is nonempty'' stays regular for every $\WRP$
map.  What fails is regularity of the preimage of an \emph{arbitrary}
regular $K$, witnessed by the explicit map $D$ above.  The map $\zeta$
is a second witness: by the probe of Section~\ref{sec:zeta-not-polyregular} the
language $\zeta^{-1}(R)$ is not regular for a suitable regular $R$, although
this route, unlike $D$, leans on the zeta lower bound.

There are two related inversion phenomena.  Under the realisation convention,
$\WRP$-realisability of a Dyck-path bijection is not preserved by inversion:
$\zeta\in\WRP$ while $\zeta^{-1}\notin\WRP$
(Section~\ref{sec:inverse-zeta}).  Ordinary closure under functional inversion
for partial maps also fails, already for the unary squaring map
$s(a^n)=a^{n^2}$.  This map is polyregular and hence belongs to $\WRP$, but its
inverse has domain $\{a^{n^2}:n\ge0\}$, which is not regular, whereas the
domain of every partial polyregular or $\WRP$ map is $\MSO$-definable and hence
regular.  Thus, $s^{-1}$ belongs to neither class.  These inversion failures
are separate expressions of the directional nature of the underlying
transduction mechanisms.
\end{remark}

\subsection{Bounded-prefix-rank collapse}

\begin{proof}[Proof of Theorem~\ref{thm:bounded-rank-collapse}]
Each coordinate source $A_{c,r}$ is a deterministic automaton accumulating a
vector in $\Z^d$ (Definition~\ref{def:rank-source}); the hypothesis says its running total
stays in the finite cube $V=\{-B,\ldots,B\}^d$ on every input in the domain.
Augment the state of $A_{c,r}$ with the current total in $V$, and add an overflow
state for transitions whose new total leaves $V$.  This gives a genuine
deterministic finite automaton on all input words.  On every
$w\in\operatorname{dom}(T)$ the overflow state is never reached, and the
product state just before position $i$ records both the control state
$q^{A_{c,r}}_i$ and the exact value $\rho_{A_{c,r}}^w(i)$.

To test a value at a free position $x$, mark $x$ in the input alphabet.  The
product automaton inspects its stored total when it reaches the marked
position and then remembers the result while reading the suffix.  After
conjoining with the condition defining the domain of $T$, it recognises exactly
the marked words satisfying $\rho_{A_{c,r}}^w(x)=v$.  Hence, for every $v\in V$, that
predicate is $\MSO$-definable in the free variable $x$, by
B\"uchi--Elgot--Trakhtenbrot (Section~\ref{sec:mso}); the predicates specifying
the source state $q^{A_{c,r}}_x$ are definable in the same way.  The domain
qualification is essential here: a source may leave $V$ on words outside
$\operatorname{dom}(T)$, but the presentation is required to reproduce $T$
only on its domain.

A prefix-additive rank function $\kappa_c$
(Definition~\ref{def:prefix-additive-rank}) is a fixed sum of these prefix
ranks, the local data $(q^{A_{c,r}}_{x_r},a_{x_r})$, and a constant.  It
therefore takes only finitely many values on inputs in the domain, and for each
such value $u$ there is an $\MSO$ formula
$\theta_{c,u}(\bar x)$ that, on the domain, expresses
$\kappa_c(\bar x)=u$.  Thus, for copy names $c,c'$, both atom-rank comparison
and atom-rank equality are $\MSO$-definable by finite disjunctions:
\[
\begin{aligned}
 R^{<}_{c,c'}(\bar x,\bar x')
   &:=\bigvee_{u<_{\mathrm{lex}}u'}
      \bigl(\theta_{c,u}(\bar x)\land
             \theta_{c',u'}(\bar x')\bigr),\\
 R^{=}_{c,c'}(\bar x,\bar x')
   &:=\bigvee_u
      \bigl(\theta_{c,u}(\bar x)\land
             \theta_{c',u}(\bar x')\bigr),
\end{aligned}
\]
where the disjunctions range over the finite sets of possible atom ranks for
the two copy names.
The output order $\prec$ of Definition~\ref{def:wrp} is therefore expressed,
on selected atoms over inputs in the domain, by the single $\MSO$ formula
\[
   R^{<}_{c,c'}(\bar x,\bar x')
   \ \lor\
   \bigl(R^{=}_{c,c'}(\bar x,\bar x')\land
          \chi_{c,c'}(\bar x,\bar x')\bigr).
\]
This formula defines exactly the original strict total order: it compares
atom ranks first and applies $\chi$ precisely when they are equal.  Replacing the
rank-sort layer by this $\MSO$ order turns the presentation into a plain
polyregular presentation
(Definition~\ref{def:polyregular}) with the same selected atoms, the same
labels, and the same output order; so $T$ is a polyregular map.
\end{proof}

\begin{proof}[Proof of Corollary~\ref{cor:rank-necessary}]
Suppose that a $\WRP$ map realising $\zeta$ had a presentation in which every
rank source was uniformly bounded on the map's domain.
Theorem~\ref{thm:bounded-rank-collapse} would make that map polyregular,
contradicting Theorem~\ref{thm:zeta-not-polyregular}.  Therefore every such
presentation contains at least one source whose prefix ranks are unbounded on
the domain of the realising map.  In the usual presentation from
Theorem~\ref{thm:zeta-wrp}, this necessary unbounded source is the height
source: its prefix rank is the path height, which on $\D_n$ can reach $n$.
\end{proof}

\section{A self-contained proof of the height-sweep theorem}
\label{app:narayana-sweep-proof}

Theorem~\ref{thm:narayana-sweep} states a known consequence of classical
sweep-map and zeta-map theory~\cite{ThomasWilliams2018,CeballosFangMuhle2020,
Fang2024,SulzgruberThiel2018}.  We include the following direct forest proof to
make the combinatorial statement self-contained in the path conventions of
this paper.  The proof is not needed to establish that $H\in\SWR$, which
already follows from Proposition~\ref{prop:alw-sweep-swr}.

\begin{proof}[Proof of Theorem~\ref{thm:narayana-sweep}]
Fix $n\ge 0$ and $P\in\D_n$.  If $n=0$, then $P=\varepsilon$ and
$H(P)=\varepsilon$, so all the claims are immediate.  Assume henceforth that
$n\ge 1$.

Pass $P$ through the contour bijection $\D_n\cong\mathcal F_n$, where
$\mathcal F_n$ denotes the set of ordered plane forests $F$ on $n$ vertices:
scanning $P$ left to right, each $U$ enters a new
vertex (a child of the current vertex, or a new root at height $0$) and each $D$
leaves the current vertex; a vertex with no children is a \emph{leaf}.  Each
up-step creates one vertex and the matching down-step closes it, so the forest
has exactly $n$ vertices, one per up-step, and no extra root is added.  For
instance, scanning $UUDD$ the first $U$ opens a root, the second $U$ opens a
child of it, and the two $D$'s close the child and then the root, giving a root
with a single leaf child (two vertices); whereas $UD$ opens and at once closes
one root, a lone leaf.  The deeper path $P=UUUDDD$ of
Example~\ref{ex:narayana-sweep} is the chain root, child, grandchild, whose only
leaf is the grandchild.  Leaves correspond exactly to the peaks $UD$ of $P$, so
this forest has $\pk(P)$ leaves and $n-\pk(P)$ internal vertices.  Write $r$ for
the number of roots, $d(v)$ for the depth of a vertex (roots at depth $0$), and
$c(v)$ for its number of children.

After adjoining a super-root whose children are the roots of the forest and
translating between the path conventions, the breadth-first child-count
encoding used below is the plane-tree special case of the map
$\Xi_{\mathrm{bounce}}$ of Ceballos, Fang, and
M\"uhle~\cite[Section~3.3]{CeballosFangMuhle2020}; see also Fang's explicit
breadth-first formulation~\cite[Construction~3.6 and
Proposition~3.7]{Fang2024}.  We derive the normal form in the present contour
and step-word conventions because it also makes the statistic calculation
transparent.

\vspace{0.5em}
\noindent\emph{Normal form.}  We claim
\[
   H(P)\;=\;U^{r}\prod_{v}D\,U^{c(v)},
\]
the product running over all $n$ vertices of the forest in right-to-left
breadth-first order, that is by increasing depth and right to left within each
depth.  Group
the steps by their starting height, the level on which the sweep sorts.  The $U$
entering a vertex $v$ starts at height $d(v)$, and the $D$ leaving $v$ starts at
height $d(v)+1$.  Hence sweep level $0$ is exactly the $U$ steps entering the $r$
roots, giving $U^{r}$, and for $\ell\ge 1$ the sweep level $\ell$ collects, over
all depth-$(\ell-1)$ vertices $v$, the $D$ leaving $v$ together with the $U$
steps entering the children of $v$.  Two facts fix the order within this level.
Within a single $v$, its children
are entered before $v$ is left, so $v$'s leaving $D$ has a larger input position
than the entering $U$ of any child; ordering by decreasing input position
therefore places that $D$ first, then those $U$'s.  Across vertices, the subtrees
at distinct depth-$(\ell-1)$ vertices occupy disjoint position-intervals, so the
same ordering lists those vertices right to left.  Each thus contributes the
block $D\,U^{c(v)}$, and concatenating the levels gives the claimed form.

\vspace{0.5em}
\noindent\emph{Bijection.}  Write $\varphi\colon\D_n\to\mathcal F_n$ for the contour
bijection above, and let $\psi$ send a forest $F\in\mathcal F_n$, with $r$ roots
and breadth-first right-to-left vertex listing $v_1,\dots,v_n$ (the order of the
normal form), to
\[
   \psi(F)=U^{r}\,D\,U^{c_1}\,D\,U^{c_2}\cdots D\,U^{c_n},
   \qquad c_i=c(v_i).
\]
This is the normal form rewritten with the indices $v_1,\dots,v_n$, so
$H=\psi\circ\varphi$.  As $\varphi$ is a bijection, it suffices to prove that
$\psi$ is a bijection $\mathcal F_n\to\D_n$.  We prove three things in turn:
$\psi$ takes its values in $\D_n$, $\psi$ is injective, and
$|\mathcal F_n|=|\D_n|$.  An injection between two finite sets of the same size
is a bijection, so these suffice.

\vspace{0.5em}
\noindent\emph{$\psi$ takes values in $\D_n$.}  Fix $F\in\mathcal F_n$ and set
$w=\psi(F)$.  The step counts are correct: $w$ has $n$ down-steps, one per vertex, and
$r+\sum_v c(v)=r+(n-r)=n$ up-steps, the child-counts summing to the number of
non-root vertices.  Since $w$ loses height only at a down-step, it stays at or
above $0$ if and only if its height is at least $1$ just before each down-step.
Just before the $k$-th down-step $w$ has made $r+\sum_{i<k}c_i$ up-steps and
$k-1$ down-steps, so we must show
\[
   r+\sum_{i<k}c_i\;\ge\;k \qquad (1\le k\le n).
\]
The left-hand side counts the $r$ roots together with all children of
$v_1,\dots,v_{k-1}$, two disjoint kinds of vertex (a root has no parent, and the
children of distinct vertices are distinct).  Each of $v_1,\dots,v_k$ is one of
these: it is a root, or else a vertex whose parent lies earlier in breadth-first
order, hence among $v_1,\dots,v_{k-1}$, which then counts it as a child.  The $k$
distinct vertices $v_1,\dots,v_k$ thus lie inside a set of size
$r+\sum_{i<k}c_i$, giving the inequality.  Hence $w\in\D_n$, and $\psi$ is a
well-defined map $\mathcal F_n\to\D_n$.

\vspace{0.5em}
\noindent\emph{$\psi$ is injective.}  Suppose $w=\psi(F)=\psi(F')$ for forests
$F,F'\in\mathcal F_n$.  The leading $U$-run and the $U$-runs after the successive
down-steps recover, from this one word $w$, the same root count $r$ and the same
list $c_1,\dots,c_n$ for both; that is, writing $v_1,\dots,v_n$ and
$v'_1,\dots,v'_n$ for the breadth-first listings of $F$ and $F'$, we have
$c(v_i)=c(v'_i)=c_i$ for every $i$.  In any forest the breadth-first order places
the roots first, then the children of $v_1$, then those of $v_2$, and so on, so
the children of $v_i$ are exactly the consecutive vertices
\[
   v_{t_i+1},\dots,v_{t_i+c_i},\qquad t_i=r+\sum_{j<i}c_j,
\]
listed right to left in sibling order; the identical description, with the same
$t_i$ and $c_i$, gives the children of $v'_i$.  The index matching
$v_i\leftrightarrow v'_i$ thus sends the $m$-th child of $v_i$ in this listing to
the $m$-th child of $v'_i$, and since both blocks run right to left, this matches
each child of $v_i$ with the child of $v'_i$ in the same relative position.
Thus, the matching
preserves both the parent relation and the sibling order.  It is therefore an
isomorphism of ordered forests, and $F=F'$.

\vspace{0.5em}
\noindent\emph{Equal cardinality.}  Finally $|\mathcal F_n|=|\D_n|$, since $\varphi$ is a
bijection.  An injection between finite sets of the same size is onto, so
$\psi\colon\mathcal F_n\to\D_n$ is a bijection, and therefore so is
$H=\psi\circ\varphi$.

\vspace{0.5em}
\noindent\emph{Swap.}  The maximal $U$ runs of $H(P)$ are the initial $U^{r}$ and the
blocks $U^{c(v)}$, mutually separated by the $D$'s, and a block is nonempty
exactly when $v$ is internal.  So $H(P)$ has $1+(n-\pk(P))$ nonempty $U$ runs,
and since a run of length $m$ contributes $m-1$ double rises,
\[
   \dr(H(P))=n-\bigl(1+(n-\pk(P))\bigr)=\pk(P)-1=\val(P).
\]
Applying the elementary identities $\dr=n-\pk$ and $\val=\pk-1$ to the word
$H(P)$ then gives $\pk(H(P))=n-\dr(H(P))=n-\val(P)=n-\pk(P)+1$, whence
$\val(H(P))=\pk(H(P))-1=n-\pk(P)=\dr(P)$.  These are the two required
identities; since $H$ is a bijection of $\D_n$ that exchanges $\val$ and $\dr$,
it swaps the two exponents in $\sum_P q^{\val(P)}t^{\dr(P)}$, forcing
$\Nar_n(q,t)=\Nar_n(t,q)$.
\end{proof}

\section{Proofs of the deficit-zero lemmas}
\label{app:deficit-zero-proofs}

We give the proofs of Lemmas~\ref{lem:dinv-coarea}
and~\ref{lem:deficit-zero-targets}, whose statements are used in
Section~\ref{sec:main}.

\begin{proof}[Proof of Lemma~\ref{lem:dinv-coarea}]
Let $a=(a_1,\ldots,a_N)$ be the area sequence of $Q$.  Fix $j$.  The
number of dinv pairs ending at $j$ is $|\{i<j:a_i-a_j\in\{0,1\}\}|$.
Since an area sequence starts at $0$ and grows by at most $1$ per step,
the prefix $(a_1,\ldots,a_{j-1})$ contains at least one occurrence of each
level $0,1,\ldots,a_j-1$.  These $a_j$ entries are not counted, so the
contribution ending at $j$ is at most $(j-1)-a_j$.  Summing,
\[
   \dinv(Q)\le\sum_{j=1}^N\bigl((j-1)-a_j\bigr)
            =\binom N2-\area(Q)
            =\coarea(Q). \qedhere
\]
\end{proof}

\begin{proof}[Proof of Lemma~\ref{lem:deficit-zero-targets}]
\emph{Existence.}
Consider the word
\[
   P=U^a(DU)^{\,b-d}D(DU)^dD^{\,a-1}.
\]
It has $a+b=N$ up-steps and the same number of down-steps.  After the initial
$U^a$, the blocks $DU$ in the first group alternate between heights $a$ and
$a-1$.  The additional $D$ lowers the height to $a-1$, and the blocks $DU$ in
the second group alternate between heights $a-1$ and $a-2$.  Since
$c\le\binom N2$, the definition of $b$ gives $b\le N-1$ and hence
$a=N-b\ge1$, so the first group never goes below height zero.  If $d=0$, the
second group is absent.  If $d>0$, then $b\ne N-1$, since
otherwise $c=\binom N2+d>\binom N2$.  Thus, $b\le N-2$ and $a\ge2$, so the
second group also remains at nonnegative height.  Finally, $D^{a-1}$ descends
from height $a-1$ to zero.  Hence $P\in\D_N$.  The starting heights of its
up-steps are
\[
   (0,1,\ldots,a-1,
     \underbrace{a-1,\ldots,a-1}_{b-d\text{ times}},
     \underbrace{a-2,\ldots,a-2}_{d\text{ times}}),
\]
which is the area sequence in the statement.  Its coarea is
\[
   \binom N2-\Bigl(\binom a2+(b-d)(a-1)+d(a-2)\Bigr)
   =\binom{b+1}{2}+d=c.
\]
Its dinv is the sum of equal and adjacent pairs among the last two plateau
levels and the last two levels of the initial staircase:
\[
   \binom{b-d}{2}+\binom d2+(b-d)d+(b-d)+2d
   =\binom{b+1}{2}+d=c.
\]
This proves existence.

\vspace{0.5em}
\noindent\emph{Uniqueness.}
We first make explicit the equality condition behind the bound, in the notation
of the proof of Lemma~\ref{lem:dinv-coarea}.  Fix a position $j$.  Call an
earlier entry $a_i$ ($i<j$) \emph{counted} if it forms a dinv pair with $j$, that
is $a_i\in\{a_j,a_j+1\}$, and call the entries sitting at the levels
$0,1,\ldots,a_j-1$ the \emph{low witnesses}; each such level does occur before
$j$, since the area sequence climbs from $0$ to $a_j$ by unit steps.  A low
witness has $a_i<a_j$, so it is never counted. This is exactly why the
contribution at $j$ is at most $(j-1)-a_j$.  The contribution \emph{equals}
$(j-1)-a_j$ precisely when those low witnesses are the \emph{only} uncounted
earlier entries, i.e.\ when
\begin{enumerate}[label=(\roman*)]
\item each level $0,1,\ldots,a_j-1$ occurs exactly once before $j$ (a repeat
would be a second uncounted entry), and
\item every other earlier entry is counted, i.e.\ equals $a_j$ or $a_j+1$ (an
entry at level $\ge a_j+2$ would be uncounted and not a witness).
\end{enumerate}
Global equality in Lemma~\ref{lem:dinv-coarea} forces (i) and (ii) at every $j$.

Now let $(a_1,\ldots,a_N)$ be an equality case, and let $r$ be the length of its
initial staircase: $a_i=i-1$ for $1\le i\le r$, with either $r=N$ or
$a_{r+1}<r$.  If $r=N$ the sequence is the full staircase $(0,1,\ldots,N-1)$ and
$c=0$, so $b=d=0$ and this is the required unique area sequence.  We may
therefore assume that $r<N$.  Put $A=r-1$, so the earlier entries
$a_1,\ldots,a_r$ are
$0,1,\ldots,A$, one at each level.  Apply the condition at $j=r+1$.  Since
$a_{r+1}\le A$, the top earlier entry $a_r=A$ is \emph{not} one of the low
witnesses $0,\ldots,a_{r+1}-1$, so by~(ii) it must be counted, that is
$A\in\{a_{r+1},a_{r+1}+1\}$, which says $a_{r+1}\ge A-1$.  (Were
$a_{r+1}\le A-2$, the level-$A$ entry would be neither a low witness nor counted,
leaving an extra uncounted entry and breaking equality at $r+1$.)  With
$a_{r+1}\le A$, this leaves $a_{r+1}\in\{A,A-1\}$.

The same two conditions, read at each later position in turn, keep every
remaining entry in $\{A,A-1\}$.  A value $\le A-2$ is excluded exactly as above,
the level-$A$ staircase entry being left uncounted and not a witness; and a value
$\ge A+1$ is excluded because, once some tail entry has already repeated level
$A$ or $A-1$, condition~(i) can no longer hold at a position whose low levels
reach all the way up to $A$.  Finally, an $A-1$ entry cannot occur before a later
$A$ entry: at that later position $a_j=A$, the earlier $A-1$ would be a second
occurrence of a level (namely $A-1$) already supplied by the staircase,
violating~(i).  Thus, all the $A$'s precede all the $A-1$'s, the tail is
of the required form $(A,\ldots,A,A-1,\ldots,A-1)$.  Let $B=N-r$ be its
length, and let $d$ be the number of trailing entries equal to $A-1$.  The
same coarea calculation as in the existence part gives
\[
  c=\binom{B+1}{2}+d,\qquad 0\le d\le B.
\]
The intervals
\[
  \left[\binom{B+1}{2},\binom{B+2}{2}-1\right]\qquad(B\ge0)
\]
are pairwise disjoint and cover the nonnegative integers, so $c$ uniquely
determines $B=b$ and $d$.  Consequently, $r=N-b=a$ and $A=a-1$, and the form
and multiplicities just proved give exactly the area sequence in the
statement.  The equality case is therefore unique.
\end{proof}
\end{document}